\documentclass[reqno]{amsart}
\usepackage{amssymb, amsthm, amsfonts, amsmath, amscd}
\usepackage[utf8]{inputenc}
\usepackage[margin=1in]{geometry}  
\usepackage{graphicx}
\usepackage{tikz}
\usepackage{microtype}
\usepackage{mathrsfs}
\usepackage{hyperref}
\usepackage[toc,page]{appendix}
\usepackage{comment}
\usetikzlibrary{decorations.markings}

\newmuskip\pFqmuskip

\newcommand*\pFq[6][8]{
  \begingroup
  \pFqmuskip=#1mu\relax
  \mathcode`\,=\string"8000
  \begingroup\lccode`\~=`\,
  \lowercase{\endgroup\let~}\pFqcomma
  {}_{#2}F_{#3}{\left[\genfrac..{0pt}{}{#4}{#5};#6\right]}
  \endgroup
}
\newcommand{\pFqcomma}{\mskip\pFqmuskip}
\newcommand\myeq{\mathrel{\stackrel{\makebox[0pt]{\mbox{\normalfont\tiny def}}}{=}}}

\theoremstyle{definition}
\newtheorem{thm}{Theorem}[section]
\newtheorem{lem}[thm]{Lemma}
\newtheorem{df}[thm]{Definition}
\newtheorem{prop}[thm]{Proposition}
\newtheorem{cor}[thm]{Corollary}
\newtheorem{rem}[thm]{Remark}

\newtheorem{assum}[thm]{Assumption}

\newcommand{\C}{\mathbb{C}}

\newcommand{\R}{\mathbb{R}}
\newcommand{\Z}{\mathbb{Z}}
\newcommand{\Zp}{\mathbb{Z}_+}
\newcommand{\N}{\mathbb{N}}
\newcommand{\X}{\mathfrak{X}}
\newcommand{\DD}{\mathcal{D}}
\newcommand{\MM}{\mathcal{M}}

\newcommand{\p}{\mathfrak{p}}
\newcommand{\boldx}{\mathbf{x}}
\newcommand{\boldy}{\mathbf{y}}
\newcommand{\hatx}{\widehat{x}}

\newcommand{\hatl}{\widehat{l}}
\newcommand{\hatm}{\widehat{m}}
\newcommand{\hatn}{\widehat{n}}
\newcommand{\hatk}{\widehat{k}}

\newcommand{\TT}{\mathbb{T}}
\newcommand{\GT}{\mathbb{GT}}
\newcommand{\GTp}{\mathbb{GT}^+}

\newcommand{\free}{\textrm{free}}

\numberwithin{equation}{section}

\begin{document}

\title[Markov processes on the duals\dots]{Markov processes on the Duals to Infinite-Dimensional Classical Lie Groups}

\author{Cesar Cuenca}
\address{Department of Mathematics, MIT, Cambridge, MA, USA}
\email{cuenca@mit.edu}

\date{}

\begin{abstract}
We construct a four parameter $z, z', a, b$ family of Markov dynamics that preserve the $z$-measures on the boundary of the branching graph for classical Lie groups of type $B, C, D$. Our guiding principle is the ``method of intertwiners'' used previously in \cite{BO2} to construct Markov processes that preserve the $zw$-measures.
\end{abstract}

\keywords{BC type z-measures, infinite dimensional spaces, intertwining processes, Doob h-transform}

\maketitle

\setcounter{tocdepth}{1}
\tableofcontents

\section{Introduction}

In recent years, plenty of works have dealt with probabilistic models which allow a thorough analysis due to their ``integrability''. The word ``integrability'' in this context means that one can express correlation functions or observables of the model by explicit algebraic formulas, which can then be analyzed in various limit regimes. Perhaps most well-known in the integrable probability world are the models arising from Schur and Macdonald processes, see \cite{BC, BG, Ok1, OR}.

The integrability of the probability measures we mentioned above have their roots in classical representation theory. In fact, the first (special cases of) Schur measures in the literature appeared in connection with the problem of harmonic analysis of the infinite symmetric group $S(\infty)$, as posed in \cite{KOV}, and are known as $z$-measures. Even more sophisticated objects, the $zw$-measures arise in connection with the problem of harmonic analysis of the infinite-dimensional unitary group $U(\infty)$. Many aspects of both the $z$-measures and $zw$-measures, as well as the probabilistic models associated to them, have been extensively studied in the last two decades, e.g., see \cite{BK, BO1, BO3, BO4, Ok2}.

All of the theory above comes actually from ``type A'' representation theory, i.e., the main players are the characters of the unitary groups (Schur polynomials) and the probability measures are defined on signatures (which parametrize irreducible representations of unitary groups) or arrays of them. Models whose underlying integrability is based on ``type BC'' representation theory have been less popular in the probability and integrable systems literature, though there are some interesting works showing that models of both types can be studied by similar techniques, but display interesting and different features, see e.g. \cite{BK, C, WZ}. Our hope is that the theory of integrable models with BC type integrability yield a wealth of applications to various branches of probability theory and mathematical physics, just like their counterparts of type A. In the present paper, we link stochastic dynamics to BC representation theory by constructing a family of Markov processes which preserve probability measures of BC representation theoretic origin.

Concretely, in this work we deal with probability measures that are the BC analogues of the $zw$-measures: they are rooted in the problem of harmonic analysis of the infinite-dimensional groups $O(\infty)$, $Sp(\infty)$ and the infinite-dimensional symmetric space $U(2\infty)/U(\infty)\times U(\infty)$. These measures will also be called $z$-measures, because they depend only on one pair of parameters $(z, z')$ as in the $S(\infty)$ case, instead of two pairs $(z, z'), (w, w')$, as in the $U(\infty)$ case. The $z$-measures corresponding to the three infinite-dimensional groups above can be defined over the same infinite-dimensional space $\Omega_{\infty}$. Interestingly, on such space one can define probability measures that depend on an extra set of real parameters $a, b > -1$ that originate from the weight function of the classical Jacobi polynomials. Therefore we can define $z$-measures depending on four parameters $(z, z', a, b)$; these probability measures have a representation theoretic meaning only for the pairs $(a, b) = (0, 0), (\frac{1}{2}, \frac{1}{2}), (\frac{1}{2}, -\frac{1}{2}), (-\frac{1}{2}, -\frac{1}{2})$, but their integrability remains for general $a, b$.

The main result of this paper is the construction of Markov dynamics on $\Omega_{\infty}$ that have the $z$-measures as their unique invariant measures. Importantly in the theory of harmonic analysis of big groups, the $z$-measures can be viewed as probability measures on point configurations in $\R_{>0} \setminus \{1\}$ with infinitely many particles (see \cite{BO1} for a deep study of the $zw$-measures from this point of view). The author has proved that the resulting point processes for $z$-measures are determinantal and their kernels have closed forms in terms of hypergeometric functions, \cite{C}. Thus the dynamics we construct on point configurations that preserve the determinantal structure is one of many that have been studied in the literature, mostly within the context of random matrix theory, see the introduction in \cite{BO2} and further references therein.

The technique employed to construct our Markov dynamics is based on the ``method of intertwiners'', which was first applied in \cite{BO2} for the hypergeometric process coming from $zw$-measures. To apply this method, it is essential to realize $\Omega_{\infty}$ as the boundary of a (graded) branching graph with formal edge multiplicities that depend on the parameters $(a, b)$. The goal then becomes to construct ``coherent'' Markov semigroups on the levels of the graph that preserve certain pushforwards of the $z$-measures. We construct the required semigroups on a general level N by using a Doob $h$-transformation of N birth-and-death processes. It is worth noting that our construction has many complications that are not present for the case of $zw$-measures. One is that our branching graph has edge multiplicities, unlike its type A analogue, and moreover there is no closed form for the edge multiplicities. Another one is that the ``jumping rates'' for our birth-and-death processes mentioned above are related to a family of orthogonal polynomials on a quadratic lattice, which are expressible in terms of the $_4F_3$ hypergeometric function. In the case of the $zw$-measures, the construction depends on similar birth-and-death processes, but the rates in that case are related to a family of orthogonal polynomials on a linear lattice and they are lower in the hierarchy of orthogonal polynomials.

Let us describe briefly the ingredients of two interesting parts of our plan. First, the argument for the invariance of certain pushforwards of the $z$-measures with respect to the semigroups we construct at finite levels of the BC branching graph makes a clever use of a finite family of orthogonal polynomials of a discrete variable on a quadratic lattice. These polynomials are eigenfunctions of the infinitesimal generator of the process at level one of the branching graph; they can be expressed in terms of the classical Wilson polynomials of a continuous variable \cite{W}, and they also appeared more recently in Neretin's work \cite{N}, reason why we will call them Wilson-Neretin polynomials. Second, the ``master relation'' which lies at the heart of the method of intertwiners, see $(\ref{intertwining})$ below, is proved by a brute-force computation involving a new binomial formula for ``shifted'' symmetric (Jacobi) polynomials, see Appendix $\ref{appendixA}$.

In the rest of this introduction, we give a more detailed account of our main result only for three of the special pairs $(a, b)$ mentioned above in order to avoid dealing with more complicated expressions. 

\subsection{BC branching graph and its boundary}\label{sec:boundaryintro}

Let $(a, b)$ be one of the pairs $(\frac{1}{2}, \frac{1}{2})$, $(\frac{1}{2} -\frac{1}{2})$, $(-\frac{1}{2}, -\frac{1}{2})$. For each positive integer $N$, let $\GTp_N$ be the set of $N$-tuples $\lambda$ of weakly decreasing nonnegative integers
\begin{equation*}
\lambda = (\lambda_1 \geq \lambda_2 \geq \ldots \geq \lambda_N\geq 0).
\end{equation*}
Elements of $\GTp_N$ are called positive $N$-signatures; they parametrize the irreducible characters of $Sp(2N)$, $SO(2N+1)$, as well as certain reducible characters of $O(2N)$ \footnote{Irreducible characters of $SO(2N)$ are parametrized by $N$-tuples of integers $(\lambda_1, \ldots, \lambda_N)$ satisfying $\lambda_1\geq\ldots\geq\lambda_{N-1}\geq |\lambda_N|$. The reducible character of $O(2N)$ that is associated to $\lambda_1\geq\ldots\geq\lambda_N\geq 0$ is the following sum of two irreducible characters: $\chi^{o(2N)}_{\lambda_1, \ldots, \lambda_N} \myeq \chi^{so(2N)}_{\lambda_1, \ldots, \lambda_N} + \chi^{so(2N)}_{\lambda_1, \ldots, -\lambda_N}$}. Given two positive signatures $\mu\in\GTp_N, \lambda\in\GTp_{N+1}$, we write
\begin{equation*}
\lambda \succ_{BC} \mu
\end{equation*}
if there exists $\nu\in\GTp_N$ such that the following inequalities hold
\begin{eqnarray*}
&&\lambda_1 \geq \nu_1 \geq \lambda_2 \geq \ldots \geq \nu_N \geq \lambda_{N+1}\\
&&\nu_1 \geq \mu_1 \geq \nu_2 \geq \ldots \geq \nu_N \geq \mu_N.
\end{eqnarray*}
It is known that the relation $\succ_{BC}$ has a representation theoretic meaning. In fact, let $G(N)$ be any one of the rank-$N$ Lie groups $SO(2N+1), Sp(2N), O(2N)$, for any $N\in\N$. For $\lambda\in\GTp_{N+1}$, denote by $\chi_{\lambda} = \chi_{\lambda}^{G(N+1)}$ the character of $G(N+1)$ parametrized by $\lambda$, as described above. By considering the natural embeddings $G(N)\hookrightarrow G(N+1)$, the restricted character $\chi_{\lambda}|_{G(N)}$ decomposes into a sum of characters $\chi_{\mu}$, $\mu\in\GTp_{N}$, where $\chi_{\mu}$ appears in the decomposition with certain multiplicity $m(\lambda, \mu)$. Then it is true that $m(\lambda, \mu) > 0$ if and only if $\lambda\succ_{BC}\mu$.

Next we describe a graph with vertex set $\GTp = \sqcup_{N\geq 1}{\GTp_N}$, whose vertices are partitioned into levels $1, 2, 3, \ldots$, and where the $N$th level consists of the vertices $\GTp_N$. The edges of the graph have multiplicities and they connect only vertices in adjacent levels: there is an edge between $\mu\in\GTp_N$ and $\lambda\in\GTp_{N+1}$ iff $\lambda\succ_{BC} \mu$, and moreover that edge has multiplicity $m(\lambda, \mu)$. We call the $\N$-graded graph just described the \textit{branching graph of classical Lie groups of type $B, C, D$}, or simply the \textit{BC branching graph}.

For each $\mu\in\GTp_N$ and $(a, b) = (\frac{1}{2}, \frac{1}{2})$ (resp. $(\frac{1}{2}, -\frac{1}{2}), (-\frac{1}{2}, -\frac{1}{2})$), let $\chi^{a, b}_{\mu}(1)$ be the character of $SO(2N+1)$ (resp. $Sp(2N), O(2N)$) parametrized by $\mu$ and evaluated at the identity matrix. Define $\chi^{a, b}_{\lambda}(1)$ similarly, for any $\lambda\in\GTp_{N+1}$. We consider the expressions
\begin{equation*}
\Lambda^{N+1}_N(\lambda, \mu) =
    \begin{cases}
        m(\lambda, \mu)\chi^{a, b}_{\mu}(1)/\chi^{a, b}_{\lambda}(1) & \textrm{if } \lambda\succ_{BC}\mu\\
        0 & \textrm{otherwise},
    \end{cases}
\end{equation*}
where $m(\lambda, \mu)$ is the multiplicity of $\chi^{a, b}_{\mu}$ in the decomposition of $\chi^{a, b}_{\lambda}|_{G(N)}$, and $G(N) = SO(2N+1), Sp(2N), O(2N)$ when $(a, b) = (\frac{1}{2}, \frac{1}{2}), (\frac{1}{2}, -\frac{1}{2}), (-\frac{1}{2}, -\frac{1}{2})$, respectively. For simplicity of notation, we omit $a, b$ from the notation of $\Lambda^{N+1}_N$.

From the definition above, it is clear that $\Lambda^{N+1}_N(\lambda, \mu)\geq 0$ and $\sum_{\mu}{\Lambda^{N+1}_N(\lambda, \mu)} = 1$. Therefore $(\Lambda^{N+1}_N(\lambda, \mu))_{\lambda, \mu}$ is a stochastic matrix of format $\GTp_{N+1}\times\GTp_N$ and provides a Markov kernel $\GTp_{N}\dashleftarrow\GTp_{N+1}$. By applying the same reasoning for all $N$, we obtain a chain of Markov kernels
\begin{equation}\label{chain}
\GTp_1 \dashleftarrow \GTp_2 \dashleftarrow \cdots\dashleftarrow \GTp_N \dashleftarrow \GTp_{N+1}\dashleftarrow\cdots
\end{equation}
We say that a sequence $\{M_N\}_{N=1, 2, \ldots}$ of probability measures on the levels $\{\GTp_N\}_{N = 1, 2, \ldots}$ is a coherent system if the following consistency condition holds:
\begin{equation*}
M_N(\mu) = \sum_{\lambda\in\GTp_{N+1}}{M_{N+1}(\lambda)\Lambda^{N+1}_N(\lambda, \mu)}, \ \forall N\geq 1, \ \mu\in\GTp_N.
\end{equation*}

If we let $\mathcal{M}_p(\GTp_N)$ be the metric space of probability measures on $\GTp_N$, then $(\ref{chain})$ induces the following chain of measurable maps
\begin{equation*}
\MM_p(\GTp_1) \leftarrow \MM_p(\GTp_2) \leftarrow \cdots\leftarrow \MM_p(\GTp_N) \leftarrow \MM_{p+1}(\GTp_{N+1})\leftarrow\cdots
\end{equation*}
The inverse limit $\lim_{\leftarrow}{\MM_p(\GTp_N)}$ is evidently a convex set, whose elements are coherent systems $\{M_N\}_{N = 1, 2, \ldots}$ of probability measures. We define the boundary of the BC branching graph $\Omega_{\infty}$ as the set of extreme points of the convex set $\lim_{\leftarrow}{\MM_p}(\GTp_N)$. It is a known result, see e.g. \cite[Thm. 9.2]{Ol2}, that there is a natural map
\begin{equation}\label{tobeiso}
\mathcal{M}_p(\Omega_{\infty}) \longrightarrow \lim_{\leftarrow}{\mathcal{M}_p(\GTp_N)}
\end{equation}
which is a bijection of sets.

The boundary $\Omega_{\infty}$ of the BC branching graph, as well as the maps $\Lambda^{\infty}_N: \MM_p(\Omega_{\infty})\longrightarrow\MM_p(\GTp_N)$ coming from $(\ref{tobeiso})$, can be described fairly explicitly. Consider the space $\R^{2\infty+1} = \R^{\infty}\times\R^{\infty}\times\R$ with the product topology, and its closed subspace $\Omega_{\infty}\subset\R^{2\infty+1}$ of elements $(\alpha, \beta, \delta)$ such that
\begin{eqnarray*}
\alpha &=& (\alpha_1 \geq \alpha_2 \geq \ldots \geq 0)\\
\beta &=& (1\geq \beta_1 \geq \beta_2 \geq \ldots \geq 0)\\
\delta &\geq& \sum_{i=1}^{\infty}{(\alpha_i + \beta_i)}.
\end{eqnarray*}

The points in $\Omega_{\infty}$ can be thought of as limits of partitions whose lengths go to infinity. In fact, consider a sequence of partitions $\{\lambda(M)\}_{M = 1, 2, \ldots}$ such that the length of each $\lambda(M)$ is $\leq M$ and assume that the limits below exist and are finite
\begin{eqnarray*}
\lim_{M\rightarrow\infty}{\frac{\lambda_i(M)}{M}} &=& \alpha_i, \hspace{.2in} i = 1, 2, \ldots\\
\lim_{M\rightarrow\infty}{\frac{\lambda_i'(M)}{M}} &=& \beta_i, \hspace{.2in} i = 1, 2, \ldots\\
\lim_{M\rightarrow\infty}{\frac{|\lambda(M)|}{M}} &=& \delta.
\end{eqnarray*}
We have denoted by $\lambda_i(M)$ the $i$-th part of the partition $\lambda(M)$, by $\lambda_i'(M)$ the $i$-th part of the conjugate partition $\lambda'(M)$ and by $|\lambda(M)|$ the size of the partition. Then clearly the resulting ``limit point'' $\omega = (\alpha, \beta, \delta)$ belongs to $\Omega_{\infty}$. It is a deep theorem of Okounkov and Olshanski, see \cite{OO1}, that the limit
\begin{equation*}
    \Lambda^{\infty}_N(\omega, \mu) = \lim_{M\rightarrow\infty}{(\Lambda^M_{M-1}\Lambda^{M-1}_{M-2}\cdots\Lambda^{N+1}_N)(\lambda(M), \mu)}
\end{equation*}
exists and depends only on $\omega$, and not on the sequence $\{\lambda(M)\}_M$. In the same paper, the authors prove that $\Omega_{\infty}$, as a set, is in bijection with the boundary of the BC branching graph.
The resulting Markov kernels $\Lambda^{\infty}_N : \Omega_{\infty}\dashrightarrow \GTp_N$ induce the coherent maps $\Lambda^{\infty}_N: \MM_p(\Omega_{\infty}) \rightarrow\MM_p(\GTp_N)$ from bijection $(\ref{tobeiso})$.

From the description above, the boundary $\Omega_{\infty}$ can be seen as a limit of the discrete spaces $\GTp_N$. On the other hand, $\GTp_N$ parametrizes the irreducible characters of the compact Lie groups $SO(2N+1), Sp(2N)$ and certain reducible characters of $O(2N)$. Thus in a sense, though not in any formal way, the boundary $\Omega_{\infty}$ is ``dual'' to the inductive limits $S(2\infty + 1) = \lim_{\rightarrow}{SO(2N+1)}$, $Sp(2\infty) = \lim_{\rightarrow}{Sp(2N)}$ and $O(2\infty) = \lim_{\rightarrow}{O(2N)}$, which are infinite-dimensional analogues of classical Lie groups,  hence the title of this paper.

\subsection{$z$-measures}

The $z$-measures are certain probability measures on the boundary $\Omega_{\infty}$ of the BC branching graph that are parametrized by pairs $z, z'$ of complex numbers satisfying certain constraints, e.g., if $z' = \overline{z}\notin\R$ and $\Re(z) > -(1+b)/2$, then the pair $(z, z')$ satisfies those constraints. For simplicity, assume $z' = \overline{z}\notin\R$ and $\Re(z) > -(1+b)/2$ for the rest of the introduction.
The $z$-measures arise naturally in the problem of harmonic analysis for big groups; see Appendix $\ref{sec:zmeasuresharmonic}$ for further details and motivation. See also  \cite{Ol1} for an in-depth discussion of the problem of harmonic analysis on $U(\infty)$, which corresponds to the type $A$ version of the story.

Due to the bijection $(\ref{tobeiso})$, giving a probability measure on $\Omega_{\infty}$ is equivalent to giving a sequence of coherent probability measures on the levels $\{\GTp_N\}_{N = 1, 2, \ldots}$. Under this equivalence, the $z$-measure associated to $(z, z', a, b)$ can be defined by the sequence of probability measures $\{M_{z, z', a, b|N}\}_{N \geq 1}$ on $\{\GTp_N\}_{N \geq 1}$ given by
\begin{eqnarray}\label{zdef1}
M_{z, z', a, b|N}(\lambda) &=& const_N \prod_{1\leq i < j \leq N}{(\hatl_i - \hatl_j)^2}\cdot\prod_{i=1}^N{W_{N}(l_i)},
\end{eqnarray}
where $l_i = \lambda_i + N - i$, $\hatl_i = \left(l_i + \frac{a+b+1}{2}\right)^2$, $const_N$ is a normalization constant and
\begin{eqnarray}\label{Wdef1}
W_{N}(x) &=& \left(x + \frac{a+b+1}{2}\right)\frac{\Gamma(x + a + b + 1)\Gamma(x + a + 1 )}{\Gamma(x + b + 1)\Gamma(x + 1)\Gamma(z-x+N)\Gamma(z'-x+N)}\nonumber\\
&&\times\frac{1}{\Gamma(z + x + N + a + b + 1)\Gamma(z' + x + N + a + b + 1)}.
\end{eqnarray}

We call $M_{z, z', a, b|N}$ the $z$-measure at level $N$ associated to $(z, z', a, b)$. The probability measures $M_{z, z', a, b|\infty}$ on $\Omega_{\infty}$ corresponding to the coherent system $\{M_{z, z', a, b|N}\}_N$ via the bijection $(\ref{tobeiso})$ are known as the spectral $z$-measures, or simply as $z$-measures.

\subsection{Markov dynamics on positive signatures}

Under our assumption on the parameters $z, z', a, b$, in particular $z' = \overline{z}\notin\R$, the rates
\begin{eqnarray*}
r_{x\rightarrow x+1} &=& \frac{(x + 2\epsilon)(x + a + 1)(x - z + N - 1)(x - z' + N - 1)}{(2x + a + b + 1)(2x + a + b+ 2)}, \hspace{.1in}x\geq 0\\
r_{x\rightarrow x-1} &=& \frac{x(x + b)(x + z + N + a + b)(x + z' + N + a + b)}{(2x + a + b + 1)(2x + a + b)}, \hspace{.1in}x\geq 1
\end{eqnarray*}
are strictly positive real numbers and define a continuous-time birth-and-death process on $\GTp_1 = \Zp$ which preserves the $z$-measures at level $1$, i.e., the probability measures given in $(\ref{zdef1})$ for $N = 1$. One can prove this fact directly, but also by showing that the orthogonal polynomials of a discrete variable, with respect to the weight $(\ref{Wdef1})$, have $r_{x\rightarrow x+1}, r_{x\rightarrow x-1}$ as coefficients in the second-order difference equation that defines them, see Section $\ref{zmeasuressubsection}$ for details. The orthogonal polynomials that we need are known as the Wilson-Neretin polynomials.

To construct Markov dynamics on $\GTp_N$ that preserve the $z$-measures at level $N$, consider a Doob $h$-transformation of $N$ independent birth-and-death processes with rates $r_{x\rightarrow x\pm 1}$. Explicitly, we consider the matrix of transition rates on $\GTp_N$ whose entries are
\begin{eqnarray*}
r^{(N)}_{\lambda\rightarrow\nu} &=& \frac{\prod_{i < j}{((n_i + \epsilon)^2 - (n_j + \epsilon)^2)}}{\prod_{i < j}{((l_i + \epsilon)^2 - (l_j + \epsilon)^2)}}\\
&&\times\left(r_{l_1\rightarrow n_1}\mathbf{1}_{\{l_i = n_i, i\neq 1\}} +\ldots + r_{l_N\rightarrow n_N}\mathbf{1}_{\{l_i = n_i, i\neq N\}}\right) - c_N\mathbf{1}_{\{\lambda = \nu\}},\nonumber
\end{eqnarray*}
where $l_i = \lambda_i + N - i$, $n_i = \nu_i + N - i$, $\epsilon = (a+b+1)/2$, and $c_N$ is certain normalizing constant (chosen to make the operator corresponding to the matrix $r^{(N)}$ vanish at the constant functions).

General techniques on continuous-time homogeneous Markov chains allow us to show that the matrices of transition rates $r^{(N)}$ uniquely define Markov semigroups $(P_N(t))_{t\geq 0}$ on $\GTp_N$ and moreover that they are Feller semigroups. The important point for us is that the Feller semigroups $(P_N(t))_{t\geq 0}$ preserve the $z$-measures at level $N$, defined above. The proof of this fact uses again the Wilson-Neretin polynomials, especially the second order difference equation these orthogonal polynomials satisfy.

Given that $(P_N(t))_{t\geq 0}$ are Feller semigroups on the levels of the BC branching graph, one could imagine they can be ``pasted together'' to obtain a Feller semigroup on the boundary $\Omega_{\infty}$. Such thought indeed comes to fruition; the main technical statement that we need to prove is the following relation for all $N$ large enough:
\begin{equation}\label{intertwining}
P_{N+1}(t)\Lambda^{N+1}_N = \Lambda^{N+1}_{N}P_N(t), \hspace{.2in} t\geq 0.
\end{equation}
It is not evident apriori that $(\ref{intertwining})$ holds. The proof is computational and based on identities for (shifted) symmetric polynomials. Once the identity above is known, the general method of intertwiners, see Theorem $\ref{methodintertwiners}$ below, can be applied to conclude the existence of a Feller semigroup $(P_{\infty}(t))_{t\geq 0}$ in $\Omega_{\infty}$ that has the $z$-measures as unique invariant probability measures.
We are ready to state the main result of this paper. Notice that this introduction assumed $(a, b)$ was one of three special pairs, but the following theorem holds for general real parameters $a\geq b\geq -1/2$.

\subsection{The main theorem}

For pairs $(z, z')\in\C^2$ as in Definition $\ref{Hdef}$, there exist Feller semigroups $(P_{\infty}(t))_{t\geq 0}$ on the infinite-dimensional space $\Omega_{\infty}$ that have the $z$-measures as their unique invariant measures. By general theory on Feller semigroups, it follows that, given any probability measure $\nu$ on $\Omega_{\infty}$, there is a Markov process on $\Omega_{\infty}$ with cadlag sample paths, initial distribution $\nu$ and having $(P_{\infty}(t))_{t\geq 0}$ as its transition function.

\subsection{Conventions}\label{conventions}

\begin{enumerate}
	\item In this paper, we often use two real parameters $a, b$ and from them we define
\begin{equation*}
\epsilon = \frac{a+b+1}{2}.
\end{equation*}
Our main result requires
\begin{equation*}
a \geq b \geq -1/2,
\end{equation*}
which implies $\epsilon \geq 0$. However, in parts of the paper, we could assume the less restrictive $a, b > -1$, which is the only requirement necessary for the existence of the Jacobi polynomials $P_{\lambda}(\cdot | a, b)$. The assumption $a \geq b \geq -1/2$ is required in order to make use of the results from \cite{OO1, OlOs}. We keep the stronger assumption throughout to make estimates a little easier.

	\item We denote $\Zp$ and $\N$ the set of nonnegative and positive integers, respectively. We also write
\begin{equation*}
\Zp^{\epsilon} \myeq \{(n + \epsilon)^2 : n\in\Zp\}
\end{equation*}
for the quadratic half-lattice. For any $x\in\Zp$, we often use the notation
\begin{equation*}
\hatx \myeq (x+\epsilon)^2
\end{equation*}
for the corresponding element of $\Zp^{\epsilon}$.

	\item In Section $\ref{positiveGT}$, we define sets $\GTp_N, \Omega_N, \Omega_N^{\epsilon}$, that depend on a natural number $N$, e.g. for $N = 1$, $\GTp_1 = \Omega_1 = \Zp, \Omega_1^{\epsilon} = \Zp^{\epsilon}$. There are natural bijections among these sets for general $N$:
\begin{eqnarray*}
\GTp_N \longleftrightarrow &\Omega_N& \longleftrightarrow \Omega_N^{\epsilon}\\
\lambda \longleftrightarrow &l& \longleftrightarrow \widehat{l}. \nonumber
\end{eqnarray*}
Once the bijections are defined in $(\ref{bijection})$, they will be used throughout the rest of the paper without further comment, i.e., if $\lambda$ is mentioned, we will use $l, \widehat{l}$ to denote the elements of $\Omega_N, \Omega_N^{\epsilon}$ associated to $\lambda\in\GTp_N$. This convention will be ubiquitous in the paper.

	\item We often will make estimates to prove a convergence as one variable tends to infinity. In those estimates, we write $const$ for a positive constant independent of the variable that goes to infinity. The specific value of $const$ may be a different one each time it appears.
\end{enumerate}

\subsection{Organization of the paper}
The present introduction states the main result of the paper in the special cases $(a, b) = (\frac{1}{2}, \frac{1}{2}), (\frac{1}{2}, -\frac{1}{2}), (-\frac{1}{2}, -\frac{1}{2})$.
In Section $\ref{sec:generalitiesfeller}$, we describe the general setting of Feller semigroups on graded graphs and their boundaries.
Of importance, Section $\ref{sec:generalitiesfeller}$ describes the method of intertwiners, which shows how to construct a Feller semigroup on the boundary of the graph from coherent Feller semigroups on the levels of the graph, and also gives conditions for proving that the Feller semigroup on the boundary has a unique invariant measure.
Afterwards, we seek to specialize the general setting for the $BC$ branching graph; such graded graph and its boundary are discussed in Section $\ref{positiveGT}$.
In Section $\ref{semigroupssection}$, we construct Feller semigroups on the levels of the BC branching graph.
Later, in Section $\ref{commuting}$, we show that these semigroups are coherent, effectively proving the existence of Feller semigroups on the boundary of the BC branching graph.
To finish the proof of the main result, in Section $\ref{zmeasuressubsection}$ we prove that the $z$-measures are the unique invariant measures of the Feller semigroups.

We include three appendices. In Appendix $\ref{sec:zmeasuresharmonic}$, we describe the connection between the BC type $z$-measures and the representation theory of infinite symmetric spaces.
In Appendix $\ref{appendixA}$, we prove a coherence relation for shifted symmetric polynomials that was used in Section $\ref{commuting}$ for the proof of coherence, and in Appendix $\ref{appendixB}$, we finish a technical point in a proof from Section $\ref{semigroupssection}$.

\subsection{Acknowledgments}
I'm very grateful to my advisor Alexei Borodin and to Grigori Olshanski, for introducing me to the subject of asymptotic representation theory, for suggesting the problem and for helpful discussions. Thanks also to G. Olshanski for providing an outline for the proof of Proposition $\ref{pro2}$.

\section{Generalities on Feller Semigroups and Graded Graphs}\label{sec:generalitiesfeller}

\subsection{Markov kernels and Markov semigroups}\label{markovkernels}

Let $X, Y$ be two measurable spaces. A \textit{(Markov) kernel} $K: X\dashrightarrow Y$ is a $[0, 1]$-valued map that takes both an element $x\in X$ and a measurable set $A\subset Y$ as arguments and satisfies (a) $K(x, \cdot)$ is a probability measure on $Y$ and (b) $K(\cdot, A)$ is a measurable function on $X$. If $K: X \dashrightarrow Y$ and $L: Y\dasharrow Z$ are kernels, we can compose them to obtain a kernel $KL: X\dasharrow Z$:
\begin{equation*}
KL(x, dz) = \int_Y{K(x, dy)L(y, dz)}.
\end{equation*}

If $X$ and $Y$ are countable discrete spaces, then a Markov kernel $K: X\dasharrow Y$ can be identified with the matrix $M = [M(x, y)]_{{\substack x\in X, \\ y\in Y}}$, $M(x, y) = K(x, \{y\})$. We then say that $M$ is a \textit{stochastic matrix}; it has the property that all its entries are nonnegative and all row sums equal $1$.

We denote by $\MM(X)$ the set of finite signed measures on $X$, and by $\MM_p(X)\subset \MM(X)$ the subset of probability measures on $X$. Then $\MM(X)$ can be made into a Banach space where the total variation of a measure defines its norm. In particular $\MM(X)$, and also $\MM_p(X)$, is a measurable space if equipped with its Borel $\sigma$-algebra. If $X, Y$ are measurable spaces, a kernel $K: X \dashrightarrow Y$ induces a contraction $\MM(X) \rightarrow \MM(Y)$ that we denote $\mu \mapsto \mu K$ and is given by \begin{equation*}
(\mu K)(dy) = \int_X{\mu(dx)K(x, dy)}.
\end{equation*}

Evidently, $\MM(X)\rightarrow\MM(Y)$ maps $\MM_p(X)$ into $\MM_p(Y)$.

If $X$ is a topological space, we denote by $C(X)$ the Banach space of real-valued bounded continuous functions on $X$. A kernel $K: X \dashrightarrow Y$ induces a contraction $C(Y) \rightarrow C(X)$ that we denote $f \mapsto Kf$ and is given by \begin{equation*}
(Kf)(x) = \int_Y{f(y)K(x, dy)}.
\end{equation*}

If $X$ is a locally compact space, we let $C_0(X)\subset C(X)$ be the subspace of continuous functions which vanish at infinity. It is not always the case that $C(Y) \rightarrow C(X)$ maps $C_0(Y)$ into $C_0(X)$; it is only a Feller kernel that has such property and we occupy the next section on the basics of their theory.

Lastly, we define Markov semigroups. Let $X$ be a measurable space; a {\it Markov semigroup on $X$} is a $\R_+$-indexed collection of Markov kernels $P(t) : X \dashrightarrow X$ satisfying the initial condition $P(0; x, F) = \mathbf{1}_{\{x\in F\}}$, for all $x\in X$ and all $F\subset X$ measurable, and the {\it Chapman-Kolmogorov equation} \begin{equation*}
P(t+s) = P(t)P(s), \hspace{.2in}t, s\geq 0.
\end{equation*}

If $X$ is a countable space, the kernels $P(t)$ are essentially (infinite) matrices whose rows and columns are parametrized by $X$. In this case, we will always require that the semigroup $(P(t))_{t\geq 0}$ is \textit{standard}, meaning we have the entry-wise limit
\begin{equation*}
\lim_{t\rightarrow 0^+}{P(t)} = I.
\end{equation*}

Markov semigroups arise naturally as the transition functions of stochastic processes $\{X(t)\}_{t\geq 0}$ with values on $X$.

From the previous discussion, the kernels $P(t)$ define natural maps $P(t): \MM(X)\rightarrow\MM(X)$, which form a semigroup in $\MM(X)$. If $X$ is a topological space, then one similarly obtains a semigroup in $C(X)$.

\subsection{Feller kernels and Feller semigroups}\label{fellersemigroups}

In this section, assume that all spaces are locally compact, second countable and equipped with their Borel $\sigma$-algebra. A Markov kernel $K: X\dashrightarrow Y$ is a {\it Feller kernel} if the induced map $C(Y) \rightarrow C(X)$ sends $C_0(Y)$ into $C_0(X)$.

A Markov semigroup $(P(t))_{t\geq 0}$ on $X$ is a {\it Feller semigroup} if each $P(t): X \dashrightarrow X$ is a Feller kernel and if the semigroup is strongly continuous, i.e., $t\rightarrow P(t)f$ is a continuous map $[0, \infty) \rightarrow C_0(X)$, for each $f\in C_0(X)$.

If $X$ is countable, the latter condition is automatic provided the entries of $P(t)$ are continuous functions in $t$ (the latter follows, in fact, from the condition that the semigroup $(P(t))_{t\geq 0}$ is standard). Moreover if $X$ is countable, the Feller condition can be restated as the following limit relation
\begin{equation}\label{fellercountable}
\lim_{i\rightarrow\infty}{P(t; i, j)} = 0, \hspace{.1in}\textrm{ for any }j\in E.
\end{equation}
The following important theorem can be found in \cite[IV.2.7]{EK}.

\begin{thm}\label{fellertheorem}
Let $(P(t))_{t\geq 0}$ be a Feller semigroup on $X$ and $\nu$ be an arbitrary probability measure on $X$. There is a cadlag stochastic process $\{X(t)\}_{t\geq 0}$ with transition function $(P(t))_{t\geq 0}$ and initial distribution $\nu$.
\end{thm}

\begin{rem}\label{processesfeller}
Due to Theorem $\ref{fellertheorem}$, we will be only interested in finding Feller semigroups on certain space, since they will provide us with Markov processes on that space for any given initial probabilty distribution.
\end{rem}

\subsection{Graded graphs and their boundaries}\label{gradedgraphsgeneralities}

A {\it graded graph} is a graph with countably many vertices, which are partitioned into levels $1, 2, 3, \ldots$ and such that edges have multiplicities and can join only vertices at adjacent levels. We allow edge-multiplicities that are not integers, though all multiplicities must be positive reals. Moreover we impose the condition that any vertex at level $N$ has at least one edge connecting it to some vertex at level $N+1$, for any $N\geq 1$, and at least one edge, but not infinitely many, connecting it to some vertex at level $N-1$, for any $N\geq 2$.

Let $V$ be the set of vertices of a graded graph and $V \myeq \bigsqcup_{N\geq 1}{V_N}$ be its decomposition into levels.
From the structure of the graded graph we are led to define certain Markov kernels $K_N^{N+1}: V_{N+1}\dashrightarrow V_N$, $N\geq 1$, or equivalently, stochastic matrices of format $V_{N+1}\times V_N$, $N\geq 1$.
Let $m(v_{N+1}, v_N)$ be the multiplicity of the edge between $v_{N+1}\in V_{N+1}$ and $v_N\in V_N$, and set $m(v_{N+1}, v_N) := 0$ if $v_{N+1}$ and $v_N$ are not connected by an edge; then $K^{N+1}_N$ is given by
\begin{equation}\label{finitekernelgeneral}
K^{N+1}_N(v_{N+1}, v_N) = \frac{m(v_{N+1}, v_N)}{\sum_{v\in V_N}{m(v_{N+1}, v)}}, \hspace{.2in}v_{N+1}\in V_{N+1}, \ v_N\in V_N.
\end{equation}
The Markov kernels $V_N \dashleftarrow V_{N+1}$ induce a chain of maps
\begin{equation*}
\MM_p(V_1) \leftarrow \MM_p(V_2) \leftarrow \cdots\leftarrow \MM_p(V_N) \leftarrow \MM_{p}(V_{N+1})\leftarrow\cdots.
\end{equation*}
The inverse limit $\lim_{\leftarrow}{\MM_p(V_N)}$ is evidently a convex set, whose elements are sequences of probability measures $\{M_N\}_{N = 1, 2, \ldots}$ on the sets $\{\GTp_N\}_{N = 1, 2, \ldots}$ satisfying the natural coherence conditions $M_N = M_{N+1}K^{N+1}_N$, $N\geq 1$, or equivalently,
\begin{equation*}
M_N(v_N) = \sum_{v_{N+1}\in V_{N+1}}{M_{N+1}(v_{N+1})K^{N+1}_N(v_{N+1}, v_N)}, \textrm{ for each }v_N\in V_N.
\end{equation*}

\begin{df}\label{boundarydef}
In the setup above, the set of extreme points $V_{\infty}$ of the convex set $\lim_{\leftarrow}{\MM_p(V_N)}$ is called the {\it boundary of the graded graph}. The boundary has a natural topology, which is the inherited one from $\lim_{\leftarrow}{\MM_p(V_N)}$ with its projective limit topology.

Since the boundary $V_{\infty}$ has a topology, and therefore a Borel $\sigma$-algebra, we can define the Banach space $\MM_p(V_{\infty})$. A general result, see \cite[Thm. 9.2]{Ol2}, shows that if $V_{\infty} \neq \emptyset$, there is a natural map
\begin{equation}\label{boundarydefeqn}
\mathcal{M}_p(V_{\infty}) \longrightarrow \lim_{\leftarrow}{\mathcal{M}_p(V_N)},
\end{equation}
which is a bijection of sets. In many cases, the bijection $(\ref{boundarydefeqn})$ is an isomorphism of measurable spaces when both sides of $(\ref{boundarydefeqn})$ are equipped with their natural $\sigma$-algebras. If this is the case, we say the boundary $V_{\infty}$ is \textit{ordinary}.
\end{df}

The bijection $(\ref{boundarydefeqn})$ arises as follows. For each $N$, consider the natural map
\begin{equation*}
\phi_N: V_{\infty} \subset \lim_{\leftarrow}{\mathcal{M}_p(V_N)} \rightarrow \mathcal{M}_p(V_N),
\end{equation*}
which is the composition of an inclusion and projection. Then there are natural Markov kernels $K^{\infty}_N: V_{\infty} \dashrightarrow V_N$, given by $K_N^{\infty}(x, A) = \phi_N(x)(A)$, for all $x\in V_{\infty}$ and (Borel) subset $A\subset V_N$. The sequence of kernels $\{K^{\infty}_N\}_{N \geq 1}$ is coherent in the sense that
\begin{equation}\label{cohinftyN}
K^{\infty}_{N+1}K^{N+1}_N = K^{\infty}_N, \textrm{ for all } N\geq 1.
\end{equation}
The induced map $\MM_p(V_{\infty}) \longrightarrow \lim_{\leftarrow}{\MM_p(V_N)}$ coming from the coherent maps $K^{\infty}_N: \MM_p(V_{\infty}) \longrightarrow \MM_p(V_N)$ is the bijective map in $(\ref{boundarydefeqn})$. Next we make the following definition of a Feller boundary.

\begin{df}\label{boundaryfeller}
Let $V_{\infty}$ be the boundary of a graded graph and assume that $V_{\infty}\neq\emptyset$. We say that $V_{\infty}$ is a \textit{Feller boundary} if
\begin{itemize}
	\item $V_{\infty}$ is locally compact, equipped with its Borel $\sigma$-algebra.
	\item All the Markov kernels $K^{N+1}_N$ and $K^{\infty}_N$ are Feller.
	\item The bijection $(\ref{boundarydefeqn})$ is an isomorphism of measurable spaces.
\end{itemize}
\end{df}

\subsection{Markov processes on the boundary: method of intertwiners}\label{methodinter}
For an expanded explanation on the method of intertwiners, see \cite{Ol1}. Here we will be content with stating the following theorem, which is proved in \cite[Prop. 2.4 and Sec. 2.8]{BO2}.

\begin{thm}\label{methodintertwiners}
\normalfont Let $(P_N(t))_{t\geq 0}$ be Feller semigroups on the levels $V_N$ of the graded graph in the setting above. Assume moreover that the following \textit{master relation}
\begin{equation}\label{commutativity}
P_{N+1}(t)\Lambda^{N+1}_N = \Lambda^{N+1}_NP_N(t)
\end{equation}
holds for all $t\geq 0$, $N\geq N_0$, for some $N_0\in\N$. Additionally assume that the boundary $V_{\infty}$ is nonempty and is a Feller boundary. Then there exists a unique Feller semigroup $(P_{\infty}(t))_{t\geq 0}$ on $V_{\infty}$ satisfying
\begin{equation}\label{commutativity2}
P_{\infty}(t)\Lambda^{\infty}_N = \Lambda^{\infty}_NP_N(t) \ \forall \ t\geq 0, \ N\geq N_0.
\end{equation}
Moreover, there exists a unique probability measure $\nu_{\infty}$ on $V_{\infty}$ invariant with respect to $(P_{\infty}(t))_{t\geq 0}$ if and only if there exist unique probability measures $\nu_N$ on $V_N$, for all $N\geq N_0$, that are invariant with respect to $(P_N(t))_{t\geq 0}$ and such that they are coherent, i.e., $\nu_N = \nu_{N+1}\Lambda^{N+1}_N$ for all $N\geq N_0$.

For all $1\leq K < N_0$, let $\nu_K \myeq \nu_{N_0}\Lambda^{N_0}_{N_0 - 1}\cdots\Lambda^{K+1}_K$. If the condition above is satisfied, $\nu_{\infty}$ corresponds to the coherent sequence $\{\nu_N\}_{N=1, 2, \ldots}$, under the bijection in $(\ref{boundarydefeqn})$.
\end{thm}

\begin{rem}
In the source cited above, the method of intertwiners is stated only for $N_0 = 1$. However that version implies the one given above, because the inverse limits $\lim_{\leftarrow N\geq 1}{\mathcal{M}_p(V_N)}$ and $\lim_{\leftarrow N\geq N_0}{\mathcal{M}_p(V_N)}$ are canonically isomorphic, and therefore the boundary $\Omega_{\infty}$ of the original graph also serves as the boundary of the truncated graph with levels $\geq N_0$.
\end{rem}

\section{BC Branching Graph and its Boundary}\label{positiveGT}

In this section, we specialize the general setting of Section $\ref{gradedgraphsgeneralities}$ for the BC branching graph, describe its boundary $\Omega_{\infty}$ and prove that $\Omega_{\infty}$ is a Feller boundary.

\subsection{BC branching graph}\label{posGT}

For any integer $N\geq 1$, let $\GTp_N \myeq \{\lambda = (\lambda_1 \geq \lambda_2 \geq \ldots \geq \lambda_N) : \lambda_i \in \Z, \ \lambda_i \geq 0 \ \forall i = 1, \dots, N\}$ be the set of $N$-tuples of weakly decreasing nonnegative integers. Also set $\GT^+ \myeq \bigsqcup_{N\geq 1}{\GTp_N}$. Elements of $\GT^+$ are called {\it positive signatures}. Note that we do not identify positive signatures which differ by trailing zeroes; for example $(4, 2, 1, 0)$ and $(4, 2, 1, 0, 0, 0)$ are distinct elements of $\GTp$, the first one belonging to $\GTp_4$ and the second one belonging to $\GTp_6$. Elements of $\GTp_N$ are called {\it $N$-positive signatures}.

We shall construct the {\it branching graph of classical Lie groups of type $B, C, D$} (or simply {\it BC branching graph}) as a graded graph with vertex set $\GT^+$. The graph will satisfy the property that only vertices in adjacent levels, i.e., levels $\GT_m^+$ and $\GT_{m+1}^+$ for some $m\in\Zp$, can be joined by an edge of certain multiplicity. The edges are determined as follows: an edge connects $\lambda\in\GT_{N+1}^+$ and $\mu\in\GTp_N$ (with certain positive multiplicity $m(\lambda, \mu) > 0$) if and only if there exists $\nu\in\GTp_N$ such that
\begin{equation*}
\lambda\succ\nu \hspace{.1in}\textrm{ and }\hspace{.1in} \nu\cup 0\succ\mu,
\end{equation*}
where $\lambda = (\lambda_1, \ldots, \lambda_N, \lambda_{N+1}) \succ \nu = (\nu_1, \ldots, \nu_N)$ means the interlacing
\begin{equation*}
\lambda_1 \geq \nu_1 \geq \lambda_2 \geq \ldots \geq \lambda_N \geq \nu_N \geq \lambda_{N+1}.
\end{equation*}
Similarly, the relation $\nu\cup 0 = (\nu_1, \ldots, \nu_N, 0) \succ \mu = (\mu_1, \ldots, \mu_N)$ means
\begin{equation*}
\nu_1 \geq \mu_1 \geq \nu_2 \geq \ldots \geq \nu_N \geq \mu_N \geq 0.
\end{equation*}
In combinatorial language, the condition above is equivalent to the existence of $\nu\in\GTp_N$ such that the containment of Young diagrams $\mu\subset\nu\subset\lambda$ holds and the skew-shapes $\lambda/\nu$, $\nu/\mu$ are horizontal strips.

If $\lambda\in\GTp_{N+1}$ and $\mu\in\GTp_N$ are joined by an edge in the BC branching graph, we write
\begin{equation*}
\lambda\succ_{BC}\mu\hspace{.2in}\textrm{ or }\hspace{.2in}\mu\prec_{BC}\lambda.
\end{equation*}

The multiplicities $m(\lambda, \mu)$, $\lambda\in\GTp_{N+1}$, $\mu\in\GTp_N$, are given explicitly in terms of a family of classical multivariate orthogonal polynomials, namely $m(\lambda, \mu) = C_N^{a, b}(\lambda, \mu)\cdot\mathfrak{P}_{\mu}(1^N | a, b)$, where the \textit{Jacobi polynomials} $\mathfrak{P}_{\mu}(\cdot | a, b)$ and the \textit{branching coefficient} $C_N^{a, b}$ are defined below in $(\ref{multjacobi})$ and $(\ref{secondhelp})$, respectively. We shall not make explicit use of the multiplicities $m(\lambda, \mu)$, but rather of the kernels $\Lambda^{N+1}_N$ defined from them as in $(\ref{finitekernelgeneral})$. We therefore only give explicit expressions for $\Lambda^{N+1}_N$ in Section $\ref{finitekernels}$ and forget about the multiplicities $m(\lambda, \mu)$.

Finally let us define some bijections of $\GTp_N$ that will be used in the rest of the paper. Let
\begin{equation*}
\Omega_N \myeq \{(x_1 > x_2 > \ldots > x_N) \in \Zp^N\}
\end{equation*}
be the set of $N$-tuples of strictly decreasing nonnegative integers and also let
\begin{equation*}
\Omega_N^{\epsilon} \myeq \{(\hatx_1 > \hatx_2 > \ldots > \hatx_N) : \hatx_i\in\Zp^{\epsilon}\},
\end{equation*}
where $\hatx = (x + \epsilon)^2$, for any $x\in\Zp$.

There is a clear bijection between the three infinite sets above given by
\begin{gather}
\GTp_N \longleftrightarrow \Omega_N \longleftrightarrow \Omega_N^{\epsilon}\label{bijection}\\
\lambda = (\lambda_1, \ldots, \lambda_N) \longleftrightarrow l = (l_1, \ldots, l_N) \longleftrightarrow \widehat{l} = (\hatl_1, \ldots, \hatl_N),\nonumber
\end{gather}
where
\begin{eqnarray*}
l_i &=& \lambda_i + N - i, \hspace{.23in}i = 1, 2, \ldots, N\\
\hatl_i &=& (l_i + \epsilon)^2, \hspace{.4in} i = 1, 2, \ldots, N.
\end{eqnarray*}

For notation, let us agree that $\lambda\in\GT^+$ corresponds to $l\in\bigsqcup_{N \geq 1}{\Omega_N}$ and $\hatl\in\bigsqcup_{N\geq 1}{\Omega_N^{\epsilon}}$. Similarly, let us agree that $\mu, \nu\in\GT^+$ correspond to $m, n\in\bigsqcup_{N \geq 1}{\Omega_N}$ and $\hatm, \hatn\in\bigsqcup_{N \geq 1}{\Omega_N^{\epsilon}}$, respectively. This terminology will be in place throughout the paper.

\subsection{Jacobi polynomials}\label{jacobi}

In order to define the kernels $\Lambda^{N+1}_N : \GT_{N+1}^+ \dashrightarrow\GTp_N$ in the next section, we recall here some definitions and results on the multivariate analogues of classical Jacobi polynomials. In this section, we can work with any parameters $a, b > -1$.

Let $\mathfrak{m}(dx)$ be the measure on $[-1, 1]$ which is absolutely continuous with density
\begin{equation*}
(1 - x)^a(1 + x)^b,
\end{equation*}
with respect to the Lebesgue measure on $[-1, 1]$. The classical (univariate) Jacobi polynomials $\{\mathfrak{P}_k(x | a, b)\}_{k\in\Zp}$, are the elements of the Hilbert space $H = L^2([-1, 1], \mathfrak{m}(dx))$ coming from the orthogonalization of the basis $\{1, x, x^2, \ldots\}$; in particular, $\mathfrak{P}_0 = 1$, $\deg\mathfrak{P}_k = k$ for all $k\in\Zp$, and $\{\mathfrak{P}_k\}_{k\in\Zp}$ is an orthogonal basis of $H$.

For any $\lambda\in\GTp_N$, we can define the {\it multivariate Jacobi polynomial}
\begin{eqnarray}\label{multjacobi}
\mathfrak{P}_{\lambda}(x_1, \ldots, x_N | a, b) \myeq \frac{\det_{1\leq i, j\leq N}[{\mathfrak{P}_{\lambda_i + N - i}(x_j | a, b)}]}{\Delta_N(x)},
\end{eqnarray}
where $\Delta_N(x)$ is the Vandermonde determinant
\begin{eqnarray*}
\Delta_N(x) = \prod_{1\leq i < j\leq N}{(x_i - x_j)}.
\end{eqnarray*}
We let $H_N\subset L^2([-1, 1]^N, \mathfrak{m}_N(dx))$ be the Hilbert space of symmetric functions on $N$ variables $x_i\in [-1, 1]$, $i = 1, \dots, N$, that are square integrable with respect to the measure $\mathfrak{m}_N(dx)$ on $[-1, 1]^N$ with density
\begin{eqnarray}\label{densityN}
\Delta_N(x)^2\prod_{i=1}^N{(1 - x_i)^a(1 + x_i)^b},
\end{eqnarray}
with respect to the Lebesgue measure on $[-1, 1]^N$.

The Jacobi polynomials $\{\mathfrak{P}_{\lambda}\}_{\lambda\in\GTp_N}$ form an orthogonal basis of $H_N$. They are not homogeneous polynomials; in fact, $\mathfrak{P}_{\lambda}$ has degree $|\lambda| \myeq \lambda_1 + \ldots + \lambda_N$ and its highest degree homogeneous component is a multiple of the Schur polynomial parametrized by $\lambda$ (see \cite{M} for the definition and combinatorial properties of Schur polynomials): \begin{equation*}
\mathfrak{P}_{\lambda}(x_1, \ldots, x_N | a, b) = c_{\lambda|N}^{a, b}s_{\lambda}(x_1, \ldots, x_N) + \textrm{ lower degree terms}, \hspace{.2in}c_{\lambda|N}^{a, b} \neq 0.
\end{equation*}

By virtue of the fact that Schur polynomials $s_{\lambda}(x_1, \ldots, x_N)$, $\lambda\in\GTp_N$, form a basis of the algebra of symmetric polynomials in $N$ variables, the polynomials $\{\mathfrak{P}_{\lambda}\}_{\lambda\in\GTp_N}$ also form a basis of the algebra of symmetric polynomials in the variables $x_1, \ldots, x_N$.

We also consider the following \textit{normalized Jacobi polynomials}
\begin{eqnarray}\label{normjacobi}
\Phi_{\lambda}(x_1, \ldots, x_N | a, b) \myeq \frac{\mathfrak{P}_{\lambda}(x_1, \ldots, x_N | a, b)}{\mathfrak{P}_{\lambda}(1, \ldots, 1 | a, b)},
\end{eqnarray}
where there are $N$ ones in the denominator. The normalization is such that $\Phi_{\lambda}(1^N | a, b) = 1$, for any $\lambda\in\GTp_N$. Definition $(\ref{normjacobi})$ is allowed because the identity
\begin{eqnarray}\label{evalones}
\mathfrak{P}_{\lambda}(1^N | a, b) = \Delta_N(\hatl)\cdot\prod_{i=1}^N{\frac{\Gamma(l_i + a + 1)2^{- N + i}}{\Gamma(l_i + 1)\Gamma(a+i)\Gamma(i)}},
\end{eqnarray}
shows that the denominator $\mathfrak{P}_{\lambda}(1^N | a, b)$ of $(\ref{normjacobi})$ is nonzero. In the formula $(\ref{evalones})$, we recall the notation $l = (l_1, \ldots, l_N)$, $l_i = \lambda_i + N - i$ and $\hatl = (\hatl_1, \ldots, \hatl_N)$, $\hatl_i = (l_i + \epsilon)^2$. Identity $(\ref{evalones})$ was proved in \cite[Prop. 7.1]{OlOs}.

\subsection{Markov kernels $\GT_{N+1}^+ \dashrightarrow \GTp_N$}\label{finitekernels}

Since all sets $\GTp_N$ are countable, a Markov kernel $\GT_{N+1}^+ \dashrightarrow \GTp_N$ is given by a stochastic matrix $[\Lambda^{N+1}_N(\lambda, \mu)]_{\lambda\in\GT_{N+1}^+, \mu\in\GTp_N}$ of format $\GT_{N+1}^+ \times \GTp_N$. We then define our desired Markov kernel by considering the coefficients of the following branching of normalized Jacobi polynomials \begin{eqnarray}\label{linksdef}
\Phi_{\lambda}(x_1, \ldots, x_N, 1 | a, b) = \sum_{\mu\in\GTp_N}{\Lambda^{N+1}_N(\lambda, \mu)\Phi_{\mu}(x_1, \ldots, x_N | a, b)}, \hspace{.2in}\lambda\in\GTp_{N+1}.
\end{eqnarray}

In more detail, $\Phi_{\lambda}(x_1, \ldots, x_N, 1 | a, b)$ is a symmetric polynomial in the variables $x_1, \ldots, x_N$, while the polynomials $\{\Phi_{\mu}\}_{\mu\in\GTp_N}$ form a basis of the algebra of symmetric polynomials on those $N$ variables. Therefore $\Phi_{\lambda}(x_1, \ldots, x_N, 1 | a, b)$ can be expressed as a finite linear combination of the polynomials $\{\Phi_{\mu}\}_{\mu\in\GTp_N}$; the coefficients of such unique linear combination determine the matrix $[\Lambda^{N+1}_N(\lambda, \mu)]$. Observe that each row of $[\Lambda^{N+1}_N(\lambda, \mu)]$ has finitely many nonzero entries. The coefficients $\Lambda^{N+1}_N(\lambda, \mu)$ generally depend on $a, b$, but for simplicity we omit them from the notation.

Let us check that $[\Lambda^{N+1}_N(\lambda, \mu)]$ is a stochastic matrix. Setting $x_1 = \ldots = x_N = 1$ in $(\ref{linksdef})$ gives
\begin{equation*}
1 = \sum_{\mu\in\GTp_N}{\Lambda^{N+1}_N(\lambda, \mu)}.
\end{equation*}
The fact that $\Lambda^{N+1}_N(\lambda, \mu) \geq 0$, for all $\mu\in\GTp_N$, is shown below, see $(\ref{links})$. Then $[\Lambda^{N+1}_N(\lambda, \mu)]_{\lambda, \mu}$ is indeed a stochastic matrix and determines a kernel $\GTp_N \dashleftarrow \GTp_{N+1}$.

An explicit expression for $\Lambda^{N+1}_N(\lambda, \mu)$ can be extracted from the proof of \cite[Prop. 7.5]{OO1}. Indeed, first observe that
\begin{equation}\label{firsthelp}
\Lambda^{N+1}_N(\lambda, \mu) = \frac{\mathfrak{P}_{\mu}(1^N | a, b)}{\mathfrak{P}_{\lambda}(1^{N+1} | a, b)}C^{a, b}_N(\lambda, \mu),
\end{equation}
where the coefficients $C^{a, b}_N(\lambda, \mu)$ are given by the expansion
\begin{equation}\label{secondhelp}
\mathfrak{P}_{\lambda}(x_1, \ldots, x_N, 1 | a, b) = \sum_{\mu\in\GTp_N}{C^{a, b}_{N}(\lambda, \mu)\mathfrak{P}_{\mu}(x_1, \ldots, x_N | a, b)}.
\end{equation}
In \cite[Sec. 7]{OO1}, it is found that $C^{a, b}_N(\lambda, \mu) = 0$ unless $\lambda \succ_{BC} \mu$, in which case
\begin{eqnarray}
C^{a, b}_N(\lambda, \mu) &=& \frac{\prod_{i=1}^{N+1}{\mathfrak{P}_{\lambda_i + N + 1 - i}(1 | a, b)}}{\prod_{i=1}^N{\mathfrak{P}_{\mu_i + N - i}(1 | a, b)}}\sum_{\substack{\nu \prec \lambda \\ \mu \prec \nu\cup 0}}{A(\nu, \mu)},\nonumber\\
\textrm{where } A(\nu, \mu) &=& \prod_{i=1}^N{B(\nu_i + N - i, \mu_i + N - i)},\nonumber\\
B(r, s) &=& \frac{(2r + a + b + 2)(2s + a + b + 1) r!}{2\cdot s!}\nonumber\label{Bmldef}\\
&& \times\frac{\Gamma(s+a+b+1)\Gamma(r+b+1)\Gamma(s+a+1)}{\Gamma(r+a+b+2)\Gamma(s+b+1)\Gamma(r+a+2)}, \hspace{.1in} s > 0,\\
B(r, 0) &=& \frac{(2r+a+b+2)r!\cdot\Gamma(r+b+1)\Gamma(a+1)\Gamma(a+b+2)}{2\cdot\Gamma(r+a+b+2)\Gamma(r+a+2)\Gamma(b+1)}.\label{Bmzerodef}
\end{eqnarray}
Thus, from $(\ref{evalones})$, we have
\begin{eqnarray}\label{links}
\Lambda^{N+1}_N(\lambda, \mu) &=& 2^NN!\cdot\frac{\Gamma(N+a+1)}{\Gamma(a+1)}\frac{\Delta_N(\hatm)}{\Delta_{N+1}(\hatl)}\cdot\widetilde{\Lambda}^{N+1}_N(\lambda, \mu)\label{stochmatrix},\\
\textrm{where }\hspace{.15in}\widetilde{\Lambda}^{N+1}_N(\lambda, \mu) &=& \sum_{\substack{\nu\in\GTp_N \\ \nu \prec \lambda \\ \mu \prec \nu\cup 0}}{\prod_{i=1}^N{B(\nu_i + N - i, \mu_i + N - i)}},\label{sumterm}
\end{eqnarray}
and $\hatl, \hatm$ correspond to $\lambda, \mu$ as described in $(\ref{bijection})$. The sum in $(\ref{sumterm})$ vanishes unless $\mu \prec_{BC} \lambda$ (for only in that case, there exists at least one $\nu\in\GTp_N$ with $\nu \prec \lambda$ and $\mu \prec \nu\cup 0$). From the formula $(\ref{stochmatrix})$ above, it is clear that $\Lambda^{N+1}_N(\lambda, \mu) \geq 0$, for all $\lambda\in\GTp_{N+1}$, $\mu\in\GTp_N$.

\subsection{The boundary of the BC branching graph}\label{boundary}

In this section and the next, we review the explicit description of the boundary of the BC branching graph, following \cite{OO1}.

Consider the space $\R^{2\infty + 1} \myeq \R^{\infty}\times\R^{\infty}\times\R$, equipped with its product topology. Let $\Omega_{\infty} \subset \R^{2\infty + 1}$ be the subset of points $\omega = (\alpha, \beta, \delta)$ satisfying
\begin{eqnarray*}
\alpha &=& (\alpha_1 \geq \alpha_2 \geq \ldots \geq 0)\\
\beta &=& (1\geq \beta_1 \geq \beta_2 \geq \ldots \geq 0)\\
\delta &\geq& \sum_{i=1}^{\infty}{(\alpha_i + \beta_i)}.
\end{eqnarray*}
Then $\Omega_{\infty}$ is a closed subspace of $\R^{2\infty + 1}$, with the inherited topology from $\R^{2\infty + 1}$. Moreover, $\Omega_{\infty}$ is a locally compact space with a countable base. It will be convenient to set
\begin{eqnarray}\label{gammadef}
\gamma = \delta - \sum_{i=1}^{\infty}{(\alpha_i  + \beta_i)} \geq 0.
\end{eqnarray}

The space $\Omega_{\infty}$ is in bijection with the boundary of the BC branching graph, as in Definition $\ref{boundarydef}$. Recall that the boundary of the BC branching graph is the set of extreme points of $\lim_{\leftarrow}{\mathcal{M}_p(\GTp_N)}$. If we give the space $\lim_{\leftarrow}{\mathcal{M}_p(\GTp_N)}$ the projective limit topology and its set of extreme points the subspace topology, then the boundary of the BC branching graph is actually homeomorphic to $\Omega_{\infty}$, with its topology inherited from $\R^{2\infty + 1}$.

We shall only be interested in the bijection $\MM_p(\Omega_{\infty}) \longrightarrow \lim_{\leftarrow}{\MM_p(\GTp_N)}$, cf. $(\ref{boundarydefeqn})$, so we focus on its description rather than in the description of the bijection between $\Omega_{\infty}$ and the set of extreme points of $\lim_{\leftarrow}{\MM_p(\GTp_N)}$.

To accomplish it we define, in the next section, Markov kernels $\Lambda^{\infty}_N: \Omega_{\infty} \dashrightarrow \GTp_N$ which are coherent in the sense that $\Lambda^{\infty}_{N+1}\Lambda^{N+1}_N = \Lambda^{\infty}_N$, cf. $(\ref{cohinftyN})$.

\subsection{Markov kernels $\Omega_{\infty}\dashrightarrow \GTp_N$}\label{infinitekernel}

For each $\omega = (\alpha, \beta,\delta)\in\Omega_{\infty}$, define the function
\begin{equation}\label{psifunction}
\Psi(x; \omega) = e^{\gamma(x - 1)}\prod_{i=1}^{\infty}{\frac{1 + \beta_i(2 - \beta_i)(x - 1)/2}{1 - \alpha_i(2 + \alpha_i)(x - 1)/2}}, \hspace{.1in} x\in [-1, 1],
\end{equation}
where $\gamma$ is defined from $\omega$ as in $(\ref{gammadef})$.
Since the sum $\sum_{i=1}^{\infty}{(\alpha_i + \beta_i)}$ is convergent, the function $\Psi(x; \omega)$ is holomorphic in a complex neighborhood of $[-1, 1]$, and in particular $\Psi(x; \omega)$ is a continuous function on $[-1, 1]$. Moreover $\Psi(1; \omega) = 1$ and $|\Psi(x; \omega)| = \Psi(x; \omega) \leq 1, \ \forall x\in [-1, 1]$.

In analogy to Section $\ref{finitekernels}$, we define a kernel of format $\Omega_{\infty}\times\GTp_N$ via the following branching equation
\begin{eqnarray}\label{markovinftyN}
\Psi(x_1; \omega)\cdots\Psi(x_N; \omega) = \sum_{\lambda\in\GTp_N}{\Lambda^{\infty}_N(\omega, \lambda)\Phi_{\lambda}(x_1, \ldots, x_N | a, b)}, \hspace{.1in} x_i \in [-1, 1].
\end{eqnarray}
In more detail, $\Psi(x_1; \omega)\cdots\Psi(x_N; \omega)$ is a continuous function on  $[-1, 1]^N$ and therefore it belongs to the Hilbert space $H_N$ of symmetric functions on $[-1, 1]^N$ which are square-integrable with respect to $\mathfrak{m}_N(dx)$, the measure with density $(\ref{densityN})$ with respect to the Lebesgue measure on $[-1, 1]^N$.
The normalized Jacobi polynomials $\Phi_{\lambda}(\cdot | a, b)$ form an orthogonal basis of $H_N$, so an expansion like in $(\ref{markovinftyN})$ exists and is unique. Just as for the finite kernels $\Lambda^{N+1}_N$ defined previously, we omit $a, b$ from the notation of the kernels $\Lambda^{\infty}_N$.

Let us check that the terms $\Lambda^{\infty}_N(\omega, \lambda)$ indeed define a Markov kernel. By setting $x_1 = x_2 = \ldots = x_N = 1$ in $(\ref{markovinftyN})$, we obtain \begin{equation*}
1 = \sum_{\lambda\in\GTp_N}{\Lambda^{\infty}_N(\omega, \lambda)}.
\end{equation*}

The fact that $\Lambda^{\infty}_N(\omega, \lambda) \geq 0$, for all $\omega\in\Omega_{\infty}$, $\lambda\in\GTp_N$ is not obvious, but it can be deduced from the results of \cite{OO1}.
The continuity of $\Lambda^{\infty}_N(\omega, \lambda)$ as a function of $\omega$ can be deduced from the continuity of the map $\Omega \rightarrow C([-1, 1]), \ \omega \mapsto \Psi(\cdot; \omega)$. The continuity of $\omega \mapsto \Psi(\cdot; \omega)$ can be proved by mimicking the approach of \cite[Proof of Theorem 8.1, Step 1]{Ol2}. Then the expressions $\Lambda^{\infty}_N(\omega, \lambda)$ yield a Markov kernel $\Omega_{\infty}\dashrightarrow\GTp_N$.

An evident property, based on the definitions $(\ref{linksdef})$ and $(\ref{markovinftyN})$ is the following
\begin{equation*}
\Lambda^{\infty}_{N+1}\Lambda^{N+1}_N = \Lambda^{\infty}_N, \hspace{.2in} N = 1, 2, \ldots.
\end{equation*}

Therefore the Markov kernels $\Lambda^{\infty}_N: \Omega_{\infty} \dashrightarrow \GTp_N$ induce maps $\Lambda^{\infty}_N: \MM_p(\Omega_{\infty}) \rightarrow \MM_p(\GTp_N)$, which are also coherent in the sense that $\Lambda^{\infty}_{N+1}\Lambda^{N+1}_N = \Lambda^{\infty}_N$.
Such coherence allows us to define
\begin{equation}\label{themap}
\MM_p(\Omega_{\infty}) \longrightarrow \lim_{\leftarrow}{\MM_p(\GTp_N)},
\end{equation}
which one can prove is a bijection by combining the general results from \cite{OO1} and \cite[Thm. 9.2]{Ol2}. In fact, an argument similar to the proof of \cite[Thm. 3.1]{BO2} shows that $(\ref{themap})$ is actually an isomorphism of measurable spaces if both spaces are equipped with their natural $\sigma$-algebras, therefore the third item of Definition $\ref{boundarydef}$ is satisfied. The bijection $(\ref{themap})$ will be important later to access the definition of $z$-measures in a very concrete way.

\subsection{Feller property of the BC branching graph}

We prove that the boundary $\Omega_{\infty}$ is Feller in the sense of Definition $\ref{boundaryfeller}$.
We have already argued for the first and third bullets in the definition, and are left to prove the second bullet.
It will follow from Propositions $\ref{pro1}$ and $\ref{pro2}$ below.

\begin{prop}\label{pro1}
For each $N \geq 1$, the Markov kernels $\Lambda^{N+1}_N : \GT_{N+1}^+ \dashrightarrow \GTp_N$ defined in Section $\ref{finitekernels}$ are Feller kernels.
\end{prop}
\begin{proof}
We need to prove that the map $C(\GTp_N) \longrightarrow C(\GTp_{N+1})$ induced by $\Lambda^{N+1}_N$ sends $C_0(\GTp_N)$ into $C_0(\GTp_{N+1})$. As the operator $C(\GTp_N) \longrightarrow C(\GT_{N+1}^+)$ is bounded (actually a contraction), $C_0(\cdot)$ is a closed subspace of $C(\cdot)$ and the delta functions span a dense subspace of $C_0(\cdot)$, it suffices to prove that $\Lambda^{N+1}_N\delta_{\mu}\in C_0(\GT_{N+1}^+)$, for all $\mu\in\GTp_N$. Equivalently, we shall prove
\begin{equation*}
\Lambda^{N+1}_N\delta_{\mu}(\lambda) = \Lambda^{N+1}_N(\lambda, \mu) \longrightarrow 0 \textrm{ as } \lambda \longrightarrow \infty, \hspace{.1in} \textrm{for any }\mu\in\GTp_N.
\end{equation*}

Given $\lambda\in\GT_{N+1}^+$, we have from $(\ref{stochmatrix})$ and $(\ref{sumterm})$ that
\begin{equation}\label{expr1}
\Lambda^{N+1}_N(\lambda, \mu) = 2^NN!\cdot\frac{\Gamma(N+a+1)}{\Gamma(a+1)}\frac{\Delta_N(\hatm)}{\Delta_{N+1}(\hatl)}\cdot\widetilde{\Lambda}^{N+1}_N(\lambda, \mu),
\end{equation}
where
\begin{eqnarray*}
\widetilde{\Lambda}^{N+1}_N(\lambda, \mu) &=& \sum_{\substack{\nu\in\GTp_N \\ \nu \prec \lambda \\ \mu \prec \nu\cup 0}}{A(\nu, \mu)} = \sum_n{\sum_{\nu}{A(\nu, \mu)}}\\
&=& \sum_n{\left(B(n + N - 1, \mu_1 + N - 1)\cdot\sum_{\nu}{\prod_{i=2}^N{B(\nu_i + N - i, \mu_i + N - i)}}\right)}.
\end{eqnarray*}
In the sums above, $n$ ranges over $n\in\Zp$, $\lambda_1 \geq n \geq \max\{\lambda_2, \mu_1\}$, and $\nu$ ranges over $\nu\in\GTp_N$ such that $\nu\prec\lambda$, $\mu\prec\nu\cup 0$ and $\nu_1 = n$. Observe that the second sum over $\nu$ is finite and the summand does not depend on the first coordinate $\nu_1$.

Well-known estimates of the Gamma function, cf. \cite[Cor. 1.4.3]{AAR}, give
\begin{equation*}
\frac{r!}{\Gamma(r + a + 2)}, \hspace{.1in}\frac{\Gamma(r + b + 1)}{\Gamma(r + a + b + 2)} \sim r^{-a - 1}, \hspace{.2in} r\longrightarrow+\infty.
\end{equation*}
If $s\neq 0$ is fixed, then by definition $(\ref{Bmldef})$ of $B(r, s)$, and the previous estimate, we obtain
\begin{eqnarray}\label{estimate1}
B(r, s) \leq const\cdot (1 + r^{-2a-1}), \ \forall r\in\Zp.
\end{eqnarray}
Similar reasoning using the definition $(\ref{Bmzerodef})$ of $B(r, 0)$ shows that $(\ref{estimate1})$ holds also when $s = 0$ and $r\rightarrow +\infty$. Moreover, $\mu\prec\nu\cup 0$ implies $\mu_1 \geq \nu_2 \geq \mu_2 \geq \ldots$, therefore
\begin{equation}\label{estimate2}
\sum_{\nu}{\prod_{i=2}^N{B(\nu_i + N - i, \mu_i + N - i)}} \leq const
\end{equation}
where the right-hand side is a constant independent of $\lambda$.

From estimates $(\ref{estimate1}), (\ref{estimate2})$ above, it follows that $\widetilde{\Lambda}^{N+1}_N(\lambda, \mu)$ is upper-bounded by a constant independent of $\lambda$ times
\begin{gather*}
\sum_{\lambda_1\geq n\geq \max\{\lambda_2, \mu_1\}}{B(n + N - 1, \mu_1 + N - 1)}\leq const\cdot(1 + \sum_{n = \lambda_2}^{\lambda_1}{n^{-1-2a}})\\
\leq const\cdot(\lambda_1 - \lambda_2 + 1).
\end{gather*}
The latter inequality above follows because $a \geq -1/2$ implies $-1-2a \leq 0$, and so $n^{-1-2a} \leq 1$ for all $n\in\N$. Then $\sum_{n = \lambda_2}^{\lambda_1}{n^{-1-2a}} \leq \sum_{n = \lambda_2}^{\lambda_1}{1} = \lambda_1 - \lambda_2 + 1$. Of course, the argument above does not make sense if $\lambda_2 = 0$. But if $\lambda_2 = 0$, the bound is evident too.

On the other hand, we can estimate the Vandermonde determinant as
\begin{gather*}
\Delta_{N+1}(\hatl) = \prod_{i<j}{(\hatl_i - \hatl_j)} = \prod_{i<j}{(l_i-l_j)(l_i+l_j+2\epsilon)}\\
\geq (\lambda_1 - \lambda_2 + 1)\prod_{j=2}^{N+1}{(\lambda_1 + \lambda_j + 2\epsilon)} \geq (\lambda_1 - \lambda_2+1)\lambda_1^N.
\end{gather*}

From $(\ref{expr1})$ and the estimates above, we have
\begin{gather}\label{estimateLambda}
\Lambda^{N+1}_N(\lambda, \mu) \leq const\cdot\frac{\lambda_1 - \lambda_2 + 1}{\lambda_1^N(\lambda_1 - \lambda_2+1)}= const\cdot\lambda_1^{-N}
\end{gather}
Since $\lambda\rightarrow\infty$ is equivalent to $\lambda_1\rightarrow\infty$, it follows that $\Lambda^{N+1}_N(\lambda, \mu) \xrightarrow{\lambda\rightarrow\infty} 0$.
\end{proof}

\begin{prop}\label{pro2}
For each $N \geq 1$, the Markov kernels $\Lambda^{\infty}_N : \Omega_{\infty}\dashrightarrow\GTp_N$ defined in Section $\ref{infinitekernel}$ are Feller kernels.
\end{prop}
\begin{proof}
The argument in this proof was proposed by Grigori Olshanski.

Like in the proof of Proposition $\ref{pro1}$, the only difficult part is to show that $\Lambda^{\infty}_N\delta_{\lambda}\in C_0(\Omega_{\infty})$, for all $\lambda\in\GTp_N$. Equivalently, we prove
\begin{equation*}
\Lambda^{\infty}_N\delta_{\lambda}(\omega) = \Lambda^{\infty}_N(\omega, \lambda) \longrightarrow 0, \textrm{ as } \omega\longrightarrow\infty, \textrm{ for any } \lambda\in\GTp_N.
\end{equation*}

This is proved in two steps. Our proof follows an idea suggested by G. I. Olshanski.\\

\textbf{Step 1.} We prove that
\begin{gather}
|\Psi(x; \omega)| \xrightarrow{\omega\rightarrow\infty} 0 \textrm{ uniformly on $I_{\epsilon}\myeq\{x\in [-1, 1] : \Re(x) < 1 - \epsilon\}$,}\label{step1}\\
\textrm{for any $\epsilon > 0$ sufficiently small.}\nonumber
\end{gather}

We can write $x = (z + z^{-1})/2$, for some complex number $z\in\C$, $|z| = 1$, the claim transforms into:
\begin{gather}
\left|\Psi\left(\frac{z+z^{-1}}{2}; \omega\right)\right| \xrightarrow{\omega\rightarrow\infty} 0 \textrm{ uniformly on $\TT_{\epsilon} \myeq \{z\in\C : |z| = 1, \Re(z) < 1 - \epsilon  \}$,}\label{step1rewritten}\\
\textrm{for any $\epsilon > 0$ sufficiently small.}\nonumber
\end{gather}

From $(\ref{psifunction})$, cf. \cite[(1.12)]{OO1}, we have
\begin{equation}\label{psirewritten}
\left|\Psi\left(\frac{z+z^{-1}}{2}; \omega\right)\right| = |e^{\frac{\gamma}{2}(z + z^{-1} - 2)}|\cdot\prod_{i=1}^{\infty}{\frac{|1 + \frac{\beta_i}{2}(z-1)||1 + \frac{\beta_i}{2}(z^{-1} - 1)|}{|1 - \frac{\alpha_i}{2}(z-1)||1 - \frac{\alpha_i}{2}(z^{-1}-1)|}}.
\end{equation}

Let $\{\omega(k)\}_{k\in\N}$, $\omega(k) = (\alpha(k), \beta(k), \delta(k))\in\Omega_{\infty}$, be a sequence of points that converges to $\infty$. Let also $\gamma(k)$ be defined from $\omega(k)$ as in $(\ref{gammadef})$.

If a subsequence $(k_m)_{m=1, 2, \ldots}$ is such that $\lim_{m\rightarrow\infty}{\alpha_1(k_m)} = \infty$, then evidently the corresponding factors $|1-\frac{\alpha_1}{2}(z-1)|^{-1} |1-\frac{\alpha_1}{2}(z^{-1}-1)|^{-1}$ converge to $0$ uniformly on $z\in\TT_{\epsilon}$. Therefore, we can assume that there exists a constant $A>0$ such that $\sup_{k\in\N}{\alpha_1(k)} < A$. Let us now make some estimates on the factors that appear in $(\ref{psirewritten})$.

\begin{itemize}
    \item $|e^{\frac{\gamma}{2}(z + z^{-1} - 2)}| = \exp{(-\gamma(1 - \Re(z)))}$.
    \item $|1 - \frac{\alpha_i}{2}(z-1)|^{-1} = \left(1 + \frac{\alpha_i(2+\alpha_i)}{2}(1 - \Re(z))\right)^{-1} \leq (1 + \alpha_i(1 - \Re(z)))^{-1} \leq \exp(-const\cdot\alpha_i(1 - \Re(z)))$, where $const > 0$ is a constant depending on $A>0$.
    \item Since $\Re(z^{-1}) = \Re(z)$, the estimation above  yields also $|1 - \frac{\alpha_i}{2}(z^{-1}-1)|^{-1} \leq \exp(-const\cdot\alpha_i(1 - \Re(z)))$, for the same positive constant.
    \item $|1 + \frac{\beta_i}{2}(z-1)| = 1 - \frac{\beta_i(2 - \beta_i)}{2}(1 - \Re(z)) \leq 1 - \frac{\beta_i}{2}(1 - \Re(z)) \leq \exp(-\frac{\beta_i}{2}(1 - \Re(z)))$.
    \item From the point above, since $\Re(z^{-1}) = \Re(z)$, $|1 + \frac{\beta_i}{2}(z^{-1} - 1)| \leq \exp(-\frac{\beta_i}{2}(1 - \Re(z)))$.
\end{itemize}
The assumption $\omega(k)\rightarrow\infty$ is equivalent to $\delta(k) = \gamma(k) + \sum_{i=1}^{\infty}{(\alpha_i(k) + \beta_i(k))} \xrightarrow{k\rightarrow\infty} \infty$. Therefore, for any $z\in\TT_{\epsilon}$, one of the factors in $(\ref{psirewritten})$ converges to $0$ as $\omega\longrightarrow\infty$. We conclude that $\left|\Psi\left(\frac{z+z^{-1}}{2}; \omega\right)\right| \xrightarrow{\omega\rightarrow\infty} 0$, uniformly on $z\in\TT_{\epsilon}$, for any small $\epsilon > 0$.\\

\textbf{Step 2.} We complete the proof of the Proposition for general $N\geq 1$.

By definition, $\prod_{i=1}^N{\Psi(x_i; \omega)} = \sum_{\lambda\in\GTp_N}{\Lambda^{\infty}_N(\omega, \lambda)\frac{\mathfrak{P}_{\lambda}(x_1, \ldots, x_N|a, b)}{\mathfrak{P}_{\lambda}(1^N|a, b)}}$ is a basis decomposition in the Hilbert space $H_N$ of symmetric functions on $[-1, 1]^N$ which are square integrable with respect to the measure with density $\Delta_N(x)^2\prod_{i=1}^N{(1-x_i)^a(1+x_i)^b}\prod_{i=1}^N{dx_i}$, with respect to the Lebesgue measure.

The elements $\{\mathfrak{P}_{\lambda}(\cdot | a, b)\}_{\lambda\in\GTp}$ form an orthogonal basis of $H_N$, so $\Lambda^{\infty}_N(\omega, \lambda)\xrightarrow{\omega\rightarrow\infty}0$ is equivalent to the convergence $\prod_{i=1}^N{\Psi(x_i; \omega(k))}\xrightarrow{weakly} 0$ in $H_N$, for any sequence $\omega(k)\rightarrow \infty$ in $\Omega_{\infty}$. 
We actually prove strong convergence, rather than weak convergence. By definition of the norm on $H_N$,
\begin{equation}\label{eqn:squarednorm}
\| \Psi(x_1; \omega(k))\cdots\Psi(x_N; \omega(k))\|_H^2 = \int_{[-1, 1]^N}{\prod_{i=1}^N{|\Psi(x_i; \omega(k))|^2}\Delta_N(x)^2\prod_{i=1}^N{\left((1-x_i)^a(1+x_i)^b dx_i\right)}}
\end{equation}

The estimates of step 1 show that $\prod_{i=1}^N{|\Psi(x_i; \omega(k))|} \xrightarrow{k\rightarrow\infty} 0$ uniformly on $(x_1, \ldots, x_N)\in[-1, 1-\epsilon]^N$, for any sufficiently small $\epsilon > 0$. By using additionally the uniform boundedness of the integrand with respect to $\omega$ (recall that $|\Psi(x; \omega)| \leq 1$ for all $x\in[-1, 1]$ and $\omega\in\Omega_{\infty}$), the squared norm $(\ref{eqn:squarednorm})$ converges to $0$ as $k\rightarrow\infty$, and the desired conclusion follows.
\end{proof}

\section{Feller Semigroups on the BC Branching Graph}\label{semigroupssection}

In this section, we construct Feller semigroups $(P_N(t))_{t\geq 0}$ on the discrete spaces $\GTp_N$.

\subsection{Regular jump homogeneous Markov chains}

We recall here some useful theory on the transition functions of regular jump homogeneous Markov chains (HMC) and set our terminology. For a more complete account on regular jump HMCs, see \cite[Ch. 1-3]{A}, \cite[Ch. 8-9]{B}.

Let $\X$ be an infinite countable space. Given a regular jump HMC $\{X(t)\}_{t\geq 0}$ on $\X$, let $(P(t))_{t\geq 0}$ be its transition function, i.e.,
\begin{equation*}
P(t; i, j) = Prob\{X(t) = j | X(0) = i\}, \hspace{.1in}t\geq 0,\ i, j\in\X.
\end{equation*}
Then $(P(t))_{t\geq 0}$ is a Markov semigroup, i.e., $P(0) = I$ and
\begin{equation}\label{eqn:chapmankolmogorov}
P(t + s) = P(t)P(s) \ \forall t, s \geq 0.
\end{equation}
The equations $(\ref{eqn:chapmankolmogorov})$ above are called the \textit{Chapman-Kolmogorov equations}. Moreover, the Markov semigroup $(P(t))_{t\geq 0}$ is \textit{standard}, meaning that we have the following entry-wise limit
\begin{equation*}
\lim_{t\rightarrow 0^+}{P(t)} = I.
\end{equation*}
In this section, we assume all Markov semigroups are standard. The entry-wise limit
\begin{equation*}
    Q = \lim_{h\downarrow 0}{\frac{P(h) - I}{h}}
\end{equation*}
of matrices exists and $Q = [q_{i, j}]_{i, j\in\X}$ is called the \textit{$q$-matrix of the semigroup $(P(t))_{t\geq 0}$}. We also have $q_{i, j} \in [0, \infty)$ whenever $i\neq j$ and $q_{i, i} \in [-\infty, 0]$. We write $q_i \myeq -q_{i, i}\in [0, \infty]$ for all $i\in\X$. It can be shown that $Q$ is \textit{stable} and \textit{conservative}, meaning that
\begin{gather}
q_i < \infty, \ \forall i\in\X,\label{positivecondition}\\
q_i = \sum_{j\in\X, j\neq i}{q_{i, j}}, \ \forall i\in\X.\label{sumzerocondition}
\end{gather}

Importantly, the semigroup satisfies Kolmogorov's backward and forward differential equations
\begin{eqnarray*}
P'(t) &=& QP(t),\\
P'(t) &=& P(t)Q.
\end{eqnarray*}

A matrix $Q$ with $q_{i, j} \geq 0$ for all $i\neq j$, and moreover satisfying $(\ref{positivecondition})$ and $(\ref{sumzerocondition})$ will be called a \textit{matrix of transition rates}.
It is a natural question to ask whether for an $\X\times\X$ matrix of transition rates $Q$, there exists a regular jump HMC on $\X$ having $Q$ as the $q$-matrix of its transition function $(P(t))_{t\geq 0}$. If so, we can also ask for conditions on $Q$ that guarantee the uniqueness of $(P(t))_{t\geq 0}$. For simplicity in the discussion to follow, we always assume that $Q$ is \textit{essential}, i.e., $q_i > 0$ for all $i\in\X$.

To answer the questions above we begin by constructing, from $Q$, a family of $\X\times\X$ nonnegative matrices $\overline{P}(t)$, $t\geq 0$, whose row sums are at most $1$ and such that they satisfy the Chapman-Kolmogorov equations and Kolmogorov's backward equation. Then in Proposition $\ref{uniquenessprop}$ below, we state that if all row sums of the matrices $\overline{P}(t)$ are exactly $1$, then $(\overline{P}(t))_{t\geq 0}$ is the unique Markov semigroup with $q$-matrix $Q$, and moreover there exist regular jump HMCs with the transition function $(\overline{P}(t))_{t\geq 0}$.

For each $n\geq 0$, define the $\X\times\X$ matrices $P^{[n]}(t)$, $n\in\N$, $t\geq 0,$ via the recurrence
\begin{equation}\label{recurrencekol}
P^{[n]}(t; i, j) =
    \begin{cases}
        e^{-q_it}\mathbf{1}_{\{i = j\}} & \textrm{if } n=0,\\
        \int_0^t{e^{-q_is}\sum_{\substack{k\in\X \\ k\neq i}}{q_{i, k}P^{[n-1]}(t-s; k, j)}ds} & \textrm{if } n\geq 1,
    \end{cases}
\end{equation}
for any $t\geq 0$ and $i, j\in\X$. Define also the entry-wise sum of matrices
\begin{equation}\label{defn:summatrices}
\overline{P}(t) \myeq \sum_{n=0}^{\infty}{P^{[n]}(t)}.
\end{equation}

All matrices $\overline{P}(t)$, $t\geq 0$, are \textit{substochastic}, i.e., $\overline{P}(t; i, j)\geq 0$, $i, j\in\X$, and $\sum_{j\in\X}{\overline{P}(t; i, j)} \leq 1$, $i\in E$. They satisfy $\overline{P}(0) = I$, the Chapman-Kolmogorov equations, $\overline{P}(t)\overline{P}(s) = \overline{P}(t+s)$, and Kolmogorov's backward equation $\frac{d}{dt}\overline{P}(t) = Q\overline{P}(t)$.

\begin{df}\label{defn:nonexplosive}
For a matrix of transition rates $Q$, define the semigroup $(\overline{P}(t))_{t \geq 0}$ via equations $(\ref{recurrencekol})$ and $(\ref{defn:summatrices})$ above. If all the matrices $\overline{P}(t)$, $t \geq 0$, are stochastic, then we say that $Q$ is \textit{non-explosive} or \textit{regular}.
\end{df}

Intuitively, the idea of non-explosion is the following. Starting from a matrix of transition rates $Q$, it is not difficult to construct a (not necessarily regular) jump HMC $\{X(t)\}_{t\geq 0}$ with transition function $(\overline{P}(t))_{t\geq 0}$ and $Q$ as its associated $q$-matrix, see \cite[Ch. 8]{B}. From $(\ref{recurrencekol})$, one can argue by induction that $P^{[n]}(t; i, j)$ is the probability that the process moves from $i$ to $j$ in exactly $n$ jumps during $[0, t]$.  Then for any $t\geq 0$, $i, j\in E$, $\overline{P}(t; i, j)$ is the probability that the process moves from $i$ to $j$ during $[0, t]$, after a \textit{finite} number of jumps. Thus if $Q$ was non-explosive, we can hope that $\{X(t)\}_{t\geq 0}$ is a regular jump HMC and therefore its transition function $(\overline{P}(t))_{t\geq 0}$ satisfies Kolmogorov's backward and forward equations. This is indeed true, and there is the additional fact that $(\overline{P}(t))_{t\geq 0}$ is uniquely attached to $Q$. We summarize our discussion in the next proposition.

\begin{prop}\label{uniquenessprop}
If $Q$ is a non-explosive matrix of transition rates on $\X$, then the semigroup $(\overline{P}(t))_{t\geq 0}$ constructed above is a standard Markov semigroup, and it is the unique solution to Kolmogorov's backward equation (or to Kolmogorov's forward equation). Moreover, $(\overline{P}(t))_{t\geq 0}$ is the unique (standard) Markov semigroup with $q$-matrix $Q$. Finally, for any probability measure on $\X$, there is a regular jump HMC with such measure as its initial distribution and with the transition function $(\overline{P}(t))_{t\geq 0}$.
\end{prop}

\subsection{General setup: a Doob $h$-transformation}\label{generalsetup}

Consider a countable state space
\begin{equation*}
E_1 = E \myeq \{e_1 < e_2 < e_3 < \ldots\} \subset\R_{\geq 0}
\end{equation*}
with no accumulation points (think of $E$ being equal to $\Zp$ or $\Zp^{\epsilon}$). We identify $E$ with $\Zp$ for convenience in notation. We construct continuous-time birth-and-death processes on $E_1 = E$ and regular jump HMCs on the spaces
\begin{equation*}
E_N \myeq \{\mathbf{x} = (x_1, \ldots, x_N)\in E^N : x_1 > x_2 > \ldots > x_N\}, \hspace{.2in} N>1,
\end{equation*}
via a \textit{Doob $h$-transformation}. Observe that $E_N$ can be identified with $\Omega_N$ or $\Omega_N^{\epsilon}$, for $N\geq 1$, if $E$ is identified with $\Zp$ or $\Zp^{\epsilon}$, respectively. Instead of focusing on the construction of the stochastic processes themselves, we focus on the construction of Feller semigroups on $E$ and $E_N$, see Remark $\ref{processesfeller}$. Propositions $\ref{generalcriterion}$ and $\ref{semigroupN}$ are the main general statements we prove. Their specializations to Propositions $\ref{fellerspecial1}$ and $\ref{fellerspecialN}$, will be the relevant ones to our setting.

\subsubsection{\textbf{Non-explosive birth-and-death matrices on E}}

A \textit{continuous-time birth-and-death process} is a regular jump HMC on $E$ such that its $q$-matrix $Q$ is tridiagonal, i.e., the only (possibly) nonzero entries of $Q$ are of the form $q_{x, x+1}, q_{x, x}$ and $q_{x, x-1}$. A stochastic process having $Q$ as the $q$-matrix of its transition function can only jump between neighboring vertices. We shall consider only matrices $Q$, or associated birth-and-death processes, such that the entries $\{q_{x, x+1}\}_{x\geq 0}$ and $\{q_{x, x-1}\}_{x\geq 1}$ are all strictly positive.

We are interested in finding sufficient conditions for non-explosiveness of $Q$.
Proposition $\ref{uniquenessprop}$ would then guarantee the uniqueness of a Markov semigroup $(P_1(t))_{t\geq 0}$ whose $q$-matrix is $Q$, as well as the existence of continuous-time birth-and-death processes with $(P_1(t))_{t\geq 0}$ as its transition function.

We call a matrix $Q$ of format $E\times E$ a \textit{birth-and-death matrix} if it has the form
\begin{gather}\label{bdmatrix}
Q = \begin{bmatrix}
    -\beta_0 & \beta_0 & 0 & 0 &  \\
    \delta_1 & -(\beta_1+\delta_1) & \beta_1 & 0 & \vdots \\
    0 & \delta_2 & -(\beta_2+\delta_2) & \beta_2 & \vdots \\
    0 & 0 & \delta_3 & -(\beta_3+\delta_3) &  \\
     & \cdots & \cdots &   & \ddots
\end{bmatrix}
\end{gather}
and where the \textit{birth rates} $\{\beta_n\}_{n\geq 0}$ and \textit{death rates} $\{\delta_n\}_{n\geq 1}$ are all strictly positive. Occasionally we write $\delta_0 \myeq 0$.
There is an important theorem which gives necessary and sufficient conditions for non-explosiveness of $Q$. It is called \textit{Reuter's criterion for birth-and-death generators}, see \cite[Ch. 8, Thm. 4.5]{B}. In order to state it, we need first to define the sequence of \textit{potential coefficients} $\{\pi_n\}_{n\geq 0}$ by
\begin{equation}\label{potentialcoeffs}
\pi_n =
    \begin{cases}
        1 & \textrm{if } n=0\\
        \frac{\beta_0\beta_1\cdots\beta_{n-1}}{\delta_1\delta_2\cdots\delta_n} & \textrm{if }n\geq 1.
    \end{cases}
\end{equation}

We only need the following obvious implication of Reuter's critetion: the birth-and-death matrix $Q$ is non-explosive if
\begin{equation*}
\sum_n{\frac{1}{\beta_n\pi_n}} = \infty.
\end{equation*}

From now on, we assume that $Q$ is a non-explosive birth-and-death matrix. From Proposition $\ref{uniquenessprop}$, there exists a unique Markov semigroup $(P_1(t))_{t\geq 0}$ with $q$-matrix $Q$. A sufficient condition that guarantees the Feller property of $(P_1(t))_{t\geq 0}$ is given in \cite[Ch. 1, Thm. 5.7]{A}. For the case when its $q$-matrix $Q$ is a birth-and-death matrix, it states that $(P_1(t))_{t\geq 0}$ is a Feller semigroup provided that for any $\lambda>0$, the equation $\mathbf{x}^T(\lambda I - Q) = 0$ has a unique solution with $\sum_i{|x_i|} < \infty$: the trivial solution $\mathbf{x} = (0, 0, \ldots)^T$. Without much difficulty, one can show that if a nontrivial solution existed, then there would also exist a nontrivial solution whose entries are all nonnegative (switch the signs to the solution $\boldx$, if necessary). According to \cite[Ch. 3, Thm. 2.3]{A}, the equation $\mathbf{x}^T(\lambda I - Q) = 0$, $\sum_i{|x_i|} < \infty$, $x_i \geq 0$, admits only the trivial solution provided that
\begin{equation*}
\sum_n{\frac{1}{\delta_n}} + \sum_n{\frac{1}{\delta_n\pi_n}\sum_{k=n+1}^{\infty}{\pi_k}} = \infty.
\end{equation*}

From the discussion above, the following proposition is obtained.

\begin{prop}\label{generalcriterion}
Assume that $Q$ is a birth-and-death matrix as in $(\ref{bdmatrix})$. If both sums
\begin{gather}
\sum_n{\frac{1}{\beta_n\pi_n}}\label{sum1}\\
\sum_n{\frac{1}{\delta_n}} + \sum_n{\frac{1}{\delta_n\pi_n}\sum_{k=n+1}^{\infty}{\pi_k}}\label{sum2}
\end{gather}
diverge, then $Q$ is non-explosive and the unique Markov semigroup $(P_1(t))_{t\geq 0}$ with $q$-matrix $Q$ is a Feller semigroup.
\end{prop}

\subsubsection{\textbf{Non-explosive matrices of transition rates on $\{E_N\}_{N>1}$}}

We describe a general recipe for constructing non-explosive $E_N\times E_N$ matrices of transition rates $Q^{(N)}$ and their associated Markov semigroups. Assume we are given the transition function $(P_1(t))_{t\geq 0}$ of a regular jump HMC on $E$ and the $q$-matrix $Q$ of $(P_1(t))_{t\geq 0}$. In particular, we know $Q$ is non-explosive. Moreover we assume that $Q$ satisfies the very special Assumption $\ref{operatorassumption}$ below; before stating it, we need some notation.

The $q$-matrix $Q = [q_{x, y}]_{x, y\in E}$ has an associated difference operator $\DD$ that acts on functions of a single variable $x\in E$ as follows:
\begin{gather}\label{defD}
(\DD f)(x) \myeq \sum_{y\in E}{q_{x, y}f(y)}, \hspace{.1in} x\in E.
\end{gather}

In general, of course, $\DD$ is not well-defined on the whole space of single-variable functions, but only on those functions for which the sum on the right-hand side of $(\ref{defD})$ converges. For $i = 1, 2, \ldots, N$, let $\DD^{[i]}$ be the operator that acts on functions $f(x_1, \ldots, x_N)$, $x_i\in E$, of $N$ variables as $\DD$ would act on $x_i$, by treating all other variables $\{x_j\}_{j\neq i}$ as constants.

As in most constructions involving a Doob $h$-transformation, the Vandermonde determinant plays a role:
\begin{eqnarray*}
&\Delta_N : E^N \rightarrow \R_+\\
&\Delta_N(x_1, \ldots, x_N) = \prod_{i<j}{(x_i - x_j)}.
\end{eqnarray*}
The function $\Delta_N$ is \textit{antisymmetric}, i.e., $\Delta_N(x_{\sigma(1)}, x_{\sigma(2)}, \ldots, x_{\sigma(N)}) = sgn(\sigma)\cdot \Delta_N(x_1, x_2, \ldots, x_N)$, for any permutation $\sigma\in S_N$. Moreover, it is positive on $E_N$, i.e., $\Delta_N(x_1, x_2, \ldots, x_N) > 0$ whenever $x_1 > x_2 > \ldots > x_N$.

In the rest of this section, we operate under the following assumption on $Q$.

\begin{assum}\label{operatorassumption}
\normalfont
\begin{itemize}
    \item The operator $\DD$, associated to $Q$ as shown in $(\ref{defD})$, is well-defined for all polynomial functions $f$ on $E$. Moreover $\DD$ stabilizes the polynomial functions of degree $\leq m$, for any $m\in\Zp$. In other words, if we let $m\in\Zp$ and $f(x) = x^m$, $x\in E$, then $(\DD f)(x)$ is a polynomial on $x$ of degree $\leq m$.
    \item The equality
    \begin{equation}\label{heigen}
    \mathcal{D}^{\free}\Delta_N = c_N\Delta_N,
    \end{equation}
    holds for some constant $c_N>0$, where $\DD^{\free}$ is the difference operator
    \begin{equation}\label{Dfree}
    \DD^{\free} := \DD^{[1]} + \ldots + \DD^{[N]}.
    \end{equation}
\end{itemize}
\end{assum}

For any $\boldx = (x_1 > \ldots > x_N), \boldy = (y_1 > \ldots > y_N) \in E_N$, define
\begin{equation}\label{Qndef}
q^{(N)}_{\boldx, \boldy} = Q^{(N)}(\boldx, \boldy) \myeq \frac{\Delta_N(\boldy)}{\Delta_N(\boldx)}\left( q_{x_1, y_1}\mathbf{1}_{\{x_i = y_i, i\neq 1\}} +  \ldots + q_{x_N, y_N}\mathbf{1}_{\{x_i = y_i, i\neq N\}} \right) - c_N\mathbf{1}_{\{\boldx = \boldy\}}.
\end{equation}
As usual, we denote $q^{(N)}_{\boldx} \myeq -q^{(N)}_{\boldx, \boldx}$. The matrix $Q^{(N)}$ of format $E_N\times E_N$ has an associated difference operator $\DD^{(N)}$ that acts on functions of a single variable $\boldx\in E_N$ as follows:
\begin{equation*}
    (\mathcal{D}^{(N)}f)(\boldx) \myeq \sum_{\boldy\in E_N}{q^{(N)}_{\boldx, \boldy}f(\boldy)}, \textrm{ for all }\boldx\in E_N.
\end{equation*}
Under Assumption $\ref{operatorassumption}$, $\DD^{(N)}$ is well-defined on polynomial functions on $N$ variables, and it stabilizes polynomial functions of degree $\leq m$ in each variable, for any $m\in\Zp$.
The reason why definition $(\ref{Qndef})$ is vital can be traced to the following lemma and its corollary.
\begin{lem}\label{Dnoperator}
\begin{equation*}
\mathcal{D}^{(N)} = \left(\frac{1}{\Delta_N}\circ \mathcal{D}^{\free}\circ \Delta_N\right) - c_N.
\end{equation*}
\end{lem}
\begin{proof}
The proof is a simple calculation:
\begin{gather*}
(\DD^{(N)}f)(\boldx) = \sum_{\boldy\in E_N}{q^{(N)}_{\boldx, \boldy}f(\boldy)}=\\
= \frac{1}{\Delta_N(\boldx)}\sum_{\boldy\in E_N}\left(\Delta_N(\boldy)(q_{x_1, y_1}\mathbf{1}_{\{x_i = y_i, i\neq 1\}} + \ldots + q_{x_N, y_N}\mathbf{1}_{\{x_i \neq y_i, i\neq N\}}) - \right.\\
\left. - c_N\Delta_N(\boldx)\mathbf{1}_{\{\boldx = \boldy\}} \right)f(\boldy)\\
= \frac{1}{\Delta_N(\boldx)}\left\{ \sum_{i=1}^N{\left( \sum_{y_i\in E}{q_{x_i, y_i}\Delta_N(\boldy)f(\boldy)}\big{|}_{\substack{y_j = x_j \\ j\neq i}} \right)} - c_N \Delta_N(\boldx)f(\boldx) \right\}=\\
= \frac{1}{\Delta_N(\boldx)}\sum_{i=1}^N{\DD^{[i]}(\Delta_Nf)(\boldx)} - c_N f(\boldx).
\end{gather*}
The result follows from the definition $(\ref{Dfree})$ of $\DD^{\free}$.
\end{proof}
\begin{cor}
$Q^{(N)}$ is a matrix of transition rates.
\end{cor}
\begin{proof}
From definition $(\ref{Qndef})$, it is clear that $q^{(N)}_{\boldx, \boldy}\geq 0$ whenever $\boldx \neq \boldy$. Moreover $q^{(N)}_{\boldx, \boldx}\neq -\infty$ is evident. The condition $q^{(N)}_{\boldx, \boldx}\leq 0$ will follow once we show the only remaining fact that $Q^{(N)}$ has zero row sums.

Applying Lemma $\ref{Dnoperator}$ to $f \equiv 1$ and the second item of Assumption $\ref{operatorassumption}$, we have
\begin{equation*}
\sum_{\boldy\in E_N}{q^{(N)}_{\mathbf{x}, \mathbf{y}}} = (\DD^{(N)}1)(\mathbf{x}) = \frac{1}{\Delta_N(\mathbf{x})}(\DD^{\free}\Delta_N)(\mathbf{x}) - c_N = 0.
\end{equation*}
\end{proof}

\begin{prop}\label{semigroupN}
The matrix $Q^{(N)} = [q^{(N)}_{\boldx, \boldy}]_{\boldx, \boldy\in E_N}$ of transition rates is non-explosive. The corresponding Markov semigroup $(P_N(t))_{t\geq 0}$ is given by
\begin{equation}\label{semiN}
P_N(t; \boldx, \boldy) = e^{-c_N t}\cdot\frac{\Delta_N(\boldy)}{\Delta_N(\boldx)}\det_{1\leq i, j\leq n}[ P_1(t; x_i, y_j) ], \hspace{.2in}\boldx, \boldy \in E_N.
\end{equation}
\end{prop}
\begin{proof}
For any $n\in\Zp$, let $\widetilde{P}_N^{[n]}(t; \mathbf{x}, \mathbf{y})$ be the probability that a process with $N$ independent particles, each moving as a regular jump HMC with semigroup $(P_1(t))_{t\geq 0}$, begin at positions $x_1 > \ldots > x_N$ at time $0$, finish at positions $y_1 > \ldots > y_N$ at time $t$ after $n$ steps, given that the particles are conditioned on not intersecting during $[0, t]$. We also let
\begin{equation*}
\widetilde{P}_N(t; \mathbf{x}, \mathbf{y}) = \sum_{n=0}^{\infty}{\widetilde{P}_N^{[n]}(t; \mathbf{x}, \mathbf{y})}, \ t\geq 0, \ \mathbf{x}, \mathbf{y}\in E_N
\end{equation*}
be the probability that the conditional process above begins at time $0$ with particles at $x_1 > x_2 > \ldots > x_N$, makes a \textit{finite number of jumps in $[0, t]$} and ends at time $t$ with particles at $y_1 > y_2 > \ldots > y_N$. Since each of the $N$ particles in the process moves as a regular jump HMC with the non-explosive $q$-matrix $Q$, then $\widetilde{P}_N(t; \mathbf{x}, \mathbf{y})$ is the probability that the conditional process above begins with particles $x_1 > x_2 > \ldots > x_N$ at time $0$ and finishes with particles $y_1 > y_2 > \ldots > y_N$ at time $t$ (we have removed the condition that the number of steps during $[0, t]$ is finite, because that follows from the non-explosiveness of $Q$).

For each $n\in\Zp$, $t\geq 0$, define the matrices
\begin{equation}\label{recurrencekol1}
P^{[n]}(t; \boldx, \boldy) =
    \begin{cases}
        e^{-q^{(N)}_{\boldx}t}\cdot\mathbf{1}_{\{\boldx = \boldy\}} & \textrm{if } n=0,\\
        \int_0^t{e^{-q^{(N)}_{\boldx}s}\cdot\sum_{\substack{\mathbf{z}\in E_N \\ \mathbf{z}\neq \boldx}}{q^{(N)}_{\boldx, \mathbf{z}}P^{[n-1]}(t-s; \mathbf{z}, \boldy)}ds} & \textrm{if } n\geq 1,
    \end{cases}
\end{equation}
and define also the entry-wise sum of matrices
\begin{equation*}
\overline{P}(t) = \sum_{n=0}^{\infty}{P^{[n]}(t)}, \hspace{.2in}t\geq 0.
\end{equation*}
To prove the non-explosion of $Q^{(N)}$, according to Definition $\ref{defn:nonexplosive}$, we need to show that all matrices $\overline{P}(t)$, $t\geq 0$, are stochastic, i.e.,
\begin{equation}\label{conditioncheck1}
\sum_{\boldy\in E_N}{\overline{P}(t; \boldx, \boldy)} = 1, \textrm{ for any }\boldx\in E_N,
\end{equation}
and for the last statement of the proposition, we need
\begin{equation}\label{conditioncheck2}
\overline{P}(t; \boldx, \boldy) = e^{-c_Nt}\frac{\Delta_N(\boldy)}{\Delta_N(\boldx)}\det_{i, j}[P_1(t; x_i, y_j)].
\end{equation}
It is actually easier to prove $(\ref{conditioncheck2})$ first. We claim
\begin{equation}\label{PNformula}
    \overline{P}_N(t; \mathbf{x}, \mathbf{y}) = e^{-c_Nt}\frac{\Delta_N(\mathbf{y})}{\Delta_N(\mathbf{x})}\widetilde{P}_N(t; \mathbf{x}, \mathbf{y}).
\end{equation}
The claim above, together with the formula of Karlin and McGregor $\widetilde{P}_N(t; \mathbf{x}, \mathbf{y}) = \det_{i, j}[ P_1(t; x_i, y_j) ]$, see \cite{KM}, proves $(\ref{conditioncheck2})$, so we concentrate on the claim $(\ref{PNformula})$. We shall prove in fact the stronger statement
\begin{equation}\label{PNnformula}
    P_N^{[n]}(t; \mathbf{x}, \mathbf{y}) = e^{-c_Nt}\frac{\Delta_N(\mathbf{y})}{\Delta_N(\mathbf{x})}\widetilde{P}_N^{[n]}(t; \mathbf{x}, \mathbf{y}), \hspace{.2in} n\geq 0.
\end{equation}
The proof is by induction on $n$. For the base case $n = 0$, observe that $\widetilde{P}_N^{[0]}(t; \boldx, \boldy)$ is zero if $\boldx\neq\boldy$ and if $\boldx = \boldy$, then $\widetilde{P}_N^{[0]}(t; \boldx, \boldy)$ is the probability that none of the $N$ independent particles at $x_1 > x_2 > \ldots > x_N$ moves during the interval $[0, t]$. Therefore
\begin{equation*}
\widetilde{P}_N^{[0]}(t; \boldx, \boldy) = \mathbf{1}_{\{\boldx = \boldy\}}\prod_{i=1}^N{e^{-tq_{x_i}}},
\end{equation*}
which together with the formula $(\ref{recurrencekol1})$ for $P^{[0]}(t; \boldx, \boldy)$ and $q^{(N)}_{\boldx} = -q_{\boldx, \boldx} = q_{x_1} + \ldots + q_{x_N} + c_N$, proves $(\ref{PNnformula})$ for $n = 0$.

For the induction step, assume that $(\ref{PNnformula})$ holds for some $n-1\in\Zp$ and we aim to prove it for $n$. A typical one-step analysis on the process of $N$ independent particles shows
\begin{equation}
\widetilde{P}_N^{[n]}(t; \mathbf{x}, \mathbf{y}) = \int_0^t{\prod_{i=1}^N{e^{-sq_{x_i}}}\sum_{\substack{\mathbf{z}\in E_N \\\mathbf{z}\neq\mathbf{x}}}{S^{(N)}_{\mathbf{x}, \mathbf{z}}\widetilde{P}_N^{[n-1]}(t-s; \mathbf{z}, \mathbf{y})}ds},
\end{equation}
where $S^{(N)}_{\mathbf{x}, \mathbf{z}} \myeq \mathbf{1}_{\{x_j = z_j, j\neq 1\}}q_{x_1, z_1} + \ldots + \mathbf{1}_{\{x_j = z_j, j\neq N\}}q_{x_N, z_N}$. Together with $(\ref{recurrencekol1})$ and the induction hypothesis, one easily proves $(\ref{PNnformula})$ for $n$, thus completing the induction step.

The remaining step is to prove $(\ref{conditioncheck1})$ which, thanks to the formula $(\ref{conditioncheck2})$ we already proved, is equivalent to
\begin{equation}\label{toprove1}
\sum_{\substack{y_1 > y_2 > \ldots > y_N\geq 0 \\ y_1, \ldots, y_N \in E}}{\det_{i, j}[P_1(t; x_i, y_j)]\Delta_N(\mathbf{y})} = e^{c_Nt}\Delta_N(\mathbf{x}), \textrm{ for all }\mathbf{x}\in E_N, t\geq 0.
\end{equation}

From the antisymmetry property of $\Delta_N$, the left-hand side of $(\ref{toprove1})$ can be written as
\begin{gather*}
\sum_{\substack{y_1 > y_2 > \ldots > y_N\geq 0 \\ y_1, \ldots, y_N\in E}}{\left(\sum_{\sigma\in S_N}{sgn(\sigma)\prod_{i=1}^N{P_1(t; x_i, y_{\sigma(i)})}}\right)\Delta_N(\mathbf{y})}\\
= \sum_{\substack{y_1 > y_2 > \ldots > y_N\geq 0 \\ y_1, \ldots, y_N\in E}}{\sum_{\sigma\in S_N}{\prod_{i=1}^N{P_1(t; x_i, y_{\sigma(i)})}\Delta_N(y_{\sigma(1)}, \ldots, y_{\sigma(N)})}}\\
= \sum_{y_1, \ldots, y_N\in E}{\prod_{i=1}^N{P_1(t; x_i, y_i)}\cdot \Delta_N(\mathbf{y})}.
\end{gather*}
We point out that in the last sum we added some terms indexed by tuples $\boldy = (y_1, \ldots, y_N)$ with $y_i = y_j$ for some $i\neq j$. No problem arises because $\Delta_N(\boldy)$ vanishes whenever $\boldy$ has two equal components.
The statement $(\ref{toprove1})$ that we need to prove becomes
\begin{equation}\label{toprove2}
\sum_{y_1, \ldots, y_N\in E}{\left( \prod_{i=1}^N{P_1(t; x_i, y_i)} \right) \Delta_N(\mathbf{y})} = e^{c_Nt}\Delta_N(\mathbf{x}).
\end{equation}
The identity $(\ref{toprove2})$ is a consequence of Assumption $\ref{operatorassumption}$. The formal derivation is somewhat tedious and left to Appendix $\ref{appendixB}$.
\end{proof}

As a consequence of Propositions $\ref{uniquenessprop}$ and $\ref{semigroupN}$, we obtain the following statement.

\begin{prop}\label{generalcriterionN}
If $Q$ is a non-explosive matrix of transition rates on $E$ satisfying Assumption $\ref{operatorassumption}$ and $Q^{(N)}$ is the $E_N\times E_N$ matrix defined in $(\ref{Qndef})$, then $Q^{(N)}$ is a non-explosive matrix of transition rates on $E_N$. There exists a unique Markov semigroup $(P_N(t))_{t\geq 0}$ with $q$-matrix $Q^{(N)}$, and is given by
\begin{equation*}
P_N(t; \boldx, \boldy) = e^{-c_N t}\cdot\frac{\Delta_N(\boldy)}{\Delta_N(\boldx)}\det_{1\leq i, j\leq n}[ P_1(t; x_i, y_j) ], \hspace{.2in}t\geq 0; \ \mathbf{x}, \mathbf{y}\in E_N,
\end{equation*}
where $(P_1(t))_{t\geq 0}$ is the unique Markov semigroup with $q$-matrix $Q$.
\end{prop}

\subsection{Specialization to the BC branching graph}\label{specialization}

\subsubsection{\textbf{Non-explosive birth-and-death matrices on $\GTp_1 = \Zp$}}

Our construction depends on two complex parameters $u, u'\in\C$.

\begin{df}\label{Hnod}
We define $\mathcal{H}^{\circ}$ as the set of pairs $(u, u')\in\C^2$ satisfying
\begin{eqnarray*}
(u - k)(u' - k) &>& 0 \textrm{, for any }k \in \Zp\\
(u + k + 2\epsilon)(u' + k + 2\epsilon) &>& 0 \textrm{, for any }k\in\N\\
u + u' + b &>& -1.
\end{eqnarray*}
Observe that if the pair $(u, u')$ satisfies either of the first two conditions above, then $u + u'\in\R$ and so the third condition makes sense.
\end{df}
\begin{rem}
Later in Section $\ref{zmeasuressubsection}$, we define the space $\mathcal{H}$ of pairs $(z, z')$ for which $z$-measures are defined. It turns out that $\mathcal{H}\subset\mathcal{H}^{\circ}$, so the construction to come in this section is slightly more general than needed for the main result of the paper.
\end{rem}

For any pair $(u, u')\in\mathcal{H}^{\circ}$, define the birth and death rates $\{\beta(x)\}_{x\in\Zp}$, $\{\delta(x)\}_{x\in\Zp}$, by
\begin{eqnarray}
\beta(x) &=& \frac{(x + 2\epsilon)(x + a + 1)(x - u)(x - u')}{2(x + \epsilon)(2x + 2\epsilon + 1)}, \hspace{.1in}x\in\Zp\label{Blabel}\\
\delta(x) &=& \frac{x(x + b)(x + u + 2\epsilon)(x + u' + 2\epsilon)}{2(x + \epsilon)(2x + 2\epsilon - 1)}, \hspace{.1in} x\in\N, \hspace{.2in}\delta(0) = 0.\label{Dlabel}
\end{eqnarray}
From the definition of the set $\mathcal{H}^{\circ}$, the rates $\{\beta(x)\}_{x\geq 0}$ and $\{\delta(x)\}_{x\geq 1}$ are all strictly positive. Define the birth-and-death matrix $R = [R(x, y)]_{x, y\in\Zp}$ by
\begin{gather}\label{R1expression}
R = \begin{bmatrix}
    -\beta(0) & \beta(0) & 0 & 0 &  \\
    \delta(1) & -\beta(1)-\delta(1) & \beta(1) & 0 & \vdots \\
    0 & \delta(2) & -\beta(2)-\delta(2) & \beta(2) & \vdots \\
    0 & 0 & \delta(3) & -\beta(3)-\delta(3) &  \\
     & \cdots & \cdots &   & \ddots
\end{bmatrix}.
\end{gather}

\begin{rem}
Below we define more general matrices of transition rates $R^{(N)}$ in $\GTp_N$, for $N>1$, reason why we denote $R$ by $R^{(1)}$ sometimes. If we want to be explicit about the pair $(u, u')\in\mathcal{H}^{\circ}$ used in the definition of $R$, we write instead $R_{u, u'}$.
\end{rem}

In order to apply Proposition $\ref{generalcriterion}$ in our case, we need to verify that the two sums shown there diverge. Clearly $(\ref{Blabel}), (\ref{Dlabel})$ give the following expression for the potential coefficients $\pi(n)$ in terms of Gamma functions:
\begin{equation*}
\pi(n) = \frac{\beta(0)\cdots\beta(n-1)}{\delta(1)\cdots\delta(n)} = \frac{const\cdot (n+\epsilon)\cdot\Gamma(n + 2\epsilon)\Gamma(n + a + 1)\Gamma(n - u)\Gamma(n - u')}{\Gamma(n + 1)\Gamma(n + b + 1)\Gamma(n + u + 2\epsilon + 1)\Gamma(n + u' + 2\epsilon + 1)},
\end{equation*}
where $const$ is a constant independent of $n$. Then
\begin{equation}\label{piasymptotics}
\pi(n) \sim n^{-2(u + u' + b) - 3}, \hspace{.2in} n\longrightarrow\infty.
\end{equation}
From $(\ref{Blabel})$, we have $\beta(n) \sim n^2$ as $n\rightarrow\infty$; thus $\frac{1}{\beta(n)\pi(n)} \sim n^{2(u+u'+b)+1}$. By the third condition of Definition $\ref{Hnod}$, $u+u'+b > -1$, therefore
\begin{equation*}
\sum_{n=0}^{\infty}{\frac{1}{\beta(n)\pi(n)}} = \infty.
\end{equation*}
From $(\ref{piasymptotics})$, it follows that the sum $\sum_k{\pi(k)}$ converges and moreover $\sum_{k = n+1}^{\infty}{\pi(k)} \geq c n^{-2(u+u'+b)-2}$, for some constant $c>0$. Also, from the definition $(\ref{Dlabel})$, $\frac{1}{\delta(n)\pi(n)}\sim n^{2(u+u'+b)+1}$, $n\rightarrow\infty$, so the term $\frac{1}{\delta(n)\pi(n)}\sum_{k=n+1}^{\infty}{\pi(k)}$ is of order $n^{-1}$, thus proving
\begin{equation*}
\sum_n{\frac{1}{\delta(n)\pi(n)}\sum_{k=n+1}^{\infty}{\pi(k)}} = \infty.
\end{equation*}

We have verified the two conditions on the birth-and-death matrix $R$ from Proposition $\ref{generalcriterion}$ and therefore we have proven the following.
\begin{prop}\label{fellerspecial1}
The birth-and-death matrix $R$ defined by $(\ref{Blabel})$, $(\ref{Dlabel})$ and $(\ref{R1expression})$ is non-explosive. It is the $q$-matrix of a unique Markov semigroup $(P_1(t))_{t\geq 0}$, which is a Feller semigroup.
\end{prop}

\subsubsection{\textbf{Non-explosive matrices of transition rates on $\GTp_N$}}\label{generalN}

Let $N > 1$ be an integer, $(u, u')\in\mathcal{H}^{\circ}$, and let $R = R_{u, u'}$ be the birth-and-death matrix defined above. We aim to define a matrix $[R^{(N)}(\lambda, \nu)]_{\lambda, \nu\in\GTp_N}$ of transition rates with rows and columns parametrized by $\GTp_N$. We shall use the general method of Section $\ref{generalsetup}$. In order to apply that general construction, we use  $E = \Zp^{\epsilon} = \{\epsilon^2, (1 + \epsilon)^2, (2 + \epsilon)^2, \ldots\}$ and $E_N = \Omega_N^{\epsilon}$, $N > 1$. Therefore we treat $R$ as a matrix with rows and columns parametrized by $\Zp^{\epsilon}$, even though we will use the notation $R(x, y)$, with $x, y\in\Zp$, for its entries. Similarly, we treat $R^{(N)}$ as a matrix with rows and columns parametrized by $\Omega_N^{\epsilon}$ even though we will use the notation $R^{(N)}(\lambda, \mu)$, with $\lambda, \mu\in\GTp_N$, for its entries.

The expression $(\ref{Qndef})$ specializes to
\begin{eqnarray}\label{defnRN}
R^{(N)}(\lambda, \nu) &=& \frac{\prod_{i < j}{(\hatn_i - \hatn_j)}}{\prod_{i < j}{(\hatl_i - \hatl_j)}}\left(R(l_1, n_1)\mathbf{1}_{\{l_i = n_i, i\neq 1\}} + R(l_2, n_2)\mathbf{1}_{\{l_i = n_i, i\neq 2\}}\right.\nonumber\\
&&\left. +\ldots + R(l_N, n_N)\mathbf{1}_{\{l_i = n_i, i\neq N\}}\right) - c_N\mathbf{1}_{\{\lambda = \nu\}}, \hspace{.2in}\lambda, \nu\in\GTp_N,
\end{eqnarray}
where
\begin{equation}\label{dN}
c_N = \frac{N(N - 1)(N - 2)}{3} - \frac{N(N - 1)}{2}(u + u' + b)
\end{equation}
and $(l, n)$, $(\hatl, \hatn)$ are associated to $(\lambda, \nu)$ as explained at the end of Section $\ref{posGT}$.

In order to obtain the benefits from Proposition $\ref{generalcriterionN}$, we need to verify the two conditions of Assumption $\ref{operatorassumption}$. For the first condition, let us look at the discrete difference operator $\DD$ associated to $R$. Naturally, $\mathcal{D}$ is defined as an operator on functions on $\Zp^{\epsilon} = \{(x + \epsilon)^2 : x\in\Zp\}$, given by
\begin{equation}\label{diffoperatorD}
(\mathcal{D}f)(\widehat{x}) = \sum_{y\in\Zp}{R(x, y)f(\widehat{y})} = \beta(x)(f(\widehat{x + 1}) - f(\widehat{x})) + \delta(x)(f(\widehat{x - 1}) - f(\widehat{x})), 
\end{equation}
where $\widehat{x} = (x+\epsilon)^2$, $x\in\Zp$.

\begin{lem}\label{operatorvalid}
The operator $\DD$ is well-defined for polynomial functions on $\Zp^{\epsilon}$ and it stabilizes them, i.e., if $f$ is a polynomial, then $(\mathcal{D}f)(\widehat{x})$ is a polynomial on $\widehat{x}$. Moreover, for any $m\in\Zp$, we have
\begin{equation*}
\mathcal{D}[(x + \epsilon)^{2m}] = a_m(x + \epsilon)^{2m} + \textrm{ lower degree terms in } (x + \epsilon)^2,
\end{equation*}
where
\begin{equation*}
a_m := m(m - 1) - m(u + u' + b).
\end{equation*}
Consequently, $\DD$ stabilizes the polynomial functions on $\Zp^{\epsilon}$ of degree $\leq m$, for any $m\in\Zp$.
\end{lem}
\begin{proof}
The proof is a simple calculation and is left to the reader. 
\end{proof}

To verify the second condition of Assumption $\ref{operatorassumption}$, notice that Lemma $\ref{operatorvalid}$ gives 
\begin{equation*}
\sum_{i=1}^N{\DD^{[i]}}\Delta_N(\mathbf{x}) = \left(\sum_{m=1}^{N-1}{a_m}\right)\Delta_N(\mathbf{x}),
\end{equation*}
where $a_m = m(m - 1) - m(u + u' + b)$. The evident equality $c_N = \sum_{m=1}^{N-1}{a_m}$, where $c_N$ is defined in $(\ref{dN})$, then shows
\begin{equation*}
    \DD^{\free}\Delta_N(\mathbf{x}) := \sum_{i=1}^N{\DD^{[i]}}\Delta_N(\mathbf{x}) = c_N\Delta_N(\mathbf{x}).
\end{equation*}

The specialization of Proposition $\ref{generalcriterionN}$ yields

\begin{prop}\label{fellerspecialN}
The matrix of transition rates $R^{(N)}$ defined in $(\ref{defnRN})$ is non-explosive. Therefore it is the $q$-matrix of a unique Markov semigroup on $\GTp_N$, which is a Feller semigroup and will be denoted by $(P_N(t))_{t\geq 0}$. It is explicitly given by
\begin{equation}\label{formsemigroup}
P_N(t; \lambda, \nu) = e^{-c_Nt}\cdot\frac{\Delta_N(\hatn)}{\Delta_N(\hatl)}\det_{i, j} [ P_1(t; l_i, n_j) ],
\end{equation}
where $c_N$ is as in $(\ref{dN})$, $(P_1(t))_{t\geq 0}$ is the unique Markov semigroup with $q$-matrix $R = R^{(1)}$, and $P_1(t; l, n)$, $l, n\in\Z$, are the entries of the stochastic matrix $P_1(t)$.
\end{prop}
\begin{proof}
Most statements above follow from Proposition $\ref{generalcriterionN}$. We are only left to show that $(P_N(t))_{t\geq 0}$ is Feller. The Feller property of $(P_1(t))_{t\geq 0}$ given by Proposition $\ref{fellerspecial1}$ states
\begin{equation}\label{feller1}
\lim_{l\rightarrow\infty}{P_1(t; l, n)} = 0, \hspace{.1in}\textrm{ for any }n\in\Zp.
\end{equation}
Evidently $\Delta_N(\hatl) \geq 1$, so $0 \leq P_N(t; \lambda, \nu) \leq const\cdot\det_{i, j}[P_1(t; l_i, n_j)]$. From $(\ref{feller1})$, and because $\lambda\longrightarrow\infty$ is equivalent to $l_1 \longrightarrow\infty$, we have $\det_{i, j}[P_1(t; l_i, n_j)] \xrightarrow{\lambda\rightarrow\infty} 0$. Hence
\begin{equation*}
\lim_{\lambda\rightarrow\infty}{P_N(t; \lambda, \nu)} = 0, \hspace{.1in}\textrm{ for any }\nu\in\GTp_N,
\end{equation*}
which is the Feller property for the semigroup $(P_N(t))_{t\geq 0}$ on the countable space $\GTp_N$.
\end{proof}

\begin{rem}
The matrix of transition rates $R^{(N)}$ can be described as follows: if $\lambda\neq\mu$, the nondiagonal entry $R^{(N)}(\lambda, \mu)$ is nonzero if and only if $\sum_i{|\lambda_i - \mu_i|} = 1$; in such case

\begin{equation}\label{Rexpression}
    R^{(N)}(\lambda, \mu) =
	\begin{cases}
                  \beta(l_k)\cdot\prod_{j\neq k}{\frac{\widehat{(l_k+1)} - \hatl_j}{\hatl_k - \hatl_j}} & \textrm{if } \mu_k = \lambda_k + 1,\\
                  \delta(l_k)\cdot\prod_{j\neq k}{\frac{\widehat{(l_k - 1)} - \hatl_j}{\hatl_k - \hatl_j}} & \textrm{if } \mu_k = \lambda_k - 1.
          \end{cases}
\end{equation}
Thus the generator of the Feller semigroup $(P_N(t))_{t\geq 0}$ is given by the following partial difference operator, cf \cite[(5.1)]{Ol4},
\begin{equation*}
\begin{gathered}
\left( D^{BC}_{u, u' a, b | N} f \right)(\hatl) = \sum_{k=1}^N \left\{ \prod_{j\neq k}{\frac{\widehat{(l_k+1)} - \hatl_j}{\hatl_k - \hatl_j}}\frac{(l_k + 2\epsilon)(l_k + a + 1)(l_k - u)(l_k - u')}{2(l_k + \epsilon)} \left( \frac{f(\widehat{l+\epsilon_k}) - f(\hatl)}{2l_k + 2\epsilon + 1} \right) \right.\\
\left. + \prod_{j\neq k}{\frac{\widehat{(l_k - 1)} - \hatl_j}{\hatl_k - \hatl_j}} \frac{(l_k + b)l_k(l_k + u + 2\epsilon)(l_k + u' + 2\epsilon)}{2(l_k + \epsilon)} \left( \frac{f(\widehat{l - \epsilon_k}) - f(\hatl)}{2l_k + 2\epsilon - 1} \right) \right\}.
\end{gathered}
\end{equation*}
The domain of the generator can be shown to be $\{f\in C_0(\Omega^{\epsilon}_N) : D^{BC}_{u, u', a, b|N} f \in C_0(\Omega_N^{\epsilon})\}$.
\end{rem}

\begin{rem}
If we want to be explicit about the pair $(u, u')\in\mathcal{H}^{\circ}$ used in the definition of $R^{(N)}$, we write instead $R^{(N)}_{u, u'}$.
\end{rem}

\section{The Master Relation}\label{commuting}

The goal of this section is to prove the coherence between the Markov semigroups from Section $\ref{specialization}$ and the Markov kernels from Section $\ref{finitekernels}$. The master relation $(\ref{commutativity})$ holds when we define the semigroups $(P_N(t))_{t\geq 0}$ from a pair $(u, u')\in\mathcal{H}^{\circ}$ that depends on $N$. For that reason, we introduce another pair $(z, z')$ of complex numbers that satisfy
\begin{eqnarray}
(z - k)(z' - k) &>& 0 \textrm{, for all }k \in \Z\nonumber\\
(z + k + 2\epsilon)(z' + k + 2\epsilon) &>& 0 \textrm{, for all }k\in\N,\label{zparamsnoname}\\
z + z' &>& - (1 + b).\nonumber
\end{eqnarray}
The constraints above are the same as those for pairs in $\mathcal{H}^{\circ}$, except the first: the condition here requires the inequality to hold just for $k\in\Z$, as opposed to for all $k\in\Zp$. Also, if $(z, z')$ satisfies all inequalities in $(\ref{zparamsnoname})$, then $(z+m, z'+m)\in\mathcal{H}^{\circ}$ for all $m\in\Zp$. 

For any $N \geq 1$, let $(P_N(t))_{t\geq 0}$ be the Feller semigroup defined in Section $\ref{specialization}$, with parameters $(u, u') = (z+N-1, z'+N-1)$. The main goal of this section is to prove
\begin{thm}\label{intertwiningthm}
With Markov kernels $\{\Lambda^{N+1}_N\}_{N \geq 1}$ as in Section $\ref{finitekernels}$, and semigroups $(P_N(t))_{N \geq 1}$ as above, the master relation
\begin{equation*}
P_{N+1}(t)\Lambda^{N+1}_N = \Lambda^{N+1}_NP_N(t)
\end{equation*}
 holds for all $N\geq 2$, $t\geq 0$.
\end{thm}

By virtue of the method of intertwiners, see Section $\ref{methodinter}$, Theorem $\ref{intertwiningthm}$ effectively shows the existence of Feller semigroups $(P_{\infty}(t))_{t\geq 0}$ depending on a pair $(z, z')\in\C^2$ (and also on the pair $(a, b)$) that satisfies $(\ref{zparamsnoname})$, and moreover
\begin{equation*}
P_{\infty}(t)\Lambda^{\infty}_N = \Lambda^{\infty}_NP_N(t), \hspace{.2in}N\geq 2, \ t\geq 0.
\end{equation*}
In Section $\ref{zmeasuressubsection}$, it will be shown that for pairs $(z, z')$ in a smaller set $\mathcal{H} \subset \mathcal{H}^{\circ}$, the semigroup $(P_{\infty}(t))_{t\geq 0}$ has the $z$-measure associated to $(z, z', a, b)$ as its unique invariant measure, thus concluding the proof of the main theorem of the paper.

\subsection{Infinitesimal master relation}

We prove a version of the master relation that does not involve the semigroups $(P_N(t))_{t\geq 0}$, but rather their $q$-matrices. Recall that we denote the $q$-matrix of $(P_N(t))_{t \geq 0}$ by $R^{(N)}_{u, u'}$ if we wish to be specific about the pair of parameters $(u, u')\in\mathcal{H}^{\circ}$ used to define it.
\begin{prop}\label{commratesprop}
Let $N\in\N$ and $(z, z')\in\C^2$ be a pair that satisfies $(\ref{zparamsnoname})$. If we set $u = z+N-1, u' = z'+N-1$, then $(u, u'), (u+1, u'+1)\in\mathcal{H}^{\circ}$ and the following relation holds
\begin{equation}\label{commrates1}
R^{(N+1)}_{u+1, u'+1}\Lambda^{N+1}_N = \Lambda^{N+1}_NR^{(N)}_{u, u'}.
\end{equation}
\end{prop}
It is clear from the definitions that $(u, u'), (u+1, u'+1)\in\mathcal{H}^{\circ}$ if $(z, z')\in\C^2$ satisfies the conditions in $(\ref{zparamsnoname})$, but the identity $(\ref{commrates1})$ is far from trivial. Let us begin with the following lemma.
\begin{lem}\label{easyeval}
Let $[A(\lambda, \nu)]$ be a $\GTp_{N+1}\times\GTp_N$ matrix with finitely many nonzero entries in each row. If for any symmetric polynomial $P(x_1, \ldots, x_N)$ with real coefficients and any $\lambda\in\GT_{N+1}^+$, we have
\begin{equation*}
\sum_{\nu\in\GTp_N}{A(\lambda, \nu)P(\hatn_1, \ldots, \hatn_N)} = 0,
\end{equation*}
then $A$ is the zero matrix.
\end{lem}
\begin{proof}
For any pairwise distinct $\nu^{(1)}, \ldots, \nu^{(k)}\in\GTp_N$, let $\hatn^{(1)}, \ldots, \hatn^{(k)}\in\Omega_N^{\epsilon}$ be defined by $\hatn^{(i)}_j = (\nu^{(i)}_j + N - j + \epsilon)^2 \ \forall 1\leq i\leq k, \ 1\leq j\leq N$. The key fact is that $\epsilon \geq 0$ implies that the orbits of $\hatn^{(1)}, \ldots, \hatn^{(k)}$ under the symmetric group action are pairwise distinct. Then there exists a polynomial $P$ such that $P(\hatn^{(1)}) \neq 0$ and $P(\hatn^{(2)}) = \dots P(\hatn^{(k)}) = 0$. The remaining details of the proof are left to the reader.
\end{proof}

Lemma $\ref{easyeval}$ paves the way to proving Proposition $\ref{commratesprop}$ indirectly. If we let the difference between the left and right-hand sides of $(\ref{commrates1})$ be the matrix $A$, of format $\GTp_{N+1}\times\GTp_N$, then we claim that $A$ has finitely many nonzero entries in each row. In fact, for any fixed $\lambda\in\GTp_{N+1}$, we have $R_{u+1, u'+1}^{(N+1)}(\lambda, \kappa) \neq 0$ if and only if $\kappa\in\GTp_{N+1}$ differs from $\lambda$ in at most one coordinate, see $(\ref{defnRN})$, and there are finitely many of those positive $N$-signatures $\kappa$. Similarly, for any fixed $\lambda\in\GTp_N$, there exists finitely many $\kappa\in\GTp_N$ such that $R^{(N)}_{u, u'}(\lambda, \kappa) \neq 0$.
On the other hand, for any fixed $\lambda\in\GTp_{N+1}$, we have $\Lambda^{N+1}_N(\lambda, \kappa) \neq 0$ if and only if $\lambda \succ_{BC}\kappa$, see Section $\ref{posGT}$, and there are finitely many of those positive $N$-signatures $\kappa$.

Therefore each product of matrices $R^{(N+1)}_{u+1, u'+1}\Lambda^{N+1}_N$ and $\Lambda^{N+1}_N R^{(N)}_{u, u'}$ has finitely many nonzero entries in each row, and therefore so does their difference, the matrix $A$.

From Lemma $\ref{easyeval}$, we could conclude that $A$ is the zero matrix if the operator defined by $A$ kills all functions $T_i(\hatn_1, \ldots, \hatn_N)$, where $\{T_i\}_i$ is a basis of the algebra of symmetric polynomials on $N$ variables. Thus we would like to introduce a choice of such a basis that makes the calculations feasible.

For any sequence $\mathbf{a} = (a_1, a_2, \ldots)$ and $k\in\Zp$, define
\begin{equation*}
(x | \mathbf{a})^k := (x - a_1)(x - a_2)\cdots(x - a_k), \textrm{ if }k \geq 1 \textrm{ and }(x|\mathbf{a})^0 := 1.
\end{equation*}
For any $\mu\in\GTp_N$, consider the constant
\begin{equation}\label{bNmu}
c(N, \mu) \myeq \prod_{(i, j)\in\mu}{(N + j - i)(N + a + j - i)} = \prod_{i = 1}^N{(N - i + 1)_{\mu_i}(N + a - i + 1)_{\mu_i}},
\end{equation}
where the first product is taken over the squares $(i, j)$ inside the Young diagram of partition $\mu$, and the second product uses the \textit{Pochhammer symbols} $(x)_n$, $x\in\C, n\in\Zp$, defined by
\begin{equation*}
(x)_n \myeq
    \begin{cases}
        \prod_{i=1}^n{(x+i-1)} & \textrm{if } n\geq 1,\\
        1 & \textrm{if } n = 0.
    \end{cases}
\end{equation*}
Consider also the sequence
\begin{equation*}
\boldsymbol{\epsilon} = (\epsilon^2, (\epsilon + 1)^2, (\epsilon + 2)^2, \ldots).
\end{equation*}
Finally, for any $\mu\in\GTp_N$, define the rational function
\begin{equation}\label{tstar}
T^*_{\mu|N}(x_1, \ldots, x_N) = \frac{1}{c(N, \mu)}\frac{\det_{i, j}[(x_i | \boldsymbol{\epsilon})^{\mu_j + N - j}]}{\det_{i, j}[(x_i | \boldsymbol{\epsilon})^{N - j}]}.
\end{equation}
One can show that the denominator $\det[(x_i | \boldsymbol{\epsilon})^{N - j}]$ is in fact the Vandermonde determinant $\Delta_N(x) = \prod_{i < j}{(x_i - x_j)}$. Moreover the rational function $T^*_{\mu|N}$ is in fact a symmetric polynomial on the $N$ variables $x_1, \ldots, x_N$. The polynomial $T^*_{\mu|N}$ has degree $|\mu| = \mu_1 + \mu_2 + \ldots + \mu_N$, but it is non-homogeneous; its top degree homogeneous part is a nonzero constant multiple of the Schur polynomial $s_{\mu}(x_1, \ldots, x_N)$, see \cite{M}.

For the positive $N$-signature $\mu$, we can also define a polynomial $T_{\mu|K}^*$ on $K$ variables $x_1, \dots, x_K$, for any $K \geq N$, by using the positive $K$-signature $\mu = (\mu_1, \dots, \mu_N, 0, \ldots, 0)$ with $K-N$ trailing zeroes instead of $\mu$ in the expression $(\ref{tstar})$.

\begin{prop}\label{coherence}
For any $\lambda\in\GTp_{N+1}$ and $\mu\in\GTp_N$, we have
\begin{equation*}
T_{\mu|N+1}^*(\hatl_1, \hatl_2, \ldots, \hatl_{N+1}) = \sum_{\nu\in\GTp_N}{\Lambda^{N+1}_N(\lambda, \nu)T_{\mu|N}^*(\hatn_1, \hatn_2, \ldots, \hatn_N)}.
\end{equation*}
\end{prop}
We leave the proof of Proposition $\ref{coherence}$ to Appendix $\ref{appendixA}$.

\begin{lem}\label{techlemma}
For any $\lambda, \mu\in\GTp_N$, we have
\begin{gather*}
\sum_{\nu\in\GTp_N}R_{u, u'}^{(N)}(\lambda, \nu)T^*_{\mu|N}(\hatn_1, \ldots, \hatn_N)\\
= \left( \sum_{j=1}^N{m_j(m_j - b - u - u' - 1)} - c_N \right)T^*_{\mu|N}(\hatl_1, \ldots, \hatl_N)\\
+ \sum_{j=1}^N{(u - m_j + 1)(u' - m_j + 1)\mathbf{1}_{\{\mu_j - 1 \geq \mu_{j+1}\}} T_{\mu - \mathbf{e}_j|N}^*(\hatl_1, \ldots, \hatl_N)},
\end{gather*}
where $c_N$ is as in $(\ref{dN})$, $\mathbf{e}_j = (0, \ldots, 0, 1, 0, \ldots, 0)$ with $1$ in the $j$-th position and zeroes elsewhere, and we set $\mu_{N+1} := 0$.
\end{lem}
\begin{proof}
\textbf{Step 1.} Recall the operators $\mathcal{D}^{(N)}$ of Section $\ref{specialization}$, which act on polynomial functions on $\Omega_N^{\epsilon}$ and are defined by
\begin{equation*}
\mathcal{D}^{(N)}f(\hatl_1, \ldots, \hatl_N) = \sum_{\nu\in\GTp_N}{R_{u, u'}^{(N)}(\lambda, \nu)f(\hatn_1, \ldots, \hatn_N)}.
\end{equation*}
In particular, for $N = 1$, we have the operator $\mathcal{D} = \mathcal{D}^{(1)}$ acts on the polynomial functions on $\Zp^{\epsilon}$ as
\begin{equation*}
\mathcal{D} f(\hatl) = \beta(l) \cdot \left(f(\widehat{l+1}) - f(\hatl)\right) + \delta(l) \cdot \left( f(\widehat{l - 1}) - f(\hatl) \right),
\end{equation*}
where the birth and death rates $\beta(l), \delta(l)$, are defined from the parameters $u, u'$.

Then we have
\begin{equation}\label{defDN}
\mathcal{D}^{(N)}T_{\mu|N}^*(\hatl_1, \ldots, \hatl_N) = \sum_{\nu\in\GTp_N}{R_{u, u'}^{(N)}(\lambda, \nu)T_{\mu|N}^*(\hatn_1, \ldots, \hatn_N)}.
\end{equation}

For any $1\leq i\leq N$, let $\mathcal{D}^{[i]}$ be the operator on the space of polynomial functions on $(\hatl_1, \dots, \hatl_N)\in\Omega_N^{\epsilon}$, which acts as $\mathcal{D}$ on the variable $i$-th variable $\hatl_i$, and treats all other variables $\hatl_j$, $j \neq i$, as constants. From Lemma $\ref{Dnoperator}$, it follows that
\begin{equation}\label{operatorTmuN}
\mathcal{D}^{(N)}T_{\mu|N}^*(\hatl_1, \ldots, \hatl_N) = \left( \frac{1}{\Delta_N(\hatl)} \circ \sum_{i=1}^N{\mathcal{D}^{[i]}} \circ \Delta_N(\hatl) - c_N \right)T_{\mu|N}^*(\hatl_1, \ldots, \hatl_N).
\end{equation}

\textbf{Step 2.} Let us find an expression for $\mathcal{D}[(\hatx | \epsilon)^m]$, $m\in\Zp$, where $\mathcal{D} = \mathcal{D}^{(1)}$ acts on $\R[\hatx]$. By definition,
\begin{equation*}
(\hatx | \boldsymbol{\epsilon})^m = ((x + \epsilon)^2 | \boldsymbol{\epsilon})^m = \prod_{i=1}^m{((x+\epsilon)^2 - (\epsilon+i-1)^2)} = (x \downarrow m)(x+2\epsilon \uparrow m),
\end{equation*}
where we have used the notation $(y\downarrow n) = \prod_{i=1}^n{(y - i + 1)}$, $(y\uparrow n) = \prod_{i=1}^n{(y+i-1)}$ for any $y\in\C$, $n\in\N$, and $(y\downarrow 0) = (y\uparrow 0) = 1$. Then
\begin{eqnarray*}
(\widehat{x+1} | \boldsymbol{\epsilon})^m - (\hatx | \boldsymbol{\epsilon})^m &=& m(2x + 2\epsilon + 1)\cdot(x\downarrow (m-1))(x+2\epsilon+1 \uparrow (m-1)),
\end{eqnarray*}
from which
\begin{eqnarray*}
\DD[(\hatx | \boldsymbol{\epsilon})^m] &=& \beta(x)((\widehat{x+1} | \boldsymbol{\epsilon})^m - (\hatx | \boldsymbol{\epsilon})^m) - \delta(x)((\hatx | \boldsymbol{\epsilon})^m - (\widehat{x-1} | \boldsymbol{\epsilon})^m)\\
&=& \frac{m \ (\hatx | \boldsymbol{\epsilon})^{m-1}}{2(x + \epsilon)}\left\{ (x + a + 1)(x - u)(x - u')(x + 2\epsilon + m - 1)\right.\\
&& -\left.(x + b)(x + u + 2\epsilon)(x + u' + 2\epsilon)(x - m + 1)\right\}\\
&=& m \ (\hatx|\boldsymbol{\epsilon})^{m - 1}\left\{ (m - b - u - u' - 1)(x - m + 1)(x + 2\epsilon + m - 1)\right.\\
&&+ \left.(a + m)(u - m + 1)(u' - m + 1) \right\}.
\end{eqnarray*}
Hence
\begin{equation}\label{explemma}
\mathcal{D}[(\hatx | \boldsymbol{\epsilon})^m] = m(m - b - u - u' - 1)(\hatx | \boldsymbol{\epsilon})^m + (a + m)m(u - m + 1)(u' - m + 1)(\hatx | \boldsymbol{\epsilon})^{m - 1}
\end{equation}

\textbf{Step 3.} From $(\ref{defDN})$ and $(\ref{operatorTmuN})$ of Step 1, and the definition of $T_{\mu|N}^*$ in $(\ref{tstar})$ give
\begin{eqnarray*}
\sum_{\nu\in\GTp_N}{R_{u, u'}^{(N)}(\lambda, \nu)T_{\mu|N}^*(\hatn_1, \ldots, \hatn_N)} &=& \frac{1}{c(N, \mu) \Delta_N(\hatl)}\cdot\sum_{i=1}^N{\mathcal{D}^{[i]}}(\det_{i, j}[(\hatl_i|\boldsymbol{\epsilon})^{\mu_j+N-j}])\\
&&-c_NT_{\mu|N}^*(\hatl_1, \ldots, \hatl_N)
\end{eqnarray*}
From $(\ref{explemma})$ of Step 2 and the definition of the polynomials $T_{\mu|N}^*$, the desired result follows. Let us make a remark on the factors $\mathbf{1}_{\{\mu_j - 1 \geq \mu_{j+1}\}}$ in the statement of the Lemma. If $\mu_j - 1 < \mu_{j+1}$, then $\mu_{j+1} = \mu_j$ and the formula $(\ref{tstar})$, with $\mu - \mathbf{e}_i$ instead of $\mu$, vanishes (the matrix in the numerator has two identical columns, so its determinant is zero).
\end{proof}

{\it Proof of Proposition $\ref{commratesprop}$}

Let $\lambda\in\GTp_{N+1}$, $\nu\in\GTp_N$ be arbitrary. By looking at the $(\lambda, \nu)$ entry of $(\ref{commrates1})$, the identity to prove is
\begin{equation}\label{commrates}
\sum_{\kappa\in\GT_{N+1}^+}{R^{(N+1)}_{u+1, u'+1}(\lambda, \kappa)\Lambda^{N+1}_N(\kappa, \nu)} = \sum_{\kappa\in\GTp_N}{\Lambda^{N+1}_N(\lambda, \kappa)R^{(N)}_{u, u'}(\kappa, \nu)}.
\end{equation}

For $\kappa\in\GTp_N$, we denote by $k, \hatk$ their associated elements in $\Omega_N, \Omega_N^{\epsilon}$, as described by the bijections in $(\ref{bijection})$. For $\kappa\in\GTp_{N+1}$, we also denote by $k, \hatk$ the associated elements in $\Omega_{N+1}, \Omega_{N+1}^{\epsilon}$.

Let us apply both sides of $(\ref{commrates})$ to $T_{\mu|N}^*(\hatn_1, \ldots, \hatn_N)$, for any $\mu\in\GTp_N$, i.e., multiply both sides of $(\ref{commrates})$ by $T_{\mu|N}^*(\hatn_1, \ldots, \hatn_N)$ and add over all $\nu\in\GTp_N$. Thus each side of $(\ref{commrates})$ becomes a double sum with finitely many terms, and we can exchange the order of summation if needed.

By virtue of Proposition $\ref{coherence}$, the left-hand side becomes 
\begin{equation*}
\sum_{\kappa\in\GT_{N+1}^+}{R_{u+1, u'+1}^{(N+1)}(\lambda, \kappa)T_{\mu|N+1}^*(\hatk_1, \ldots, \hatk_{N+1})}.
\end{equation*}
In view of Lemma $\ref{techlemma}$, this expression is
\begin{eqnarray*}
\left( \sum_{j=1}^{N+1}{\widetilde{m}_j(\widetilde{m}_j - b - u - u' - 3)} - \widetilde{c}_{N+1} \right) T_{\mu|N+1}^*(\hatl_1, \ldots, \hatl_{N+1})\\
+ \sum_{j=1}^{N+1}{(u - \widetilde{m}_j + 2)(u' - \widetilde{m}_j + 2)\mathbf{1}_{\{\mu_j - 1 \geq \mu_{j+1}\}}T_{\mu - \mathbf{e}_j | N+1}^*(\hatl_1, \ldots, \hatl_{N+1})},
\end{eqnarray*}
where $\widetilde{m}_j = \mu_j + N + 1 - j$ and $\widetilde{c}_{N+1}$ has the same definition as $c_{N+1}$ except that $u, u'$ are replaced by $u + 1, u' + 1$, respectively.

On the other hand, if we apply Lemma $\ref{techlemma}$ to the right-hand side, we obtain
\begin{eqnarray*}
\sum_{\kappa\in\GTp_N}\Lambda^{N+1}_N(\lambda, \kappa)\left\{\left( \sum_{j=1}^N{m_j(m_j - b - u - u' - 1)} - c_N \right)T^*_{\mu|N}(\hatk_1, \ldots, \hatk_N)\right.\\
+ \left.\sum_{j=1}^N{(u - m_j + 1)(u' - m_j + 1)\mathbf{1}_{\{\mu_j - 1 \geq \mu_{j+1}\}} T_{\mu - \mathbf{e}_j|N}^*(\hatk_1, \ldots, \hatk_N)}\right\}.
\end{eqnarray*}
Proposition $\ref{coherence}$ then shows the expressions are equal if
\begin{eqnarray*}
\sum_{j=1}^{N+1}{\widetilde{m}_j(\widetilde{m}_j - b - u - u' - 3)} - \widetilde{c}_{N+1} &=& \sum_{j=1}^N{m_j(m_j - b - u - u' - 1)} - c_N\\
(u - \widetilde{m}_j + 2)(u' - \widetilde{m}_j + 2) &=& (u - m_j + 1)(u' - m_j + 1), \hspace{.2in}j = 1, 2, \ldots, N,
\end{eqnarray*}
both of which are easily checked. The proof of Proposition $\ref{commratesprop}$ is finished, because of Lemma $\ref{easyeval}$ and the analysis above.

\subsection{From $q$-matrices to semigroups}

We finish the proof of Theorem $\ref{intertwiningthm}$ by lifting the infinitesimal version shown in Proposition $\ref{commratesprop}$. The idea of the proof is to prove the commutativity relation $P_{N+1}(t)\Lambda^{N+1}_N = \Lambda^{N+1}_NP_N(t)$ as an equality of operators $C_0(\GTp_N)\longrightarrow C_0(\GTp_{N+1})$: for any $N \in \N$, $N \geq 2$, and any $f\in C_0(\GTp_N)$, we prove 
\begin{equation}\label{desequality}
P_{N+1}(t)\Lambda^{N+1}_Nf = \Lambda^{N+1}_NP_N(t)f, \textrm{ for all } t\geq 0.
\end{equation}
Since finitely supported functions on $\GTp_N$ span a dense subspace of $C_0(\GTp_N)$, it will be sufficient to prove $(\ref{desequality})$ for all finitely supported functions $f$ on $\GTp_N$.
In order to do that, let $A$ be the operator $C_0(\GTp_{N+1})\longrightarrow C_0(\GTp_{N+1})$ associated to the $q$-matrix $R^{(N+1)}_{z+N, z'+N}$ (here $(z, z')$ is any complex pair satisfying the inequalities $(\ref{zparamsnoname})$). From Kolmogorov's backward differential equation and the infinitesimal master relation $(\ref{commratesprop})$, one can show that both sides of the desired equality $(\ref{desequality})$ satisfy the following linear ODE with values in the Banach space $C_0(\GTp_{N+1})$ and with certain initial condition:
\begin{eqnarray*}
    F'(t) &=& AF(t), \hspace{.2in}t\geq 0,\\
    F(0) &=& \Lambda^{N+1}_Nf.
\end{eqnarray*}
Finally, one can invoke a general theorem about uniqueness of solutions $F(t)$ to the above \textit{Cauchy problem} to finish the proof.

In \cite[Sub. 6.3]{BO2}, the reader can find a much more detailed explanation of the idea above, which applies to our situation. The only unclear point is the statement in the following lemma, which is a technical condition guaranteeing that we can apply the general result on uniqueness of solutions to a Cauchy problem.
\begin{lem}\label{techcauchy}
Let $N\geq 2$, let $f$ be a finitely supported function on $C_0(\GTp_N)$, and let $A$ be the operator $C_0(\GTp_{N+1})\longrightarrow C_0(\GTp_{N+1})$ associated to $R^{(N+1)} = R^{(N+1)}_{z+N, z'+N}$. Then
\begin{equation*}
A(\Lambda^{N+1}_Nf) \in C_0(\GTp_{N+1}).
\end{equation*}
\end{lem}

\begin{proof}
It is convenient to name the constants $u = z+N$, $u' = z'+N$. Since $\Lambda^{N+1}_N$ and $A$ are linear operators, it is enough to prove the proposition when $f$ is the delta function $\delta_{\mu}$, i.e., we show $A(\Lambda^{N+1}_N\delta_{\mu})\in C_0(\GTp_{N+1})$, for any $\mu\in\GTp_N$. This statement is equivalent to
\begin{equation*}
\sum_{\nu\in\GTp_{N+1}}{R^{(N+1)}(\lambda, \nu)\Lambda^{N+1}_N(\nu, \mu)} \xrightarrow{\lambda\rightarrow\infty} 0.
\end{equation*}
Recall that $R^{(N+1)}(\lambda, \nu)\neq 0$ iff $\sum_k{|\lambda_k - \nu_k|} \in \{0, 1\}$, or equivalently iff $\nu = \lambda$ or $\nu = \lambda \pm e_k$ for some $k = 1, \ldots, N$, where $e_k = (0, \ldots, 0, 1, 0, \ldots, 0)$ has a $1$ in the $k$-th position and $0$'s elsewhere. Moreover
\begin{equation*}
R^{(N+1)}(\lambda, \lambda) = -\sum_{\substack{\nu\in\GTp_{N+1} \\ \nu \neq \lambda}}{R^{(N+1)}(\lambda, \nu)}.
\end{equation*}
The limit to prove becomes
\begin{gather}
\sum_{k=1}^{N+1}\left\{ \mathbf{1}_{\{\lambda_{k-1} \geq \lambda_k + 1\}}R^{(N+1)}(\lambda, \lambda + e_k)(\Lambda^{N+1}_N(\lambda + e_k, \mu) - \Lambda^{N+1}_N(\lambda, \mu))\right.\nonumber\\
\left. + \mathbf{1}_{\{\lambda_k - 1 \geq \lambda_{k+1}\}} R^{(N+1)}(\lambda, \lambda - e_k)(\Lambda^{N+1}_N(\lambda - e_k, \mu) - \Lambda^{N+1}_N(\lambda, \mu)) \right\}
\xrightarrow{\lambda\rightarrow\infty} 0,\label{toproveconvergence1}
\end{gather}
where we have used the conventions $\lambda_0 := +\infty$, $\lambda_{N+2} := 0$.

We claim that
\begin{equation}\label{fineestimate}
\Lambda^{N+1}_N(\lambda, \mu) \leq const\cdot \lambda_1^{-N-1},
\end{equation}
where $const$ is a positive constant, independent of $\lambda$. Assume for the moment the validity of this estimate. The expressions $(\ref{Rexpression})$ for $R^{(N)}$ show that $|R^{(N)}(\lambda, \lambda \pm e_k)| \leq const\cdot\lambda_1^2$, for all $k = 1, \ldots, N$, where $const$ depends on $N$ but not on $\lambda$. By combining this observation with estimate $(\ref{fineestimate})$, we have that the absolute value of the sum in $(\ref{toproveconvergence1})$ is upper bounded by $const\cdot \lambda_1^{-N+1}\leq const\cdot \lambda_1^{-1}$, for $N\geq 2$. Since $\lambda\rightarrow\infty$ is equivalent to $\lambda_1\rightarrow\infty$, then $(\ref{toproveconvergence1})$ follows.

We are left to prove the estimate $(\ref{fineestimate})$. In $(\ref{estimateLambda})$, we proved the weaker estimate $\Lambda^{N+1}_N(\lambda, \mu) \leq const\cdot \lambda_1^{-N}$, but it can be improved if we use a finer estimate for the Vandermonde $\Delta_{N+1}(\hatl)$ when $N\geq 2$.  We have
\begin{gather*}
\Delta_{N+1}(\hatl) = \prod_{i<j}{(\hatl_i - \hatl_j)} = \prod_{i<j}{(l_i - l_j)(l_i + l_j + 2\epsilon)}\\
\geq (\lambda_1 - \lambda_2 + 1)(\lambda_1 - \lambda_3 + 2)\prod_{j=2}^{N+1}{(\lambda_1 + \lambda_j + 2\epsilon)}\geq const\cdot(\lambda_1 - \lambda_2 + 1)\lambda_1^{N+1},
\end{gather*}
where we have used that $\lambda \succ_{BC} \mu$ implies $\mu_1 \geq \lambda_3$ and so $\lambda_1 - \lambda_3 + 2 \geq const\cdot \lambda_1$.

By following the same steps as in the proof of $(\ref{estimateLambda})$, with the finer estimate on $\Delta_{N+1}(\hatl)$, we obtain the desired $\Lambda^{N+1}_N(\lambda, \mu) \leq const\cdot \lambda_1^{-N-1}$.
\end{proof}

\begin{rem}
We believe Lemma $\ref{techcauchy}$ holds also for $N = 1$ (and consequently so would Theorem $\ref{intertwiningthm}$). If we consider $(\lambda_1, \lambda_2)\in\GTp_2$, $(\mu)\in\GTp_1$, the lemma for $N = 1$ is equivalent to the convergence of the sum
\begin{gather*}
\beta(\lambda_1 + 1)\frac{(\lambda_1 + 2 + \epsilon)^2 - (\lambda_2 + \epsilon)^2}{(\lambda_1 + 1 + \epsilon)^2 - (\lambda_2 + \epsilon)^2} \left\{ \Lambda^2_1((\lambda_1+1, \lambda_2), \mu) - \Lambda^2_1((\lambda_1, \lambda_2), \mu) \right\}\\
+ \mathbf{1}_{\{\lambda_1 - 1 \geq \lambda_2\}}\delta(\lambda_1 + 1)\frac{(\lambda_1 + \epsilon)^2 - (\lambda_2 + \epsilon)^2}{(\lambda_1 + 1 + \epsilon)^2 - (\lambda_2 + \epsilon)^2} \left\{ \Lambda^2_1((\lambda_1-1, \lambda_2), \mu) - \Lambda^2_1((\lambda_1, \lambda_2), \mu) \right\}\\
+\mathbf{1}_{\{\lambda_1 \geq \lambda_2+1\}}\beta(\lambda_2)\frac{(\lambda_1 + 1 + \epsilon)^2 - (\lambda_2 + 1+ \epsilon)^2}{(\lambda_1 + 1 + \epsilon)^2 - (\lambda_2 + \epsilon)^2} \left\{ \Lambda^2_1((\lambda_1, \lambda_2 + 1), \mu) - \Lambda^2_1((\lambda_1, \lambda_2), \mu) \right\}\\
+\mathbf{1}_{\{\lambda_2 - 1\geq 0\}}\delta(\lambda_2)\frac{(\lambda_1 + 1 + \epsilon)^2 - (\lambda_2 - 1+ \epsilon)^2}{(\lambda_1 + 1 + \epsilon)^2 - (\lambda_2 + \epsilon)^2} \left\{ \Lambda^2_1((\lambda_1, \lambda_2 - 1), \mu) - \Lambda^2_1((\lambda_1, \lambda_2), \mu) \right\}
\end{gather*}
to $0$, as $\lambda\longrightarrow\infty$. We even have the following explicit expression for the links
\begin{equation*}
\Lambda^2_1((\lambda_1, \lambda_2), (\mu)) = \frac{2(a+1)\sum_{\lambda_1 \geq n\geq \max\{\mu, \lambda_2\}}{B(n, \mu)}}{(\lambda_1 - \lambda_2 + 1)(\lambda_1 + \lambda_2 + 1 + 2\epsilon)}.
\end{equation*}
The difficulty is that many cases arise for the sequence $\{(\lambda_1(k), \lambda_2(k))\}_{k=1, 2, \ldots}$ such that $\lambda_1(k)\xrightarrow{k\rightarrow\infty}\infty$, and we also need to treat the cases $\mu\neq 0$ and $\mu = 0$ separately. The author has verified the desired convergence when $\lambda_1(k) > \lambda_2(k) > \mu > 0$ for all $k$ sufficiently large, but the proof is tedious. Since proving the master relation for $N = 1$ does not strengthen our main result, we have chosen to leave it out.
\end{rem}

\section{Invariance of the $z$-measures}\label{zmeasuressubsection}

In this section, we use the semigroups $(P_N(t))_{t\geq 0}$ and matrices of transition rates $R^{(N)}$, which are defined using the parameters $(u, u') = (z + N - 1, z' + N - 1)$, as in Section $\ref{commuting}$. Here, we restrict even further the space of pairs $(z, z')$ that we use: the $z$-measures exist for pairs $(z, z')$ in a subset $\mathcal{H}\subset\mathcal{H}^{\circ}$, which will be defined shortly.

\subsection{Space of parameters $(z, z')$}

Define the following sets:
\begin{eqnarray*}
\mathcal{Z} &\myeq& \mathcal{Z}_{princ}\sqcup\mathcal{Z}_{compl}\sqcup\mathcal{Z}_{degen}\\
\mathcal{Z}_{princ} &\myeq& \{(z, z')\in\C^2 \setminus \R^2 : z' = \overline{z}\}\\
\mathcal{Z}_{compl} &\myeq& \{(z, z')\in\R^2 : \exists \ m\in\Z,\ m<z, z'<m+1\}\\
\mathcal{Z}_{degen} &\myeq& \bigsqcup_{m\in\N}{\mathcal{Z}_{degen, m}}\\
\mathcal{Z}_{degen, m} &\myeq& \{(z, z')\in\R^2 : z=m, z'>m-1 \textrm{ or }z'=m, z>m-1\}.
\end{eqnarray*}
Observe that $(z, z')\in\mathcal{Z}$ implies that $z+z'\in\R$. Moreover we have
\begin{prop}[\cite{Ol2}, Lemma 7.9]\label{gammaineqs}
Let $(z, z')\in\C^2$, then
\begin{itemize}
	\item $(\Gamma(z-k)\Gamma(z'-k))^{-1} \geq 0$ for all $k\in\Z$ iff $(z, z')\in\mathcal{Z}$.
	\item $(\Gamma(z-k)\Gamma(z'-k))^{-1} > 0$ for all $k\in\Z$ iff $(z, z')\in\mathcal{Z}_{princ}\sqcup\mathcal{Z}_{compl}$.
	\item If $(z, z')\in\mathcal{Z}_{degen, m}$, $m\in\N$, then $(\Gamma(z-k)\Gamma(z'-k))^{-1} = 0$ for $k \geq m$ and $(\Gamma(z-k)\Gamma(z'-k))^{-1} > 0$ for $k<m$.
\end{itemize}
\end{prop}

\begin{df}\label{Hdef}
We define $\mathcal{H}$ as the space of pairs $(z, z')\in\C^2$ such that both $(z, z')$ and $(z+2\epsilon, z'+2\epsilon)$ belong to $\mathcal{Z}_{princ}\sqcup\mathcal{Z}_{compl}$ and moreover $z+z'+b>-1$.
\end{df}

Observe that all pairs $(z, z')\in\mathcal{H}$ satisfy the inequalities in $(\ref{zparamsnoname})$. Consequently the master relation, Theorem $\ref{intertwiningthm}$, holds if we define the Feller semigroups $(P_N(t))_{t\geq 0}$ using a pair $(z, z')\in\mathcal{H}$, as described in the previous section. The method of intertwiners, Theorem $\ref{methodintertwiners}$, then shows there exists a unique Feller semigroup $(P_{\infty}(t))_{t \geq 0}$ on $\Omega_{\infty}$ satisfying the relations $(\ref{commutativity2})$. The Feller semigroups $(P_{\infty}(t))_{t \geq 0}$ depend on the parameters $z, z'$, as well as the real parameters $a \geq b \geq -1/2$; for convenience, we omit such dependence from the notation.

To finish the proof of the main theorem stated in the introduction, we define below the $z$-measures associated to a tuple $(z, z', a, b)$, $(z, z')\in\mathcal{H}$, $a \geq b \geq -\frac{1}{2}$, and prove that they are the unique invariant measures of the semigroups $(P_{\infty}(t))_{t\geq 0}$.

\begin{rem}
\normalfont Below we define the $z$-measures for $(z, z')\in\mathcal{H}$, but it is also possible to define them for degenerate pairs $(z, z')\in\mathcal{Z}_{degen}$. One can define Markov chains on $\GTp_N$ that preserve the pushforwards of the ``degenerate'' $z$-measures. We believe these Markov chains have a limit Markov process (living in a finite-dimensional subspace of $\Omega_{\infty}$) which is a time-dependent determinantal point process. Techniques from \cite{Gor} probably can prove this conjecture. However, we have restricted ourselves to study the more complicated dynamics in the non-degenerate case $(z, z')\in\mathcal{H}$.
\end{rem}

\subsection{$z$-measures on the BC branching graph}

The \textit{spectral $z$-measure} (or simply \textit{$z$-measure}) associated to $(z, z')\in\mathcal{H}$ (and the pair $(a \geq b\geq -1/2)$) is certain probability measure $M_{z, z', a, b | \infty}$ on the boundary $\Omega_{\infty}$ of the BC branching graph. By virtue of the bijection $(\ref{themap})$, we can define it by giving a coherent system of probability measures $\{M_{z, z', a, b | N}\}_{N = 1, 2, \ldots}$ on the levels $\GTp_N$ of the BC branching graph. Let us proceed with this approach; for $(z, z')\in\mathcal{H}$, define
\begin{equation}\label{zmeasureN}
M_{z, z', a, b|N}(\lambda) = C_{z, z', a, b|N}^{-1} \prod_{1\leq i < j \leq N}{(\hatl_i - \hatl_j)^2}\cdot\prod_{i=1}^N{W_{z, z', a, b|N}(l_i)}, \hspace{.1in} \lambda\in\GTp_N,
\end{equation}
where
\begin{equation}\label{Wexpression}
\begin{gathered}
W_{z, z', a, b|N}(x) = (x + \epsilon)\frac{\Gamma(x + 2\epsilon)\Gamma(x + a + 1 )}{\Gamma(x + b + 1)\Gamma(x + 1)}\\
\times\frac{1}{\Gamma(z - x + N)\Gamma(z' - x + N)\Gamma(z + x + N + 2\epsilon)\Gamma(z' + x + N + 2\epsilon)},
\end{gathered}
\end{equation}

\begin{equation}\label{Cexpression}
C_{z, z', a, b|N} = \prod_{i=1}^N{\frac{\Gamma(a+i)\Gamma(b+z+z'+i)\Gamma(i)}{\Gamma(z+i)\Gamma(z'+i)\Gamma(z+b+i)\Gamma(z'+b+i)\Gamma(z+z'+N+a+b+i)}}.
\end{equation}

With the restrictions on the parameters $z, z', a, b$ in place, the expressions $W_{z, z', a, b|N}(x)$, $x\in\Zp$, $C_{z, z', a, b|N}^{-1}$ can be shown to be well-defined. For example, we can argue that the equality $z + b = -1$ that would make the formula for $C_{z, z', a, b | N}^{-1}$ ill-defined does not occur. Indeed if $z+b=-1$, then $z<0$ and $(z, z')\in\mathcal{Z}_{compl}$; consequently $z' < 0$ and so $z+z'+b<z+b=-1$, a contradiction. Similarly one gets rid of the unwanted scenarios $z+b=-n$ and $z'+b=-n$, for some $n\in\Zp$.
Therefore $M_{z, z', a, b|N}(\lambda)$ is also well-defined for any $\lambda\in\GTp_N$.

Note that the formula above also defines $z$-measures for all $a, b>-1$, but the main theorem of this paper requires $a\geq b\geq -1/2$, so we do not need to consider them in full generality.

\begin{thm}
The expressions $M_{z, z', a, b|N}$, with $(z, z')\in\mathcal{H}$, define a probability measure on $\GTp_N$, which assigns strictly positive probabilities to all elements of $\GTp_N$. They are consistent with the kernels $\Lambda^{N+1}_N$ in the sense that
\begin{equation}\label{coherencez}
M_{z, z', a, b|N}(\mu) = \sum_{\lambda\in\GTp_{N+1}}{M_{z, z', a, b|N+1}(\lambda)\Lambda^{N+1}_N(\lambda, \mu)}, \textrm{ for any } \mu\in\GTp_N.
\end{equation}
Therefore the coherent system $\{M_{z, z', a, b|N}\}_{N\geq 1}$ determines a probability measure on $\Omega_{\infty}$, due to $(\ref{tobeiso})$, that we call the \textit{spectral $z$-measure $M_{z, z', a, b|\infty}$ associated to $(z, z', a, b)$}.
\end{thm}

\begin{proof}
By virtue of Proposition $\ref{gammaineqs}$, we have $M_{z, z', a, b|N}(\lambda) > 0$ for any $\lambda\in\GTp_N$, whenever $(z, z')\in\mathcal{H}$. Thus the expression $(\ref{zmeasureN})$ defines a probability measure on $\GTp_N$ as long as
\begin{equation}\label{finitemeasure}
C_{z, z', a, b|N} = \sum_{\lambda\in\GTp_N}{M'_{z, z', a, b|N}(\lambda)},
\end{equation}
where
\begin{equation*}
M'_{z, z', a, b|N}(\lambda) = \prod_{1\leq i < j\leq N}{(\hatl_i - \hatl_j)^2}\cdot\prod_{i=1}^N{W_{z, z', a, b|N}(l_i)}.
\end{equation*}
We actually prove $(\ref{finitemeasure})$ for all pairs $(z, z')$ belonging to the domain $U \myeq \{(z, z')\in\C^2 : \Re(z+z'+b) > -1\}$. The equality was proved in \cite{OlOs}, for the special case when $\Re z, \Re z' > -(1+b)/2$. We extend the equality by analytic continuation. All we need to show is that both sides of the desired equality $(\ref{finitemeasure})$ are analytic for $(z, z')\in U$. The expression $C_{z, z', a, b|N}$ is clearly analytic on $U$, so it only remains to prove that the sum in the right-hand side of $(\ref{finitemeasure})$ converges absolutely on $U$.

Due to Euler's reflection formula $\Gamma(z)\Gamma(1 - z) = \frac{\pi}{\sin(\pi z)}$, see \cite[Thm 1.2.1]{AAR}, for large $x\in\Zp$ we can express $W(x) = W_{z, z', a, b|N}(x)$ as
\begin{gather*}
W(x) = \frac{(x+\epsilon)\sin(\pi z)\sin(\pi z')}{\pi^2}\cdot\frac{\Gamma(x+a+b+1)\Gamma(x+a+1)}{\Gamma(x+b+1)\Gamma(x+1)}\\
\times\frac{\Gamma(x+1-z-N)\Gamma(x+1-z'-N)}{\Gamma(x+z+N+2\epsilon)\Gamma(x+z'+N+2\epsilon)}.
\end{gather*}
The well known asymptotics of Gamma functions yield
\begin{equation}\label{familiarasymp}
|W(x)| \leq const\cdot (1 + x)^{1 - 2\Re\Sigma - 4N}, \forall \ x\in\Zp,
\end{equation}
where we denoted $\Sigma = z+z'+b$. On the other hand, we clearly have
\begin{equation*}
\prod_{1\leq i<j\leq N}{(\hatl_i - \hatl_j)^2} \leq \hatl_1^{2N-2}\cdots\hatl_{N-1}^2.
\end{equation*}
It follows that the sum $\sum_{\lambda\in\GTp_N}{|M'_{z, z', a, b|N}(\lambda)|}$ is upper bounded by
\begin{gather}
const \sum_{l_1, \ldots, l_N\in\Zp}{\hatl_1^{2N-2}\cdots\hatl_{N-1}^{2}((1+l_1)\cdots(1+l_N))^{1-2\Re\Sigma-4N}}\nonumber\\
\leq const \sum_{l_1, \ldots, l_N\in\Zp}{\hatl_1^{2N-2}\cdots\hatl_{N-1}^{2N-2}((1+l_1)\cdots(1+l_N))^{1-2\Re\Sigma-4N}}\nonumber\\
= const \left( \sum_{l\in\Zp}{\hatl^{2N-2} (1+l)^{1-2\Re\Sigma-4N}} \right)^N = const \left( \sum_{l\in\Zp}{(l+\epsilon)^{4N-4} (1+l)^{1-2\Re\Sigma-4N}} \right)^N \label{tobebounded}
\end{gather}

For large $l\in\Zp$, we have $l < 1+l < 2l$ and $l < l+\epsilon < 2l$. On the other hand, our assumption $\Re\Sigma > -1$ implies the convergence of $\sum_{l\in\Zp}{l^{(4N - 4)+(1-2\Re\Sigma -4N)}} = \sum_{l\in\Zp}{l^{-3-2\Re\Sigma}}$. It follows that $(\ref{tobebounded})$ is bounded and therefore $\sum_{\lambda\in\GTp_N}{M'_{z, z', a, b|N}(\lambda)}$ is absolutely convergent, as desired.

Next, we prove the coherence relation $(\ref{coherencez})$, for all pairs in the domain $(z, z')\in U' \myeq \{(z, z')\in\C^2 : \Re (z + z' + b) > -1\}\cap\{z \neq -1, -2, \dots\}\cap\{z' \neq -1, -2, \dots\}\cap\{z + b \neq -1, -2, \dots\}\cap\{z' + b \neq -1, -2, \dots\}$.
This is stronger than the statement we wanted to show because clearly $\mathcal{H} \subset U'$.

In the case that $\Re z, \Re z' > -\frac{(1+b)}{2}$, the coherence relation $(\ref{coherencez})$ is the main result of \cite{OlOs}. By the principle of analytic continuation, we need to show that both sides of $(\ref{coherencez})$ are holomorphic on $U'$. The left side is clearly holomorphic on $U'$, by inspecting $(\ref{zmeasureN})$. As for the right side of $(\ref{coherencez})$, we need to show that the sum below is absolutely convergent
\begin{equation*}
\sum_{\lambda\in\GTp_{N+1}}{M_{z, z', a, b|N+1}(\lambda)\Lambda^{N+1}_N(\lambda, \mu)},
\end{equation*}
on $U'$, for any $\mu\in\GTp_N$.

This follows easily from (a) $\Lambda^{N+1}_N(\lambda, \mu)\in [0, 1]$, (b) the analysis in the previous paragraphs which shows that the sum $\sum_{\lambda\in\GTp_{N+1}}{M'_{z, z', a, b|N+1}(\lambda)}$ is absolutely convergent on $U \supset U'$, and (c) the fact that $C_{z, z', a, b|N+1}^{-1}$ is holomorphic on $U'$, see the few words after equation $(\ref{Cexpression})$ for a brief explanation.
\end{proof}

\subsection{Wilson-Neretin hypergeometric orthogonal polynomials}\label{neretinsub}

We review the theory of a family of hypergeometric orthogonal polynomials whose weight has the same form as $W_{z, z', a, b|N}$. The result below was proved in \cite{N}, but we follow the exposition in \cite[Section 8]{BO4}.

Take arbitrary complex numbers $a_1, a_2, a_3, a_4$ and $\alpha$, and consider the function
\begin{equation}\label{weightneretin}
w(t | a_1, a_2, a_3, a_4; \alpha) \myeq \frac{t+\alpha}{\prod_{j=1}^4{\Gamma(a_j + \alpha + t)\Gamma(a_j - \alpha - t)}}, \hspace{.1in} t\in\Z,
\end{equation}
which is treated as a weight function on the quadratic lattice $\{\ldots, (-1+\alpha)^2, \alpha^2, (1+\alpha)^2, \ldots\}$.

\begin{prop}\label{orthpolyinfo}
The polynomials
\begin{eqnarray*}
&&Q_n((t+\alpha)^2) = \frac{\Gamma(2-a_1-a_2+n)\Gamma(2-a_1-a_3+n)\Gamma(2-a_1-a_4+n)}{\Gamma(2-a_1-a_2)\Gamma(2-a_1-a_3)\Gamma(2-a_1-a_4)}\\
&&\times\pFq{4}{3}{-n,,, n+3-a_1-a_2-a_3-a_4 ,,, 1-a_1+t+\alpha ,,, 1-a_1-t-\alpha}{2-a_1-a_2 ,,, 2-a_1-a_3 ,,, 2-a_1-a_4}{1}
\end{eqnarray*}
are orthogonal with respect to the weight $w(t)$ in $(\ref{weightneretin})$, i.e.,
\begin{equation*}
    \sum_{k=-\infty}^{\infty}{Q_m((k + \alpha)^2)Q_n((k + \alpha)^2)w(k)} = 0, \hspace{.1in}\textrm{ for } m\neq n.
\end{equation*}
\end{prop}
\begin{rem}\label{kn}
\normalfont The polynomial $Q_n$ has degree $n$ with respect to the variable $(t+\alpha)^2$. Observe that $Q_n$ is not monic: its leading coefficient is
\begin{equation*}
k_n = \frac{\Gamma(2n+3-a_1-a_2-a_3-a_4)}{\Gamma(n+3-a_1-a_2-a_3-a_4)}
\end{equation*}
\end{rem}
\begin{rem}\label{neretintowilson}
\normalfont The orthogonal polynomials appearing in Proposition $\ref{orthpolyinfo}$ can be expressed in terms of the well-known \textit{Wilson polynomials}; in fact
\begin{equation}\label{neretinwilson}
Q_n((t+\alpha)^2) = r_n(-(t + \alpha)^2; 1 - a_1, 1 - a_2, 1 - a_3, 1 - a_4),
\end{equation}
 where $r_n$ is the degree $n$ Wilson polynomial in the variable $(t+\alpha)^2$, see \cite[Ch. 9, Sub. 9.1]{KLS}. (In \cite{KLS} and other sources, the degree $n$ Wilson polynomial is denoted by $W_n$, but since the letter $W$ is already in use for the weight function $(\ref{Wexpression})$, we use $r_n$ instead.)

In the setting of finite systems of discrete orthogonal polynomials on a quadratic lattice, they appear in a work of Neretin, see \cite{N}. We will call them the \textit{Wilson-Neretin polynomials}.
\end{rem}

\begin{rem}\label{halflattice}
Observe that $t$ ranges over $\Z$ instead of $\Zp$. However, if some $a_j$ equals $1 - \alpha$, then the weight $w(t)$ vanishes for $t\in\Z_{<0}$ and the Wilson-Neretin polynomials become a system of orthogonal polynomials in the quadratic half-lattice $\{\alpha^2, (\alpha+1)^2, (\alpha+2)^2, \ldots\}$.

The system is moreover a \textit{finite system} of orthogonal polynomials, meaning that only finitely many polynomials $Q_n$ will exist for a choice of parameters $\alpha, a_1, a_2, a_3, a_4$. In fact, one can show $w(t) \sim t^{5 - 2\sum_i{a_i}}$ as $t\rightarrow\infty$, cf. $(\ref{familiarasymp})$ above. If $m = m(a_1, a_2, a_3, a_4)$ is the largest integer that is strictly below $2\sum_i{a_i} - 6$, then $w(t)$ has up to $m$ finite moments and so only the polynomials $Q_1((t+\alpha)^2), \ldots, Q_n((t+\alpha)^2)$, $n = \lfloor (m+2)/4 \rfloor$ exist\footnote{This simple observation follows because $n$ is the largest positive integer for which any sum of the form $\sum_{t\in\Zp}{w(t)q_n((t+\alpha)^2)q_{n-1}((t+\alpha)^2)}$, with $q_n, q_{n-1}$ polynomials of degrees $n, n-1$, converges.}.
\end{rem}

\subsection{Difference equation for Wilson-Neretin polynomials}

We specialize the previous discussion to our case of interest.

We write $\hatx = (x+\epsilon)^2$, for $x\in\Zp$, as usual. Let $\{\p_n(\hatx)\}_n$ be the monic orthogonal polynomials coming from the orthogonalization of $(1, \hatx, \hatx^2, \ldots)$ in the Hilbert space $L^2(\Zp^{\epsilon},\ W\cdot\mu)$, where $\mu$ is the counting measure of $\Zp^{\epsilon}$ and we denote $W = W_{z, z', a, b|N}$. In particular, $\mathfrak{p}_0 = 1$, $\mathfrak{p}_n$ is of degree $n$ in the variable $\hatx = (x + \epsilon)^2$ (therefore of degree $2n$ as a polynomial on $x$), and the polynomials satisfy the orthogonality relations
\begin{equation*}
\sum_{x\in\Z}{\p_n(\widehat{x}) \p_m(\widehat{x})W(x)} = 0\textrm{, whenever } n\neq m.
\end{equation*}

Observe that upon making the substitutions
\begin{equation*}
\Gamma(x + a + 1) = \frac{\pi}{\sin(\pi(x + a + 1))\Gamma(-x - a)}, \hspace{.1in} \Gamma(x + 2\epsilon) = \frac{\pi}{\sin(\pi(x + 2\epsilon))\Gamma(1 - x - 2\epsilon)}
\end{equation*}
into the formula of $W_{z, z', a, b|N}$, for $x\in\Zp$, we have
\begin{equation*}
W_{z, z', a, b|N}(x) = w(x \ | \ 1 - \epsilon, \ -a + \epsilon, \ z + N + \epsilon, \ z' + N + \epsilon \ ; \ \epsilon) \times const,
\end{equation*}
where
\begin{equation*}
const = \frac{\pi^2}{\sin(\pi a)\sin(\pi(a+b))},
\end{equation*}
and where $w(t | a_1, a_2, a_3, a_4; \alpha)$ is defined above in $(\ref{weightneretin})$. Therefore the orthogonal polynomials $\p_n(\hatx)$ must be (up to constant factors) the Wilson-Neretin polynomials with parameters 
\begin{equation}\label{identification}
\begin{array}{r@{}l}
    &{} \alpha = \epsilon, \hspace{.1in} a_1 = 1 - \epsilon, \hspace{.1in} a_2 = -a+\epsilon\\
    &{} a_3 = z+N+\epsilon, \hspace{.1in} a_4 = z'+N+\epsilon
\end{array}
\end{equation}
Note also that $a_1 = 1 - \epsilon$ puts us in the setting of Remark $\ref{halflattice}$, i.e., the polynomials $\mathfrak{p}_n$ form a finite system of orthogonal polynomials on the quadratic half lattice $\Zp^{\epsilon}$. Moreover, since $2\sum_i{a_i} - 6 = 2(z+z'+b + 2N - 1)$  and $z+z'+b > -1$, the last statement of Remark $\ref{halflattice}$ yields the existence of the polynomials $\p_0, \p_1, \ldots, \p_{N-1}$.

Furthermore, from Remarks $\ref{kn}$ and $\ref{neretintowilson}$, it follows that
\begin{equation*}
\mathfrak{p}_n(\widehat{x}) = \frac{1}{k_n}r_n(-\widehat{x}),
\end{equation*}
where $k_n$ is the leading coefficient of the Wilson polynomial $r_n(-\hatx \ ; \ \epsilon, 1 + a - \epsilon, 1 - z - N - \epsilon, 1 - z' - N - \epsilon)$. In particular, they satisfy the following second degree difference equation, characteristic of the Wilson polynomials, see \cite[Sec. 9, Sub. 9.1]{KLS}:
\begin{eqnarray*}
n(n + 1 - z - z' - b - 2N)\p_n(\widehat{x}) &=& a^+(x)(\p_n(\widehat{x + 1}) - \p_n(\widehat{x}))\\
&&+ a^-(x)(\p_n(\widehat{x - 1}) - \p_n(\widehat{x})),
\end{eqnarray*}
where
\begin{eqnarray*}
a^+(x) &=& \frac{(x + 2\epsilon)(x + a + 1)(x + 1 - z - N)(x + 1 - z' - N)}{2(x + \epsilon)(2x + 2\epsilon + 1)},\\
a^-(x) &=& \frac{x(x + b)(x + z + N + a + b)(x + z' + N + a + b)}{2(x + \epsilon)(2x + 2\epsilon - 1)}.
\end{eqnarray*}

Given the coincidence of the expressions $a^{\pm}(x)$ above with the entries for the matrix $R$ of transition rates, the discussion above can be rephrased as
\begin{prop}\label{diff}
The Wilson-Neretin polynomials $\p_n(\hatx)$ with parameters $(\ref{identification})$ are eigenfunctions of the operator $\DD$ associated to $R = R^{(1)}_{z+N-1, z'+N-1}$ with eigenvalues $\gamma_n = n(n+1-z-z'-b-2N)$:
\begin{equation*}
\DD\p_n(\widehat{x}) = \sum_{y\in\Zp}{R(x, y)\p_n(\widehat{y})} = n(n + 1 - z - z'  - b - 2N)\p_n(\widehat{x}), \hspace{.2in}x\in\Zp.
\end{equation*}
\end{prop}

In the next section, we need the following simple statement.

\begin{lem}\label{reversibility}
For any $\lambda, \mu\in\GTp_N$, we have
\begin{equation*}
M_{z, z', a, b|N}(\lambda)R^{(N)}(\lambda, \mu) = M_{z, z', a, b|N}(\mu)R^{(N)}(\mu, \lambda).
\end{equation*}
\end{lem}
\begin{proof}
The statement is evidently true when $\lambda = \mu$ and also when $\sum_i{|\lambda_i - \mu_i|} \geq 2$ (in which case both sides of the desired equality are zero).

Thus, we only need to consider the case in which $\lambda$ and $\mu$ differ in only one coordinate $i$ and $\mu_i = \lambda_i \pm 1$. Without loss of generality, assume $\mu_i = \lambda_i + 1$. In view of the formulas $(\ref{zmeasureN})$ and $(\ref{Rexpression})$, the desired equality becomes
\begin{equation}\label{simpleeq}
W(x)\beta(x) = W(x + 1)\delta(x+1),
\end{equation}
where $x = l_i = \lambda_i + N - i$, $W = W_{z, z', a, b|N}$ and $\beta(x), \delta(x)$ are the birth and death rates in $(\ref{Blabel}), (\ref{Dlabel})$. The last equality is readily checked by using formulas $(\ref{R1expression})$ and $(\ref{Wexpression})$.
\end{proof}
\begin{rem}
\normalfont The equality $(\ref{simpleeq})$ is not a coincidence. The expressions for $\beta(x), \delta(x)$ are the coefficients of the second-order difference equation satisfied by the Wilson-Neretin polynomials. The general theory for orthogonal hypergeometric polynomials of a single variable, see \cite[Ch. 2-3]{NSU}, proves that the polynomial solutions to such a difference equation are orthogonal with respect to a weight satisfying a relation like that in $(\ref{simpleeq})$.
\end{rem}

\subsection{Proof of invariance of the $z$-measures}

\begin{thm}
For any pair $(z, z')\in\mathcal{H}$, the spectral $z$-measure $M_{z, z', a, b|\infty}$ is the unique invariant probability measure with respect to the semigroup $(P_{\infty}(t))_{t\geq 0}$.
\end{thm}
\begin{proof}

\textbf{Step 1.} From the method of intertwiners, see Theorem $\ref{methodintertwiners}$, it suffices to prove for any $N\geq 1$ that the $N$th level $z$-measure $M_{z, z', a, b|N}$ is the unique invariant probability measure with respect to the semigroup $(P_N(t))_{t\geq 0}$. We will instead prove the invariance at the infinitesimal level, i.e., we show
\begin{eqnarray}\label{invrates}
\sum_{\lambda\in\GTp_N}{M_{z, z', a, b|N}(\lambda)R_{z+N-1, z'+N-1}^{(N)}(\lambda, \nu)} = 0, \textrm{ for all } \nu\in\GTp_N, \ N \geq 1.
\end{eqnarray}
The invariance of $M_{z, z', a, b|N}$ with respect to $(P_N(t))_{t\geq 0}$, as well as the uniqueness of the invariant probability measure, is a general fact: an irreducible, regular jump HMC with associated $q$-matrix $A$ has a unique invariant measure if there exists some probability measure $\mu$ on such that $\mu^TA = 0$ (the infinitesimal version of invariance) and in that case $\mu$ is that invariant measure, see \cite[Ch. 8, Thms. 5.1 and 5.3]{B}.\\

\textbf{Step 2.} The desired $(\ref{invrates})$, up to a nonzero constant, is equivalent to
\begin{eqnarray}\label{toproveinv}
&\sum_{l\in\Omega_N}\left(\prod_{i=1}^N{W(l_i)}\right) \Delta_N(\widehat{l}) \times\\
&\left( R(l_1, n_1)\mathbf{1}_{\{l_i = n_i, i\neq 1\}} + \ldots + R(l_N, n_N)\mathbf{1}_{\{l_i = n_i, i\neq N\}}  - c_N\mathbf{1}_{\{l = n\}}\right) = 0,\nonumber
\end{eqnarray}
where $W = W_{z, z', a, b|N}$ and recall $c_N$ was defined in $(\ref{dN})$. The proof will use the Wilson-Neretin polynomials $\mathfrak{p}_n$ with parameters $(\ref{identification})$ introduced previously; in this step, we simply make the observation that
\begin{eqnarray}\label{vandermondeneretin}
\Delta_N(\hatl) = \prod_{1\leq i<j\leq N}{(\hatl_i - \hatl_j)} = \det_{1\leq i, j\leq N}[\p_{N - j}(\hatl_i)],
\end{eqnarray}
which follows from the fact that each $\mathfrak{p}_n(\hatx)$ is a monic polynomial of degree $n$ in $\hatx$.\\

\textbf{Step 3.} We prove $(\ref{toproveinv})$. We begin by claiming that
\begin{equation}\label{lhsdesired}
\prod_{i=1}^N{W(n_i)}(\mathcal{D}^{[1]} + \ldots + \mathcal{D}^{[N]})\det_{i, j}[\p_{N - j}(\hatn_i)] = \prod_{i=1}^N{W(n_i)}(\gamma_0 + \ldots + \gamma_{N-1})\det_{i, j}[\p_{N - j}(\hatn_i)]
\end{equation}
for any $l, n\in\Omega_N$, where $\gamma_j = j(j+1-z-z'-b-2N)$ and $\DD^{[i]}$ is the operator which acts on the polynomial ring $\R[\hatn_1, \ldots, \hatn_N]$ as $\DD$ would act on $\hatn_i = (n_i + \epsilon)^2$, by treating all other variables $\{\hatn_j\}_{j\neq i}$ as constants. The equality above holds because $\DD^{[1]} + \ldots + \DD^{[N]}$ acts on all elements of column $k$ of the matrix $[\p_{N - j}(\hatn_i)]_{i, j=1}^N$, as the eigenvalue $\gamma_{N - k}$, see Proposition $\ref{diff}$.

Due to Lemma $\ref{reversibility}$, for any $1\leq i, j\leq N$, we have
\begin{equation*}
W(n_i)\cdot\mathcal{D}^{[i]}\p_{N - j}(\hatn_i) =
\sum_{l_i\in\Zp}{W(n_i)R(n_i, l_i)\p_{N - j}(\hatl_i)} =
\sum_{l_i\in\Zp}{W(l_i)R(l_i, n_i)\p_{N - j}(\hatl_i)}.
\end{equation*}
Because of $(\ref{vandermondeneretin})$, the identity just given and using again $\det_{i, j}[p_{N - j}(\hatl_i)] = \Delta_N(\hatl)$, the left-hand side of $(\ref{lhsdesired})$ is
\begin{equation}\label{sumzmeasures}
\sum_{l_1, \ldots, l_N\in\Zp}{\left( \prod_{i=1}^N{W(l_i)} \right)\Delta_N(\hatl)\left( R(l_1, n_1)\mathbf{1}_{\{l_i = n_i, i\neq 1\}} + \ldots + R(l_N, n_N)\mathbf{1}_{\{l_i = n_i, i\neq N\}} \right)}.
\end{equation}
The only terms that contribute to the sum above are those $l = (l_1, \ldots, l_N)$ for which $l_i\in\{n_i-1, n_i, n_i+1\}$ for some $i$ and $l_j = n_j$ for the rest of indices $j$. Then we can restrict the sum $(\ref{sumzmeasures})$ to $N$-tuples $l$ for which $l_1 \geq l_2 \geq \ldots \geq l_N$. Observe additionally that if $l_i = l_j$ for some $i\neq j$, then $\Delta_N(\hatl) = 0$. Therefore we can further restrict the sum $(\ref{sumzmeasures})$ to $N$-tuples $l$ for which $l_1 > l_2 > \ldots > l_N$, i.e., the sum is now over $l\in\Omega_N$. It follows that the left-hand side of $(\ref{lhsdesired})$ equals
\begin{equation*}
\sum_{l\in\Omega_N}{\left( \prod_{i=1}^N{W(l_i)} \right)\Delta_N(\hatl)\left( R(l_1, n_1)\mathbf{1}_{\{l_i = n_i, i\neq 1\}} + \ldots + R(l_N, n_N)\mathbf{1}_{\{l_i = n_i, i\neq N\}} \right)}.
\end{equation*}
On the other hand, one can easily check $c_N = \gamma_0 + \ldots + \gamma_{N-1}$, from which the right-hand side of $(\ref{lhsdesired})$ equals $c_N\left(\prod_{i=1}^N{W(n_i)}\right)\det_{1\leq i, j\leq N}[\p_{N-j}(\hatn_i)] = c_N\left(\prod_{i=1}^N{W(n_i)}\right)\Delta(\hatn)$. Hence $(\ref{toproveinv})$ is proved.
\end{proof}

\begin{appendix}

\section{$z$-measures and harmonic analysis on big groups}\label{sec:zmeasuresharmonic}

In this appendix, we explain how the $z$-measures are related to the representation theory of big groups, thus giving more motivation to study the coherent probability measures given by the expressions in $(\ref{zdef1})$. The proofs of all statements below are found in \cite{N2, OO1, Ol3, Ol2, OlOs}.

An \textit{infinite symmetric space} $G/K$ is an inductive limit of Riemannian symmetric spaces $G(n)/K(n)$ of rank $n$, with respect to a natural chain of maps $G(1) \rightarrow G(2) \rightarrow \dots $ that are compatible with the inclusions $K(n) \subset G(n)$. Our main examples are 
\begin{enumerate}
	\item $G/K = U(2\infty)/U(\infty)\times U(\infty) = \lim_{\rightarrow}{(U(2n)/U(n)\times U(n))}$.
	\item $G/K = O(\infty)\times O(\infty)/O(\infty) = \lim_{\rightarrow}{(O(\widetilde{n})\times O(\widetilde{n})/O(\widetilde{n}))}$.
	\item $G/K = Sp(\infty)\times Sp(\infty)/Sp(\infty) = \lim_{\rightarrow}{(Sp(n)\times Sp(n) / Sp(n))}$.
\end{enumerate}
(In (2), $\widetilde{n}$ indicates that the rank of the Riemannian symmetric space in question is $\lfloor \widetilde{n}/2 \rfloor$, so $\widetilde{n}$ can be $2n$ or $2n+1$; either one leads to the same $G/K$, up to isomorphism.)

We assume that $(G, K)$ is any of the three pairs above. A \textit{spherical representation} of $(G, K)$ is a pair $(T, \xi)$, where $T$ is a unitary representation of $G$, on some Hilbert space $H$, and $\xi\in H$ is a cyclic, $K$-invariant unit vector. A \textit{spherical function} of $(G, K)$ is a $K$-biinvariant, positive definite and continuous function $\phi: G \rightarrow \C$ such that $\phi(1_G) = 1$. There is a bijective correspondence between (equivalence classes of) spherical representations and spherical functions, given by
\begin{equation*}
\begin{gathered}
(T, \xi) \leftrightarrow \phi, \hspace{.2in} \phi(g) = \left( T(g)\xi, \xi \right)_H.
\end{gathered}
\end{equation*}

Observe that the set of spherical functions of $(G, K)$ has the natural structure of a convex set $\Upsilon^{G/K}$.
The elements of the set $Ex(\Upsilon^{G/K})$ of extreme points are called extreme spherical functions. The extreme spherical functions of $(G, K)$ are in bijective correspondence with irreducible spherical representations of $(G, K)$. The results from \cite{OO1} show that the set $Ex(\Upsilon^{G/K})$ is in bijective correspondence with the boundary $\Omega_{\infty}$ of the BC branching graph, described in Section $\ref{sec:boundaryintro}$ above. For any $\omega\in\Omega_{\infty}$, we let $\phi^{\omega}$ be the corresponding extreme spherical function of $(G, K)$.

Following the ideas of \cite{Ol2}, one can construct a natural family of spherical representations of $(G, K)$ that should be considered as \textit{generalized quasiregular representations of $G/K$}. For each infinite symmetric space $G/K$ above, the family of representations $\{(T_z, \xi_z)\}_z$ is parametrized by a complex parameter $z\in\C\setminus\Z$ subject to the constraint $\Re z > -(1+b)/2$, where $b = 0, -\frac{1}{2}, \frac{1}{2}$, for the spaces (1), (2) and (3), respectively.

There are actually two constructive definitions of the generalized quasiregular representations $(T_z, \xi_z)$. Working out in detail either of these constructions goes beyond the scope of this paper, and all the tools needed are already in the literature, \cite{N2, Ol2}. Instead let us give a brief summary of the ideas.

\begin{itemize}
	\item The representations $T_z$ of $G/K$ can be constructed as inductive limits of a chain of isometries
\begin{equation*}
\dots \rightarrow H_N = L^2(G(N)/K(N), \mu_N) \rightarrow H_{N+1} = L^2(G(N+1)/K(N+1), \mu_{N+1}) \rightarrow \dots
\end{equation*}
of the quasiregular representations of the finite dimensional Riemannian symmetric spaces $(G(N), K(N))$. We denoted $\mu_N$ to the normalized Haar measure on $G(N)/K(N)$. The corresponding representations $T_z$ live in the Hilbert space $H = \lim_{\rightarrow}{H_N}$ and they depend strongly on the choice of embeddings $H_N \rightarrow H_{N+1}$. For each of our three pairs $(G, K)$, there corresponds a distinguished family of sequences of embeddings $\{H_N \rightarrow H_{N+1}\}_{N \geq 1}$ and the elements of the family are parametrized by one complex parameter $z$ satisfying $\Re z > -(1+b)/2$. The spherical vector $\xi_z$ is the inductive limit $\lim_{\rightarrow}{f'_{z|N}/\|f'_{z|N}\|}$ of normalizations of the vectors $f'_{z|N}\in H_N$ that are described below (see $(\ref{eqn:sphericalvector})$ and the paragraph afterwards). This construction shows why we call $\{T_z\}_z$ the family of generalized quasiregular representations of $G/K$.
	\item There is certain completion $\overline{G/K} \supset G/K$ of the space $G/K$, which is not a group but has a natural $G$-action. It can be constructed by following the ideas of Neretin, \cite{N2}.

The space $\overline{G/K}$ admits an analogue of the Haar measure $\mu$. In fact, there exists a one-parameter family of probability measures $\{\mu^{(z)}\}_{z\in\C :\ \Re z > -(1+b)/2}$ that are quasi-invariant, with respect to the action of $G$.
The representations $T_z$ can be realized in the Hilbert spaces $L^2(\overline{G/K}, \mu^{(z)})$. The corresponding spherical vectors $\xi_z$ are all equal to the constant function $\mathbf{1}\in L^2(\overline{G/K}, \mu^{(z)})$ that is identically equal to one.
\end{itemize}

Let us describe the sequence of vectors $\{f'_{z|N}\in H_N\}_{N \geq 1}$ that give rise to the spherical vector $\xi_z = \lim_{\rightarrow}{f'_{z|N}}$, as in the first item above.
For each of the series in (1), (2), (3) above, there is an isomorphism between the subspace of $K(N)$-invariant vectors of $L^2(G(N)/K(N), \mu_N)$ and the Hilbert space $H_N^{a, b}$ of symmetric functions on $[-1, 1]^N$ that are square integrable with respect to the measure $\mathfrak{m}^{a, b}_N$ whose density (with respect to the Lebesgue measure on $[-1, 1]^N$) is
\begin{equation*}
\prod_{1\leq i < j\leq N}{(x_i - x_j)^2}\cdot \prod_{k=1}^N{(1 - x_k)^a(1 + x_k)^b},
\end{equation*}
and we use the following pairs of real parameters
\begin{enumerate}
	\item $(a, b) = (0, 0)$.
	\item $(a, b) = (\pm\frac{1}{2}, -\frac{1}{2})$.
	\item $(a, b) = (\frac{1}{2}, \frac{1}{2})$.
\end{enumerate}
(In (2), $a = -\frac{1}{2}$ if $\widetilde{n} = 2n$ and $a = \frac{1}{2}$ if $\widetilde{n} = 2n+1$.)

For any $z\in\C$, $\Re z > -(1+b)/2$, and any $N \geq 1$, the function
\begin{equation}\label{eqn:sphericalvector}
f_{z|N}(x_1, \dots, x_N) = \prod_{i=1}^N{(1 + x_i)^z}.
\end{equation}
is an element of the Hilbert space $H_N^{a, b}$. Let $f'_{z|N}$ be the corresponding $K(N)$-invariant vector in $H_N$, under the isomorphism alluded to before. The sequence $\{f'_{z|N}\}_{N \geq 1}$ is such that the isometric embedding $H_N \rightarrow H_{N+1}$ that gives rise to the representation $T_z$ maps $f'_{z|N}$ to $f'_{z|N+1}$, and therefore it maps $f'_{z|N}/\|f'_{z|N}\|$ to $f'_{z|N+1}/\|f'_{z|N+1}\|$. Thus there exists an inductive limit $\xi_z = \lim_{\rightarrow}{f'_{z|N}/\|f'_{z|N}\|}$ of length $1$, which is the desired spherical vector of $T_z$.

Next, let us explain how the $z$-measures $M_{z, z', a, b|\infty}$ and corresponding sequences of probability measures $\{M_{z, z', a, b|N}\}_{N \geq 1}$ on the spaces $\{\GTp_N\}_{N \geq 1}$, given in $(\ref{zdef1})$, arise from the generalized quasiregular representations $\{T_z\}_z$.

Instead of using the language of spherical representations, we use the more convenient language of spherical functions. Let $\{\phi_z\}_z$ be the family of spherical functions that correspond to the generalized quasiregular representations $\{(T_z, \xi_z)\}_z$. A general statement, due to Olshanski, \cite[Thm 9.2]{Ol2}, shows that the spherical function $\phi_z$ is a (continual) convex combination of the extreme spherical functions $\phi^{\omega}$, $\omega\in\Omega_{\infty}$. This means that there exists a unique probability measure $M_z$ on $\Omega_{\infty}$ such that
\begin{equation*}
\phi_z = \int_{\Omega_{\infty}}{\phi^{\omega}M_z(d\omega)}.
\end{equation*}

The probability measure $M_z$ is the $z$-measure associated to the generalized quasiregular representation $(T_z, \xi_z)$.
A $z$-measure corresponds to a ``coherent'' sequence of probability measures on the discrete spaces $\GTp_N$, by virtue of the bijection $(\ref{tobeiso})$.

We can describe explicitly the corresponding sequence of probability measures on the spaces $\GTp_N$, $N \geq 1$.
Since $G = \lim_{\rightarrow}{G(N)}$, the spherical function $\phi_z : G \rightarrow \C$ is fully determined by its sequence of restrictions $\{\left.\phi_z\right|_{G(N)}\}_{N \geq 1}$. Each $\left.\phi_z\right|_{G(N)}$ is an element of the subspace of $K(N)$-invariant vectors of the Hilbert space $H_ N = L^2(G(N)/K(N), \mu_N)$; moreover $\left.\phi_z\right|_{G(N)}(1_{G(N)}) = 1$ and $\left.\phi_z\right|_{G(N)}$ is a positive definite function.
Recall that an orthogonal basis of the subspace of $K(N)$-invariant vectors of $H_N$ is $\{\chi_{\lambda} : \lambda\in\GTp_N\}$, where $\GTp_N$ parametrizes the irreducible rational representations of $G(N)$ with a $K(N)$-invariant vector and each $\chi_{\lambda}$ is a normalized spherical function of $(G(N), K(N))$. For the pairs (2), (3) above, the spherical functions $\chi_{\lambda}$ coincide with normalized irreducible characters of the orthogonal and symplectic groups, respectively.
Then there is a canonical expansion
\begin{equation}\label{zmeasures:info1}
\left.\phi_z\right|_{G(N)} = \sum_{\lambda\in\GTp_N}{M_{z|N}(\lambda)\cdot\chi_{\lambda}}, \hspace{.1in} M_{z|N}(\lambda)\in\R.
\end{equation}

Because the function $\left.\phi_z\right|_{G(N)}$ is positive definite and normalized at the unit element of $G(N)$, we have $M_{z|N}(\lambda) \geq 0$ and $\sum_{\lambda\in\GTp_N}{M_{z|N}(\lambda)} = 1$. Then $M_{z|N}$ is a probability measure on $\GTp_N$. The sequence $\{M_{z|N}\}_{N \geq 1}$, in fact, is the one that corresponds to $M_z$ under the bijection $(\ref{tobeiso})$.

An explicit formula for $M_{z|N}(\lambda)$, $\lambda\in\GTp_N$, can be calculated explicitly. Under the isomorphism between the subspace of $K(N)$-invariant vectors of $H_N$ and $H_N^{a, b}$, the image of the basis $\{\chi_{\lambda}\}_{\lambda}$ is given by the (orthogonal) basis $\{\widetilde{\mathfrak{P}}_{\lambda}(x_1, \dots, x_N | a, b)\}_{\lambda} \subset H_N^{a, b}$ of normalized and multivariate Jacobi polynomials
\begin{equation*}
\widetilde{\mathfrak{P}}_{\lambda}(x_1, \dots, x_N | a, b) = \frac{\mathfrak{P}_{\lambda}(x_1, \dots, x_N | a, b)}{\mathfrak{P}_{\lambda}(1^N | a, b)},
\end{equation*}
see Section $\ref{jacobi}$ below. By following the ideas in \cite[Sec. 6]{Ol2}, the coefficients $M_{z|N}(\lambda)$ can be calculated as follows, cf. \cite[Lemma 6.4]{Ol2},
\begin{equation}\label{eqn:coeffsz}
M_{z|N}(\lambda) = \frac{\left|(f_{z|N}, \widetilde{\mathfrak{P}}_{\lambda})_H\right|^2}{\left\| f_{z|N} \right\|_H^2 \left\| \widetilde{\mathfrak{P}}_{\lambda} \right\|_H^2}, \hspace{.2in} H = H_N^{a, b}.
\end{equation}
The calculation of all three terms in $(\ref{eqn:coeffsz})$ is carried out in \cite{OlOs}, for general parameters $a, b > -1$ and $z\in\C$, $\Re z > -(1+b)/2$, and the answer is given by the formula $(\ref{zdef1})$ for $z' = \overline{z}$.

The expression $(\ref{zdef1})$, with representation-theoretic origin for $z' = \overline{z}$, and the pairs $(a, b)\in\{(0, 0), (1/2, 1/2), (\pm 1/2, -1/2)\}$, suggests to consider an analytic continuation replacing $\overline{z}$ by $z'$ and as general real parameters $a, b$ as possible.
Our choice of $a \geq b \geq -1/2$ comes mainly because we need the main result of \cite{OO1}. The conditions on the pair $(z, z')$, that we give in Definition $\ref{Hdef}$, are sufficient for all expressions of the form $(\ref{zdef1})$ to be nonnegative (and in fact, strictly positive).
The result of this paper and others, \cite{BO4, C, OlOs}, suggests that the $z$-measures with four parameters $z, z', a, b$ is the most general object that can be analyzed thoroughly, despite the fact that most quadruples $(z, z', a, b)$ do not have any representation theoretic origin. The recent paper \cite{OlOs2} indicates that the BC type $z$-measures admit further a natural one-parameter $\theta > 0$ degeneration (all papers cited before treat BC $z$-measures with $\theta = 1$) and it would be interesting to study them.

\section{Binomial Formula and Coherence Property for Jacobi Polynomials}\label{appendixA}

In this appendix, we prove Proposition $\ref{coherence}$. We prove it in the equivalent form of Proposition A.3. First, we show a ``binomial formula'' for Jacobi polynomials. The proof below was extracted from the arguments of \cite[Proof of Thm. 1.2]{OO2}.

For any partition $\mu$ of length $\ell(\mu) = K\leq N$, let $t_{\mu|N}^*$ be the function on $\GTp_N$ defined by

\begin{equation*}
t_{\mu|N}^*(\lambda) := \frac{\det_{1\leq i, j\leq N}[(\hatl_i| \boldsymbol{\epsilon})^{\mu_j + N - j}]}{\prod_{1\leq i < j\leq N}{(\hatl_i - \hatl_j)}}, \hspace{.2in}\lambda\in\GTp_N,
\end{equation*}
where we set $\mu_{K+1} = \ldots = \mu_N = 0$.

\begin{prop}\label{binomialformula}
For any $\lambda\in\GTp_N$, we have
\begin{eqnarray}
\Phi_{\lambda}(1 + x_1, \ldots, 1 + x_N | a, b)  = \sum_{\mu\in\GTp_N}{\frac{t_{\mu|N}^*(\lambda)s_{\mu}(x_1, \ldots, x_N)}{2^{|\mu|}c(N, \mu)}},
\end{eqnarray}
where $s_{\mu}$ is the Schur function parametrized by $\mu$ and
\begin{equation}\label{cNmudef}
c(N, \mu) \myeq \prod_{i = 1}^N{(N - i + 1)_{\mu_i}(N + a - i + 1)_{\mu_i}}.
\end{equation}
\end{prop}
\begin{rem}
The sum in the proposition above is finite because one can easily show $t_{\mu|N}^*(\lambda) = 0$ unless $|\mu| \leq |\lambda|$.
\end{rem}
\begin{proof}
We recall a well-known identity: if $\displaystyle f_i(x) = \sum_{m = 0}^{\infty}{a_m^{(i)}x^m}$, for all $1\leq i\leq N$, then
\begin{eqnarray}\label{binomialstuff}
\frac{\det_{1\leq i,j\leq N}[f_i(x_j)]}{\prod_{1\leq i < j \leq N}{(x_i - x_j)}} = \sum_{\mu\in\GTp_N}{\det_{1\leq i, j\leq N}[a^{(i)}_{\mu_j + N - j}]s_{\mu}(x_1, \ldots, x_N)},
\end{eqnarray}
where $s_{\mu}(x_1, \ldots, x_N)$ is the Schur polynomial associated to partition $\mu$. We shall apply this result to the normalized Jacobi polynomials
\begin{equation}\label{fnfi}
\Phi_{\lambda_i + N - i}(1 + x | a, b) = \frac{\mathfrak{P}_{\lambda_i + N - i}(1 + x | a, b)}{\mathfrak{P}_{\lambda_i + N - i}(1 | a, b)}
\end{equation}
for which the following expansion is well-known, see e.g. \cite[(4.21.2)]{S},
\begin{equation}\label{P1t}
\Phi_k(1 + x | a, b) = \sum_{m = 0}^{\infty}{\frac{k(k - 1)\cdots(k - m + 1)(k + a + b + 1)\cdots(k + a + b + m)}{2^m m! (a + 1)_m}x^m}
\end{equation}
From $(\ref{binomialstuff})$ and the definition of $\mathfrak{P}_{\lambda}(\cdot | a, b)$, the left-hand side of $(\ref{binomialstuff})$ is $\Phi_{\lambda}(1 + x_1, \ldots, 1 + x_N | a, b)\frac{\mathfrak{P}_{\lambda}(1^N | a, b)}{\prod_{i=1}^N{\mathfrak{P}_{\lambda_i + N - i}(1 | a, b)}}$. By virtue of $(\ref{evalones})$, the latter expression becomes
\begin{eqnarray*}
2^{-\frac{N(N-1)}{2}}\Phi_{\lambda}(1 + x_1, \ldots, 1 + x_N | a, b)\prod_{i=1}^N{\frac{1}{(i-1)!(a + 1)_{i-1}}}\prod_{i<j}{(\hatl_i - \hatl_j)}.
\end{eqnarray*}
In view of $(\ref{P1t})$, the right side of $(\ref{binomialstuff})$ is $\sum_{\mu\in\GTp_N}{\det_{i, j}[a^{(i)}_{\mu_j + N - j}]s_{\mu}(x_1, \ldots, x_N)}$, where
\begin{equation*}
a^{(i)}_m = \frac{l_i(l_i-1)\cdots(l_i - m + 1)(l_i + a + b + 1)\cdots(l_i + a + b + m)}{2^m m! (a + 1)_m}.
\end{equation*}
If we use $(l_i - s)(l_i + a + b + 1 + s) = (l_i + \epsilon)^2 - (\epsilon + s)^2$, $s = 0, 1, \ldots, m-1$, then
\begin{eqnarray*}
\det_{i, j}[a^{(i)}_{\mu_j + N - j}] &=& \frac{\det_{i, j}[( \hatl_i | \boldsymbol{\epsilon})^{\mu_j + N - j}]_{i, j = 1}^N}{\prod_{i = 1}^N{2^{\mu_i + N - i}(\mu_i + N - i)!(a + 1)_{\mu_i + N - i}}}\\
&=& \frac{t_{\mu|N}^*(\lambda)\prod_{i < j}{(\hatl_i - \hatl_j)}}{\prod_{i = 1}^N{2^{\mu_i + N - i}(\mu_i + N - i)!(a + 1)_{\mu_i + N - i}}}.
\end{eqnarray*}
The result follows readily.
\end{proof}
\begin{prop}\label{coherencet}
Let $N\in\N$ be arbitrary and $\mu\in\GTp_N$, $\lambda\in\GTp_{N+1}$. Then
\begin{equation}
\frac{t_{\mu|N+1}^*(\lambda)}{c(N + 1, \mu)} = \sum_{\nu\in\GTp_N}{\Lambda^{N+1}_N(\lambda, \nu)\frac{t_{\mu|N}^*(\nu)}{c(N, \mu)}}.
\end{equation}
\end{prop}
\begin{proof}
From the definition of the kernels $\Lambda^{N+1}_N$, we have
\begin{eqnarray*}
\Phi_{\lambda}(1 + x_1, \ldots, 1 + x_N, 1 | a, b) = \sum_{\nu\in\GTp_N}{\Lambda^{N + 1}_N(\lambda, \nu)\Phi_{\nu}(1 + x_1, \ldots, 1 + x_N | a, b)}.
\end{eqnarray*}
Applying Proposition $\ref{binomialformula}$ twice, and interchanging the sums (it is allowed because there are only finitely many nonzero terms in total), we obtain
\begin{equation}\label{luckyeqn}
\sum_{\mu\in\GTp_{N+1}}{\frac{t_{\mu|N+1}^*(\lambda)s_{\mu}(x_1, \ldots, x_N, 0)}{2^{|\mu|}c(N+1, \mu)}} = \sum_{\mu\in\GTp_N}{s_{\mu}(x_1, \ldots, x_N)\sum_{\nu\in\GTp_N}{\Lambda^{N+1}_N(\lambda, \nu)\frac{t_{\mu|N}^*(\nu)}{2^{|\mu|}c(N, \mu)}}}
\end{equation}
By the stability of the Schur functions, see \cite[Section I]{M}, for any $\mu\in\GTp_{N+1}$, we have
\begin{equation*}
s_{\mu}(x_1, \ldots, x_N, 0) =
    \begin{cases}
        s_{\mu}(x_1, \ldots, x_N) & \textrm{if } \mu_{N+1} = 0\\
        0 & \textrm{if } \mu_{N+1} > 0.
    \end{cases}
\end{equation*}
Because the Schur functions $s_{\mu}(x_1, \ldots, x_N)$, $\mu\in\GTp_N$, are linearly independent in the space of polynomials on $x_1, \ldots, x_N$, the coefficient of $s_{\mu}(x_1, \ldots, x_N)$ in the left and right-hand sides of $(\ref{luckyeqn})$ must be equal. The result follows.
\end{proof}

\begin{rem}
Proposition $\ref{coherence}$ follows from Proposition $\ref{coherencet}$ and the equality
\begin{equation*}
T_{\mu|N}^*(\hatl_1, \ldots, \hatl_N) = \frac{t_{\mu|N}^*(\lambda)}{c(N, \mu)},
\end{equation*}
for any $\lambda\in\GTp_N$ and partitions $\mu$ with $\ell(\mu) \leq N$. We have used $t_{\mu|N}^*(\lambda)$ is this appendix because for $(a, b) = (\frac{1}{2}, \frac{1}{2}), (\frac{1}{2}, -\frac{1}{2})$ and $(-\frac{1}{2}, -\frac{1}{2})$, the polynomials $t_{\mu|N}^*$ arise in the representation theory of the rank $N$ Lie algebras $\mathfrak{g}(N)$ of type $B, C$ and $D$, respectively. To see how, recall the Harish-Chandra isomorphism
\begin{equation*}
    Z(\mathfrak{g}(N)) \longrightarrow M^*(N)
\end{equation*}
between the center $Z(\mathfrak{g}(N))$ of the enveloping algebra $U(\mathfrak{g}(N))$ and the subalgebra $M^*(N)\subset\C[x_1, \ldots, x_N]$ of polynomials $p(x_1, \ldots, x_N)$ which can be written in the form
\begin{equation*}
p(x_1, \ldots, x_N) = q((x_1 + N - 1 + \epsilon)^2, \ldots, (x_N + \epsilon)^2),
\end{equation*}
for some symmetric polynomial $q$ in $N$ variables ($M^*(N)$ is called sometimes an algebra of ``shifted'' symmetric polynomials). There are distinguished bases $\{\mathbb{T}_{\mu|N}\}_{\mu}$ of the centers $Z(\mathfrak{g}(N))$, see \cite{OO2}, the images of which are the polynomials $t_{\mu|N}^*$.

Theorem 3.1. from \cite{OO2} was important to us. Its content is the existence of natural \textit{averaging operators}
\begin{equation*}
    Z(\mathfrak{g}(N)) \longrightarrow Z(\mathfrak{g}(N+1))
\end{equation*}
mapping $\mathbb{T}_{\mu|N}$ to a constant factor of $\mathbb{T}_{\mu|N+1}$. Under the Harish-Chandra isomorphism, this statement can be interpreted as a \textit{coherence formula} for the polynomials $t_{\mu|N}^*$, at least for the three special pairs $(a, b)$. Proposition $\ref{coherencet}$ above is a generalization for all $a, b > -1$.
\end{rem}

\section{End of Proof of Proposition $\ref{semigroupN}$}\label{appendixB}

Fix $N\in\N$. In this appendix, we prove
\begin{equation}\label{toprove3}
\sum_{y_1, \ldots, y_N\in E}{\left( \prod_{i=1}^N{P_1(t; x_i, y_i)} \right) \Delta_N(\mathbf{y})} = e^{c_Nt}\Delta_N(\mathbf{x})
\end{equation}
under Assumption $\ref{operatorassumption}$. We write the two items of the assumptions more concretely; the first item is equivalent to the existence of real numbers $\{a_{i, j}\}_{i\geq j\geq 0}$ such that
\begin{equation}\label{firstassumption}
\sum_{y\in E}{q_{x, y}y^m} = a_{m, m}x^m + a_{m, m-1}x^{m-1} + \ldots + a_{m, 0}, \hspace{.2in}m\in\Zp.
\end{equation}
The second item of Assumption $\ref{operatorassumption}$ is the identity
\begin{equation}\label{secondassumption}
\sum_{i=1}^N{\left(\sum_{y\in E}{q_{x_i, y}\Delta_N(x_1, \ldots, x_{i-1}, y, x_{i+1}, \ldots, x_N)}\right)} = c_N\Delta_N(x_1, \ldots, x_N).
\end{equation}
Let $\mathcal{X}_N$ be the space of polynomial functions on $N$ variables $x_1, x_2, \ldots, x_N$ with real coefficients and such that the degree of each variable $x_i$ is at most $N-1$. In particular $\Delta_N\in\mathcal{X}_N$. We use the notation
\begin{equation*}
[N-1] \myeq \{0, 1, 2, \ldots, N-1\}.
\end{equation*}
Then $\mathcal{X}_N$ has the linear basis $\{x_1^{k_1}x_2^{k_2}\cdots x_N^{k_N}\}_{k_1, k_2, \ldots, k_N\in [N-1]}$ and $\dim\mathcal{X}_N = N^N<\infty$.

\subsection{Operator $\mathbf{Q}: \mathcal{X}_N \longrightarrow \mathcal{X}_N$} Define the operator $\mathbf{Q}$ by
\begin{equation*}
(\mathbf{Q}f)(x_1, \ldots, x_N) \myeq \sum_{i=1}^N{\left(\sum_{y\in E}{q_{x_i, y}f(x_1, \ldots, x_{i-1}, y, x_{i+1}, \ldots, x_N)}\right)},
\end{equation*}
where $f$ is any polynomial function. By assumption $\ref{firstassumption}$, $\mathbf{Q}$ is well-defined. The operator $\mathbf{Q}$ is given by the matrix $M$ of format $E^N\times E^N$ with entries
\begin{equation*}
M(\mathbf{x}, \mathbf{y}) = \mathbf{1}_{\{x_j = y_j, j\neq 1\}}q_{x_1, y_1} + \ldots + \mathbf{1}_{\{x_j = y_j, j\neq N\}}q_{x_N, y_N}, \hspace{.2in}\boldx, \boldy\in E^N.
\end{equation*}
Then the operator $\mathbf{Q}$ can be expressed in terms of $M$ by
\begin{equation*}
(\mathbf{Q}f)(\mathbf{x}) = \sum_{\mathbf{y}\in E^N}{M(\mathbf{x}, \mathbf{y})f(\mathbf{y})}.
\end{equation*}

\subsection{Operators $\mathbf{P}(t): \mathcal{X}_N \longrightarrow \mathcal{X}_N$, $t\geq 0$} Let $f$ be a polynomial function. The action of $\mathbf{P}(t)$ is given by
\begin{equation*}
(\mathbf{P}(t)f)(\mathbf{x}) \myeq \sum_{y_1, \ldots, y_N \in E}{\prod_{i=1}^N{P_1(t; x_i, y_i) f(\mathbf{y})}}.
\end{equation*}
The operator $\mathbf{P}(t)$ can be given by a matrix $M(t)$ of format $E^N\times E^N$ given by
\begin{equation*}
M(t)(\mathbf{x}, \mathbf{y}) = \prod_{i=1}^N{P_1(t; x_i, y_i)}.
\end{equation*}
The operators $\mathbf{P}(t)$, $t\geq 0$, can also be defined first for polynomial functions of the form $f(y_1, y_2, \ldots, y_N) = y_1^{k_1}y_2^{k_2}\cdots y_N^{k_N}$, $k_i\in [N-1]$, by
\begin{equation}\label{defPt}
(\mathbf{P}(t)f)(\mathbf{x}) = \prod_{i=1}^N{\left( \sum_{y\in E}{P_1(t; x_i, y)y^{k_i}} \right)},
\end{equation}
and then extended to $\mathcal{X}_N$ by linearity. We need to show that $(\ref{defPt})$ is a good definition for an operator on $\mathcal{X}_N$: we prove (a) the sum $\sum_{y\in E}{P_1(t; x, y)y^k}$ is convergent for any $x\in\Zp$, $k\in\Zp$, and (b) $F_k(x) = \sum_{y\in E}{P_1(t; x, y)y^k}$ is a polynomial in $x$ of degree $\leq k$.

For (a), fix $t\geq 0$ and $k\in\Zp$. From $(\ref{firstassumption})$, there exists a constant $c>0$ such that
\begin{equation*}
\sum_{y\in E}{Q(x, y)y^k} < c(1 + x^k),
\end{equation*}
for some constant $c>0$ and all $x\in\Zp$. We claim
\begin{equation}\label{claimappendix}
\sum_{y\in E}{P_1(t; x, y)y^k} \leq e^{t(c + q_x)}x^k + e^{t(c+q_x)}, \hspace{.2in}x\in\Zp,
\end{equation}
which would prove (a). Since $Q$ is non-explosive, $\displaystyle P_1(t; x, y) = \sum_{n=0}^{\infty}{P_1^{[n]}}(t; x, y)$, where
\begin{equation}\label{recurrenceappendix}
P_1^{[n]}(t; x, y) =
    \begin{cases}
        e^{-q_xt}\cdot\mathbf{1}_{\{x = y\}} & \textrm{if } n=0,\\
        \int_0^t{e^{-q_xs}\cdot\sum_{z\neq x}{q_{x, z}P_1^{[n-1]}(t-s; z, y)}ds} & \textrm{if } n\geq 1.
    \end{cases}
\end{equation}
The claim $(\ref{claimappendix})$ follows from the stronger claim
\begin{equation*}
\sum_{y\in E}{P_1^{[n]}(t; x, y)y^k} \leq \frac{(c+q_x)^nt^n}{n!}x^k + \frac{(c+q_x)^nt^n}{n!}, \hspace{.2in}n\geq 0,
\end{equation*}
and the latter claim can be easily proved by induction on $n$, the recurrence relation $(\ref{recurrenceappendix})$ and the inequality $e^{-tq_x} \leq 1$, $t\geq 0$.

We now prove (b), i.e., we show that $F_k(t, x) = \sum_{y\in E}{P_1(t; x, y)y^k}$ is a polynomial of degree $\leq k$ in $x$. From Kolmogorov's forward equation $P_1'(t) = P_1(t)Q$ and $(\ref{firstassumption})$, we have
\begin{gather*}
\frac{\partial}{\partial t}F_k(t; x) = \sum_{y\in E}{P_1'(t; x, y)y^k} = \sum_{y\in E}{\left( \sum_{z\in E}{P_1(t; x, z)Q(z, y)} \right)y^k}\\
= \sum_{z\in E}{P_1(t; x, z)\left( \sum_{y\in E}{Q(z, y)y^k} \right)} = \sum_{z\in E}{P_1(t; x, z)\left( a_{k, k}z^k + \ldots + a_{k, 0} \right)}\\
= a_{k, k}F_k(t; x) + a_{k, k-1}F_{k-1}(t; x) + \ldots + a_{k, 0}F_0(t; x).
\end{gather*}
Since $P_1(0) = 1$, we also have $F_k(0, x) = x^k$. One can conclude easily, by induction on $k$, and by the differential equation derived above, that $F_k(t; x)$ is indeed a polynomial of degree $\leq k$ on $x$. For instance, the ODE $\partial_t F_1(t; x) = a_{1, 1}F_1(t; x) + a_{1, 0}F_0(t; x)$ with initial condition $F_1(0, x) = x$ is solved by $F_1(t; x) = e^{ta_{11}}\left( x + \int_0^t{a_{1, 0}F_0(t; x)dt} \right)$, and since $F_0(t; x)$ is a constant independent of $x$, then $F_1(t; x)$ is a linear polynomial on $x$.

\subsection{End of argument}

Recall $\mathbf{Q}$ is realized by the $E^N\times E^N$ matrix $M$ with entries
\begin{equation*}
M(\mathbf{x}, \mathbf{y}) = \mathbf{1}_{\{x_j = y_j, j\neq 1\}}q_{x_1, y_1} + \ldots + \mathbf{1}_{\{x_j = y_j, j\neq N\}}q_{x_N, y_N},
\end{equation*}
while $\mathbf{P}(t)$, $t\geq 0$, are realized by the $E^N\times E^N$ matrices $M(t)$, $t\geq 0$, with entries
\begin{equation*}
M(t)(\boldx, \boldy) = \prod_{i=1}^N{P_1(t; x_i, y_i)}.
\end{equation*}
From Kolmogorov's backward equation $P_1'(t) = QP_1(t)$, it follows that
\begin{eqnarray*}
\frac{d}{dt}M(t)(\boldx, \boldy) &=& \sum_{i=1}^N{\prod_{\substack{1\leq j\leq N\\ j\neq i}}{P_1(t; x_j, y_j)} P_1'(t; x_i, y_i)}\\
&=& \sum_{i=1}^N{\prod_{\substack{1\leq j\leq N\\ j\neq i}}{P_1(t; x_j, y_j)} \sum_{y\in E}{q_{x_i, y}P_1(t; y, y_i)}}\\
&=& \sum_{\mathbf{z}\in E^N}{(\mathbf{1}_{\{x_j = z_j, j\neq 1\}}q_{x_1, z_1} + \ldots + \mathbf{1}_{\{x_j = z_j, j\neq N\}}q_{x_N, z_N})\cdot\prod_{i=1}^N{P_1(t; z_i, y_i)}}\\
&=& \sum_{\mathbf{z}\in E^N}{M(\boldx, \mathbf{z})M(t)(\mathbf{z}, \boldy)}.
\end{eqnarray*}
In matrix notation, we have proved $M'(t) = MM(t)$. By the initial condition $P_1(0) = 1$, we have also $M(0) = 1$. In the language of linear operators on $\mathcal{X}_N$, we have proved that $\mathbf{P}(t)$, $t\geq 0$, solves the Cauchy problem
\begin{eqnarray*}
\mathbf{P}'(t) &=& \mathbf{Q}\mathbf{P}(t),\\
\mathbf{P}(0) &=& 1.
\end{eqnarray*}
Since $\mathcal{X}_N$ is finite dimensional, we conclude that $\mathbf{P}(t) = e^{t\mathbf{Q}}$. The second assumption $\ref{secondassumption}$ gives $\mathbf{Q}\Delta_N = c_N\Delta_N$. Then $\mathbf{P}(t)\Delta_N = e^{c_Nt}\Delta_N$, for all $t\geq 0$, which is precisely the identity $(\ref{toprove3})$ we wanted to prove.
\end{appendix}

\end{document}